\documentclass[11pt]{amsart}
\usepackage{amssymb,verbatim,graphicx}
\usepackage{graphics,epsfig,psfrag} 

\newtheorem{theorem}{Theorem}[section]

\newtheorem{corollary}[theorem]{Corollary}

\newtheorem{conjecture}[theorem]{Conjecture}

\newtheorem{proposition}[theorem]{Proposition}
\newtheorem{lemma}[theorem]{Lemma}
\newtheorem{lem}[theorem]{}
\theoremstyle{definition}
\newtheorem{definition}[theorem]{Definition}
\theoremstyle{remark}
\newtheorem{remark}[theorem]{Remark}
\newtheorem{example}[theorem]{Example}
\newcommand{\blem}{\begin{lem} \rm}
\newcommand{\elem}{\end{lem}}

\newcommand\A{\mathcal{A}}
\newcommand\M{\mathcal{M}}

\newcommand{\W}{\mathcal{W}}

\newcommand{\T}{\mathcal{T}}
\newcommand{\J}{\mathcal{J}}

\newcommand{\R}{\mathbb{R}}
\renewcommand{\H}{\mathbb{H}}

\newcommand{\C}{\mathbb{C}}

\newcommand{\Z}{\mathbb{Z}}

\newcommand{\ddt}{\frac{d}{dt}}

\renewcommand{\P}{\mathbb{P}}
\newcommand{\PP}{\mathcal{P}}

\newcommand\lie[1]{\mathfrak{#1}}

\newcommand{\h}{\lie{h}}
\newcommand{\g}{\lie{g}}

\renewcommand{\t}{\lie{t}}

\newcommand{\on}{\operatorname}
\newcommand{\Si}{\Sigma}

\newcommand{\ainfty}{{$A_\infty$\ }}

\newcommand{\Def}{\on{Def}}

\newcommand{\dist}{\on{dist}}

\newcommand{\pre}{{\on{pre}}}

\newcommand{\fr}{{\on{fr}}}

\newcommand{\dual}{\vee}

\newcommand{\Lag}{\on{Lag}}

\newcommand{\loc}{{\on{loc}}}

\newcommand{\End}{\on{End}}

\newcommand{\x}{\times}
\newcommand{\de}{\delta}
\newcommand{\al}{\alpha}

\newcommand{\Aut}{ \on{Aut} } 
 
\newcommand{\aut}{ \on{aut} }

\newcommand{\Hol}{ \on{Hol} }

\newcommand{\Hom}{ \on{Hom}}
\newcommand{\Ind}{ \on{Ind}}
\renewcommand{\ker}{ \on{ker}}

\newcommand{\Vol}{  \on{Vol}}
\newcommand{\diag}{  \on{diag}}

\newcommand\dirac{/\kern-1.2ex\partial} 
\newcommand\qu{/\kern-.7ex/} 
\newcommand\lqu{\backslash \kern-.7ex \backslash} 

\newcommand\dr{r_+ \kern-.7ex - \kern-.7ex r_-}
 
\newcommand{\labell}\label

\renewcommand{\d}{{\on{d}}}
\newcommand{\ol}{\overline}
\newcommand{\olp}{\ol{\partial}}

\newcommand\Phinv{\Phi^{-1}}

\newcommand\eps{\epsilon}

\newcommand\om{\omega}

\newcommand{\f}{\frac}
\newcommand{\lan}{\langle}
\newcommand{\ran}{\rangle}
\newcommand{\hh}{{\f{1}{2}}}

\newcommand{\ti}{\tilde}

\newcommand\cE{\mathcal{E}}

\newcommand\cF{\mathcal{F}}
\newcommand\cB{\mathcal{B}}
\newcommand\cA{\mathcal{A}}
\newcommand\mdeg{\gamma}
\newcommand\univ{\on{univ}}
\newcommand\cT{\mathcal{T}}
\newcommand\cI{\mathcal{I}}
\newcommand\cH{\mathcal{H}}

\renewcommand{\ss}{{\on{ss}}}

\newcommand\Map{\on{Map}}
\newcommand\rank{\on{rank}}
\newcommand\ev{\on{ev}}

\newcommand\Vect{\on{Vect}}

\newcommand\G{\mathcal{G}}

\newcommand\grad{\on{grad}}
\newcommand\reg{{\on{reg}}}

\newcommand\bra[1]{ < \kern-.7ex {#1} \kern-.7ex >} 
\newcommand\bdefn{\begin{definition}}
\newcommand\edefn{\end{definition}}
\newcommand\bea{\begin{eqnarray*}}
\newcommand\eea{\end{eqnarray*}}
\newcommand\bcv{\left[ \begin{array}{r} }
\newcommand\ecv{\end{array} \right] }

\newcommand\bma{\left[ \begin{array}{l} }
\newcommand\ema{\end{array} \right]}
\newcommand\ben{\begin{enumerate}}
\newcommand\een{\end{enumerate}}
\newcommand\beq{\begin{equation}}
\newcommand\eeq{\end{equation}}
\newcommand\bex{\begin{example}}
\newcommand\bsj{\left\{ \begin{array}{rrr} }
\newcommand\esj{\end{array} \right\}}

\newcommand\eex{\end{example}}
\newcommand\crit{{\on{crit}}}

\newcommand\sx{*\kern-.5ex_X}

\def\mathunderaccent#1{\let\theaccent#1\mathpalette\putaccentunder}
\def\putaccentunder#1#2{\oalign{$#1#2$\crcr\hidewidth \vbox
to.2ex{\hbox{$#1\theaccent{}$}\vss}\hidewidth}}

\begin{document}

\title{Gauged Floer theory of toric moment fibers}

\author{Christopher T. Woodward}


\address{Mathematics-Hill Center,
Rutgers University, 110 Frelinghuysen Road, Piscataway, NJ 08854-8019,
U.S.A.}  \email{ctw@math.rutgers.edu}

\thanks{Partially supported by NSF
 grant DMS0904358}

 \thanks{Corrected from the published version in {\em Geometric and
     Functional Analysis} 21 (2011) 680--749, by permission of the
   editors.}

\begin{abstract}  
We investigate the small area limit of the gauged Lagrangian Floer
cohomology of Frauenfelder \cite{frau:thesis}.  The resulting
cohomology theory, which we call {\em quasimap Floer cohomology}, is
an obstruction to displaceability of Lagrangians in the symplectic
quotient.  We use the theory to reproduce the results of
Fukaya-Oh-Ohta-Ono \cite{fooo:toric1}, \cite{fooo:toric2} and Cho-Oh
\cite{chooh:toric} on non-displaceability of moment fibers of
not-necessarily-Fano toric varieties and extend their results to toric
orbifolds, without using virtual fundamental chains.  Finally we
describe a conjectural relationship with Floer cohomology in the
quotient.
\end{abstract}

\maketitle

\tableofcontents

\section{Introduction} 

Let $X$ be a symplectic manifold. A Lagrangian submanifold $L \subset
X$ is (Hamiltonian) {\em displaceable} if there exists a Hamiltonian
diffeomorphism $\phi: X \to X $ such that $L \cap \phi(L) =
\emptyset$.  Otherwise, $L$ is called {\em non-displaceable}.
Hamiltonian non-displaceability questions are among the most basic in
symplectic topology.  For example if $L = \{ (x,x) | x \in X \}
\subset X^- \times X$ is the diagonal and $X$ is compact then Arnold
conjectured (and Floer \cite{floer:lag} later proved under certain
conditions) that $L$ is non-displaceable.  Floer introduced a method
for proving non-displaceability based on the study of a complex whose
underlying vector space is generated by intersection points $L \cap
\phi(L)$, if transversal, and whose differential counts finite-energy
holomorphic strips with boundary values in $(L,\phi(L))$.  The
resulting {\em Floer cohomology group} $HF(L,\phi(L))$ is independent
of the choice of $\phi$, so that if $L$ is displaceable then
$HF(L,\phi(L))$ vanishes.  On the other hand, one can sometimes
compute $HF(L,\phi(L))$ by, for example, taking $\phi$ small and
identifying it with the Morse homology.

In a series of papers \cite{fooo:toric1}, \cite{fooo:toric2}
Fukaya-Oh-Ohta-Ono used this strategy to prove non-displaceability
results for certain moment fibers of toric varieties, generalizing
earlier work of Cho-Oh \cite{chooh:toric} and Entov-Polterovich
\cite{entov:rigid}.  Most of these fibers have vanishing Floer
cohomology, and many are displaceable by elementary means, which we
will discuss further in a moment.  Fukaya et al show that a moment
fiber has non-vanishing Floer cohomology if there is a critical point
of a {\em potential} obtained by counting holomorphic disks with
boundary on the Lagrangian.  Fukaya et al were able to write the
potential as the sum of a {\em naive potential} plus quantum
corrections arising from sphere bubbles, and show that the naive
potential is given by an explicit formula which appeared in the paper
of the physicists Hori-Vafa \cite[5.16]{ho:mi}.

One can take a different approach to this problem using the
realization of a toric variety as a symplectic quotient.  Let $G$ be a
compact connected group with Lie algebra $\g$ and $X$ a Hamiltonian
$G$-manifold with moment map $\Phi: X \to \g^\dual := \Hom(\g,\R)$.
If $G$ acts freely on the level set $\Phinv(0)$ then the {\em
  symplectic quotient} $X \qu G = \Phinv(0)/G$ is a smooth symplectic
manifold, or more generally a symplectic orbifold if the action has
finite stabilizers.  In particular, any smooth projective toric
variety can be realized as the quotient $X \qu G$ of a vector space
$X$ by the action of a torus $G$.  Given a Lagrangian $L \subset X \qu
G$ we ask whether $L$ is displaceable.  The pre-image $\ti{L}$ of $L$
is a $G$-invariant Lagrangian in $X$.  An approach to
non-displaceability for Lagrangians in Hamiltonian $G$-manifolds was
introduced by Frauenfelder \cite{frauen:mfh}.  At first sight, his
theory looks even more complicated than that of Fukaya et al: he
counts pairs $(A,u)$ consisting of a connection $A \in
\Omega^1(\Sigma,\g)$ on $\Sigma := \R \times [0,1]$ together with a
map $u : \Sigma \to X$ satisfying a pair of {\em vortex equations}:
$u$ is holomorphic with respect to the connection determined by $A$
and the curvature $F_A$ is equal to minus the pull-back $u^* \Phi$ of
the moment map: $ F_A = - u^* \Phi \Vol_\Sigma $.  Here $\Vol_\Sigma
\in \Omega^2(\Sigma)$ is a choice of area form, in this case a
multiple of the standard area form, and constitutes a parameter in the
theory that can be varied.  In the limit $\Vol_\Sigma \to \infty$, the
vortex equations become equivalent to the Cauchy-Riemann equation for
a map to the quotient $X \qu G$, so one expects the gauged Floer
theory to be related to the Floer theory of the quotient \cite{ga:gw}.
In the case that $X$ has no holomorphic spheres, the gauged Floer
theory has better compactness properties than the theory in $X \qu G$.
Frauenfelder \cite{frau:thesis} used his gauged Floer theory to prove
a version of the Arnold-Givental conjecture in this context.

In this paper we study the zero-area limit $\Vol_\Sigma \to 0$ of
gauged Floer theory.  The corresponding limit for gauged Gromov-Witten
invariants was studied in Gonzalez-Woodward \cite{small}, and used to
prove a version of the abelianization conjecture for Gromov-Witten
invariants.  In the zero-area limit the vortex equations reduce to the
Cauchy-Riemann equation in $X$.  We show that the resulting {\em
  quasimap Floer theory} retains the relationship to
non-displaceability, that is, its non-vanishing obstructs
displaceability, in cases where it is defined.  For any toric moment
fiber one obtains an \ainfty algebra, in a way that avoids one of the
main technical complications of the theory of Fukaya et al: since all
holomorphic disks are regular for the standard complex structure,
there is no need for Kuranishi structures or virtual fundamental
chains.  A classification result of Cho-Oh \cite{chooh:toric} gives an
explicit formula for the holomorphic disks, see Theorem
\ref{blaschke}.  The classification leads to an explicit formula for
the curvature of the resulting \ainfty algebra and an explicit
criterion for the non-vanishing of the quasimap Floer cohomology.
This allows us to reproduce the results of Fukaya et al
\cite{fooo:toric1}, \cite{fooo:toric2}, as well as give extensions to
orbifold quotients such as weighted projective spaces.

The main result Theorem \ref{main} below is an explicit sufficient
condition for a moment fiber of a toric orbifold to be
non-displaceable; its form is the same as that of Fukaya et al
\cite{fooo:toric2}.  Let $X \cong \C^N$ be Hermitian vector space, $H
\cong U(1)^N$ the standard maximal torus of $\Aut(X)$, and $G \subset
H$ a sub-torus.  Choose a moment map for the $H$-action; this induces
a moment map for the action of $G$.  Suppose that the symplectic
quotient $X \qu G$ is locally free and so a symplectic orbifold.  It
has a residual action of a torus $T := H/G$, and the residual moment
map $\Psi: X \qu G \to \t^\dual$ defines a homeomorphism of $(X \qu
G)/T$ onto its {\em moment polytope} $\Psi(X \qu G)$.  We denote by
$v_1,\ldots,v_N \in \t$ the images of minus the standard basis vectors
$e_1,\ldots,e_N \in \h \cong \R^N$ in $\t$.  The polytope $\Psi(X \qu
G)$ is given by linear inequalities
\begin{equation} \label{lineq} 
\Psi(X \qu G) = \{ \lambda \in \t^\dual \ | \ l_i(\lambda) \ge 0 \} ,
\quad l_i(\lambda)/2\pi := \langle \lambda, v_i \rangle -c_i, \quad i =
1,\ldots, N .\end{equation}
where $\lan \cdot, \cdot \ran: \t^\dual \times \t \to \R$ is the
canonical pairing and $c_1,\ldots,c_N$ are constants given by the
choice of moment map.  Let $\Lambda$ be the {\em universal Novikov
  field} consisting of possibly infinite sums of real powers of a
formal variable $q$,
$$ \Lambda = \left\{ \sum_{n =0}^\infty a_n q^{d_n} , \quad a_n \in
\C, d_n \in \R, \quad \lim_{n \to \infty} d_n = \infty \right\} .$$
Let $\Lambda_+$ resp. $\Lambda_0$ denote the subrings consisting of
sums with only positive resp. non-negative powers.  Any fiber
$L_\lambda = \Psi^{-1}(\lambda)$ over an interior point $\lambda \in
\on{int}(\Psi(X \qu G))$ is a Lagrangian torus, namely a single free
$T$-orbit.  We identify $H^1(L_\lambda,\Lambda_0) \cong
H^1(T,\Lambda_0)^T \cong \t^\dual \otimes \Lambda_0$, so that in
particular for any $v \in \t$ and $\beta \in H^1(L_\lambda,\Lambda_0)$ we
have a pairing $\lan v, \beta \ran \in \Lambda_0$ and an exponential
$e^{\lan v, \beta \ran} \in \Lambda_0$.  Define the {\em Hori-Vafa
  potential}
$$ W^G_\lambda: H^1(L_\lambda,\Lambda_0) \to \Lambda_0, \quad \beta
\mapsto \sum_{i=1}^N e^{\lan v_i, \beta \ran } q^{l_i(\lambda)} ;$$
as in \cite[(5.16)]{ho:mi}.

\begin{theorem} \label{main}   If $ W^G_\lambda$ has a critical point, 
then $L_\lambda \subset X \qu G$ is non-displaceable.
\end{theorem} 
\noindent In the case that $X \qu G$ is compact and smooth and there
exists a non-degenerate critical point this is due to Fukaya et al
\cite{fooo:toric1}, and without the non-degeneracy condition in their
second paper \cite{fooo:toric2}.  Overlapping results using different
methods are given by Entov-Polterovich \cite{entov:rigid}.  Related
works include Alston \cite{alston:cliff}, Alston-Amorim \cite{al:tf},
and Abreu-Macarini, in progress, who compute the Floer cohomology of
toric moment fibers with the real part of the ambient toric variety.
The existence of $\lambda$ in the interior such that $W^G_\lambda$ has
a critical point for compact $X \qu G$ is proved in Fukaya et al
\cite[Proposition 4.7]{fooo:toric1} for toric varieties with rational
symplectic classes; the authors conjecture that the rationality
condition is not necessary.  For non-compact $X \qu G$, there may not
exist any such $\lambda$, for example if $X \qu G = \C$ then every
compact Lagrangian is displaceable.  As far as the author knows,
existence of a non-displaceable Lagrangian in an arbitrary compact
symplectic manifold is an open question.  We also prove a Theorem
\ref{bulk} which includes bulk deformations as Fukaya et al's second
paper \cite{fooo:toric2}.

Our proof differs from Fukaya et al in several ways.  We already
mentioned that we count disks in $X$ rather than in $X \qu G$.  The
proof in Fukaya et al depends on a detailed study of the correction
terms arising in their potential from sphere bubbling in $X \qu G$,
which is not necessary in our case since $X$ has no holomorphic
spheres.  Furthermore we use the combined Morse-Fukaya approach to the
construction of \ainfty algebras, in which one counts configurations
consisting of gradient lines and holomorphic disks with Lagrangian
boundary conditions; this avoids various difficulties with choices of
generic chains or smoothness of moduli spaces of stable maps.  A
generic perturbation of the gradient flow equations gives an \ainfty
structure on the space of Morse cochains of $L$.  A disadvantage of
this approach is that the source objects (tree disks etc.) have
somewhat more complicated moduli spaces than the usual realizations of
associahedra, multiplihedra etc.  Furthermore the structure maps of
the \ainfty algebras constructed this way do not satisfy a divisor
equation of the type described by Cho \cite{cho:products}, who noted
that the perturbation scheme may destroy the necessary
forgetful maps.  The latter requires us to take a slightly different
definition of the potential than Fukaya et al.  Then we have to show
that the special case of the divisor equation that we need does in
fact hold, and implies that any critical point of the potential gives
rise to non-vanishing Floer cohomology.  For pairs of Lagrangians
intersecting transversally in $X \qu G$, we define an \ainfty bimodule
by counting configurations of Floer trajectories in $X$, holomorphic
disks, and gradient trees.  The pre-image Lagrangians intersect only
cleanly in $X$, and this requires several results (exponential decay,
energy quantization, gluing) for Lagrangian clean intersections which
are probably known to experts, but which do not seem to have
completely appeared in the literature.

These results on non-displaceability should be contrasted with those
of McDuff \cite{mc:di} who gives a method for {\em displacing} the
fibers of a moment polytope of a toric variety, based on the
observation that if $\lambda$ is sufficiently close to the boundary of
$\Psi(X \qu G)$ then $L_\lambda$ is small in some Darboux chart and so
displaceable.  More precisely a vector $\alpha$ in the coweight
lattice $\t_\Z$ is {\em integrally transverse} to an open facet $F$ of
the moment polytope $\Psi(X \qu G)$ if $\{ \alpha \}$ can be completed
to a basis of $\t_\Z$ by vectors parallel to $F$.  A {\em probe} with
direction $\alpha \in \t_\Z$ and initial point $\mu \in F$ is the open
half of the line segment lying in the direction of $\alpha$ from
$\mu$.  If $\lambda$ lies in a probe, then $L_\lambda$ is
displaceable.


\begin{example} 
In the case $X \qu G = \P^1$ with moment polytope 
$[0,1]$, 
$ W_\lambda^G(\beta) = e^{\beta} q^{\lambda} + e^{-\beta} q^{1 -
  \lambda} $
which has a critical point iff $ \lambda = 1/2$.  The fiber
$L_{\lambda}$ for $\lambda = 1/2$ is the unique non-displaceable
moment fiber in $\P^1$, since any other fiber is displaced by a
rotation, and $L_\lambda$ is not displaceable since it separates
$\P^1$ into two disks of equal area.  \end{example}

\begin{example} This example is a case for which $X \qu G$ is
non-compact.  Suppose that $X = \C^3$ and $G = \C^*$ acts with
weights $-1,-1,1$, and the moment map is chosen so that 
$X \qu G$ is the blow-up of $X = \C^2$ at
$(0,0)$, so that the moment polytope is $ \{ (\lambda_1,\lambda_2) \in
\R_{\ge 0}^2, \lambda_1 + \lambda_2 \ge 1 \}$.  The gauged potential
is
$$ W_\lambda^G(\beta) = e^{\beta_1} q^{\lambda_1} +
e^{\beta_2} q^{\lambda_2} + e^{\beta_1 + \beta_2} q^{\lambda_1  + \lambda_2 - 1} .$$
This has a critical point iff $ (\lambda_1,\lambda_2) = (1,1)$, so the
fiber over $(1,1)$ is not displaceable.\end{example} 

\begin{example} 
This example appears in Fukaya et al \cite[Example 4.7]{fooo:toric1}
and illustrates the dependence of the number of non-displaceable
fibers on the choice of symplectic form: Suppose that $X \qu G$ is the
toric blow-up of $\P^1 \times \P^1$ with moment polytope
$$ \Psi(X \qu G) = \{ (\lambda_1,\lambda_2) \in [-1,1]^2 \ |
\ \lambda_1 + \lambda_2 \leq 1 + \alpha \} $$
for some real parameter $\alpha \in (-1,1)$ describing the size of the
blow-up.

\begin{proposition}  The unique non-displaceable toric fibers 
for $X \qu G$ the blow-up of $(\P^1)^2$ are described by the following
three cases: (i) for $\alpha = 0$, the fiber over $\lambda = (0,0)$
(ii) for $\alpha > 0$, the fibers over $(0,0)$ and $(\alpha,\alpha)$
(iii) for $\alpha < 0$, the fibers over
$(-\alpha/3,-\alpha/3),(\alpha, \alpha + 1)$, $(\alpha + 1,\alpha)$.
\end{proposition} 

\noindent See Figure \ref{blowup}.  

\begin{proof}  The gauged potential (which in this case
is the same as the potential in Fukaya et al) is
\begin{multline}
 W_\lambda^G(\beta) = e^{\beta_1} q^{1 + \lambda_1} + e^{\beta_2} q^{1
   + \lambda_2} + e^{-\beta_1} q^{1 - \lambda_1} + e^{-\beta_2}
 q^{1 - \lambda_2} + e^{-\beta_1-\beta_2} q^{1 + \alpha - \lambda_1 -
   \lambda_2} .\end{multline}
The critical points are described in \cite[Example 4.7]{fooo:toric1}.
The other fibers are displaceable by the technique of McDuff
\cite{mc:di}.
\end{proof} 
\noindent Note that the non-displaceable fibers ``collide and
scatter'' at $\alpha = 0$.  The number of non-displaceable fibers
depends on the symplectic form, but the multiplicities in each case do
add up to the dimension $5$ of the cohomology $H(X \qu G)$
\cite{fooo:toric1}.  The non-displaceable fibers for other choices of
symplectic form are discussed in Fukaya et al \cite{fooo:toric2}.  The
values of the potential, counted with multiplicity, match the
eigenvalues for the quantum action of $c_1(TX)$ on the quantum
cohomology $QH(X \qu G)$, see \cite{fooo:toric1}.  
\end{example} 

\begin{figure}[ht]
\begin{picture}(0,0)%
\includegraphics{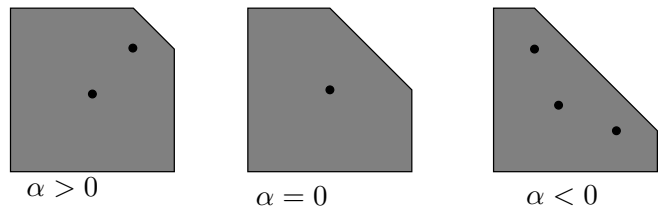}%
\end{picture}%
\setlength{\unitlength}{4144sp}%
\begingroup\makeatletter\ifx\SetFigFontNFSS\undefined%
\gdef\SetFigFontNFSS#1#2#3#4#5{%
  \reset@font\fontsize{#1}{#2pt}%
  \fontfamily{#3}\fontseries{#4}\fontshape{#5}%
  \selectfont}%
\fi\endgroup%
\begin{picture}(3894,1233)(79,-2222)
\put(1560,-2176){\makebox(0,0)[lb]{{$\alpha = 0 $}%
}}
\put(189,-2127){\makebox(0,0)[lb]{{$\alpha > 0 $}%
}}
\put(3177,-2176){\makebox(0,0)[lb]{{$\alpha < 0 $}%
}}
\end{picture}%
\caption{Non-displaceable fibers for the blow-up of $\P^1 \times \P^1$}
\label{blowup}\end{figure} 

\begin{example} 
This example involves the orbifold case.  Consider the weighted
projective plane, $X \qu G = \P(1,n_1,n_2)$ the symplectic quotient of
$X = \C^3$ by the $S^1$ action with weights $(1,n_1,n_2)$.  The
two-torus $T$ acts on $X \qu G$ in Hamiltonian fashion with moment
polytope $\Psi(X \qu G)$ the convex hull of $(0,0), (n_1,0), (0,n_2)$.

\begin{proposition}  For $X \qu G = \P(1,n_1,n_2)$ 
the fiber of $\Psi$ over $\lambda = \diag( (n_1n_2)/(n_1 + n_2 + 1))$
is non-displaceable.
\end{proposition}  

\begin{proof}  
The gauged potential is
$$ W^G_\lambda(\beta) = e^{\beta_1} q^{\lambda_1} + e^{\beta_2} q^{\lambda_2}
+ e^{-\beta_1 n_2 - \beta_2 n_1} q^{n_1 n_2 - n_2 \lambda_1 - n_1 \lambda_2}
.$$
The derivatives are given by 
$$ e^{-\beta_1} \partial_{\beta_1} W^G_\lambda(\beta) = q^{\lambda_1} - n_2 e^{- \beta_1(n_2 + 1) -\beta_2 n_1} q^{n_1n_2 - n_2 \lambda_1 - n_1 \lambda_2} $$
$$ e^{-\beta_2} \partial_{\beta_2} W^G_\lambda(\beta) = q^{\lambda_2} - n_1 e^{- \beta_1 n_2- \beta_2(n_1 +1)} q^{n_1n_2 - n_2 \lambda_1 - n_1 \lambda_2} .$$
This has solutions in $\beta \in \Lambda_0$ if and only if
$$ n_1 n_2 - (n_2 + n_1 + 1 )\lambda_1 - (n_1 + n_2) \lambda_2 = 0, 
 n_1 n_2 - (n_1 + n_2) \lambda_1 - (n_1 + n_2 + 1) \lambda_2 = 0 .$$
The unique solution to these equations is $ \lambda = \diag(
(n_1n_2)/(n_1 + n_2 + 1)) $ as claimed.
\end{proof}

\noindent
 The case of $X\qu G = \P(1,3,5)$ is shown below in Figure
\ref{p135}, taken from \cite{mc:di}, with 
shaded regions displaceable
by McDuff's probes.  The non-displaceable fiber over $\lambda =
(5/3,5/3)$ is surrounded by an open subset of fibers for which
displaceability is an open question.  
(The small line segment
connecting $(1,1)$ with $(3/2,3/2)$ consist of fibers that are not
displaceable by probes, either.  In all the previous examples except
this one, the combination of the McDuff method with the Floer
theoretic methods completely resolved the question of
non-displaceability of moment fibers.)  We remark that homological
mirror symmetry for the $B$-model on, in particular, $\P(1,3,5)$ is
proved in Auroux et al \cite{auroux:wpp}.  It would be interesting to
know whether the ``twisted sectors'' in the orbifold quantum
cohomology of these weighted projective spaces play any role in the
displaceability of moment fibers.  Note that in this case there are
two kinds of twisted sectors, coming from the two orbifold
singularities, and these match the intersection of the ``unknown
region'' with the boundary of the polytope.
\end{example} 

\begin{figure}[ht]
\includegraphics[height=2in]{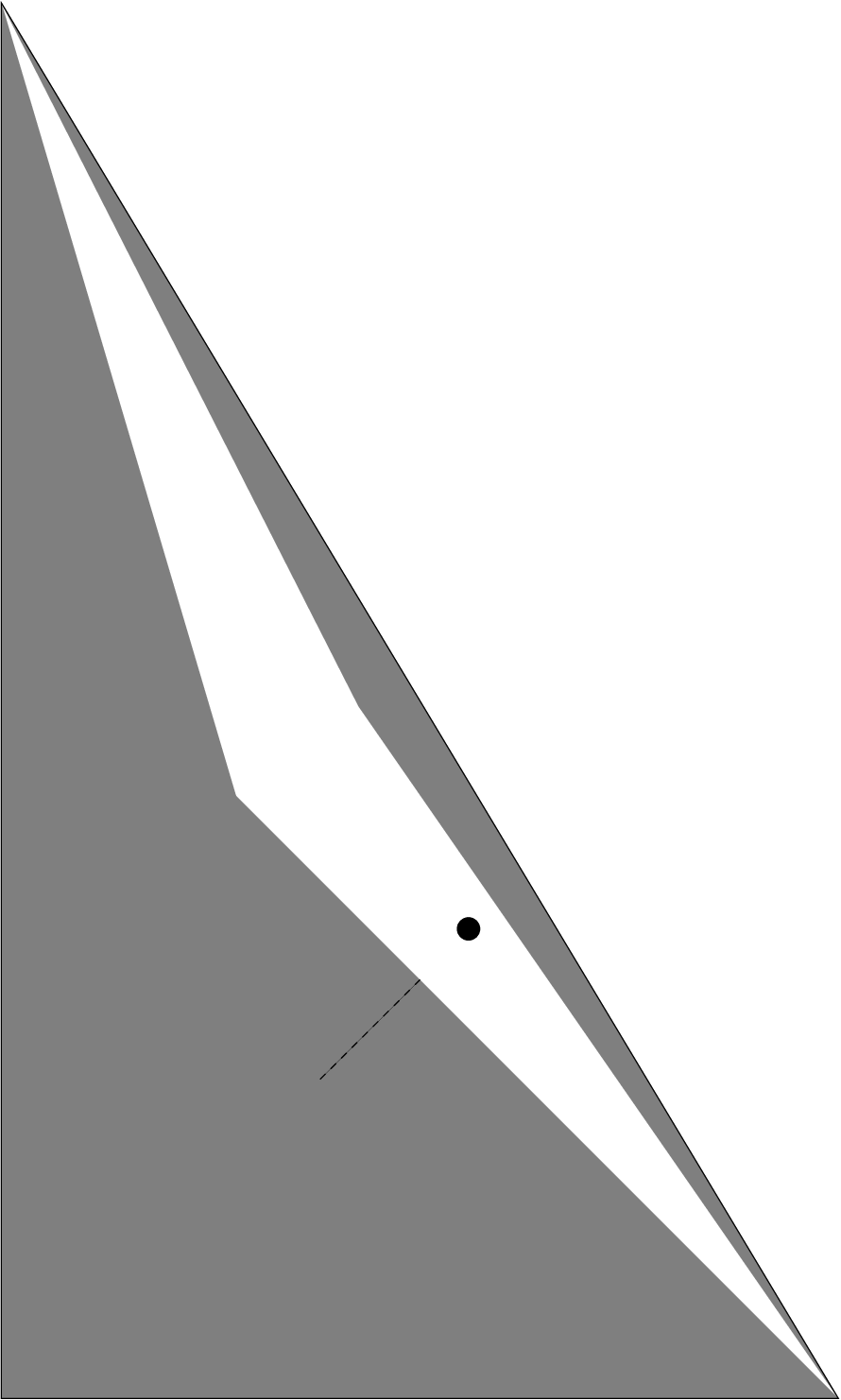}
\caption{Displaceable and non-displaceable fibers for $\P(1,3,5)$.} 
\label{p135}
\end{figure} 

\begin{example} 
As a final example we consider non-displaceability for the weighted
projective line $X \qu G = \P(1,2)$ with polytope $\Psi(X \qu G) =
[0,1]$ and orbifold singularity the fiber over $1$.  The fiber over
$2/3$ is non-displaceable, by the main Theorem \ref{main} and a
computation similar to the first example.  However, any fiber over
$\lambda \in [1/2,1]$ is non-displaceable in this case, since any
displacing flow would have to map the preimage of $[\lambda,1]$ into
the preimage of $[0,\lambda)$.  But this is impossible, since any
  Hamiltonian vector field vanishes at the orbifold point, the
  pre-image of $1$, which is therefore stationary under any
  Hamiltonian flow. 
\end{example} 

The idea of studying curves in $X$ instead of $X \qu G$ is not new and
fits into a long line of studies of {\em gauged sigma models}
\cite{wi:ph}, \cite{gi:eq}, \cite{ho:mi}, \cite{ho:phb}.  This
approach is the basis of Givental's study of mirror symmetry for
complete intersections in toric varieties \cite{gi:eq}.  Givental's
idea was to relate the invariants obtained by integration over moduli
spaces of {\em quasimaps} to those associated to the moduli spaces of
stable maps in the symplectic quotient, by algebraic arguments
involving localization.  By computing the (twisted) quasimap
invariants, Givental obtained for example a formula for the descendent
Gromov-Witten potential for the quintic three-fold.  The space of
quasimaps was identified with the moduli spaces of vortices by
J. Wehrheim \cite{jw:vi}.  In the last section we explain how
arguments similar to that of Gaio-Salamon \cite{ga:gw} lead {\em
  heuristically} to relationship with Floer theory in the quotient.
(However, as far as non-displaceability goes, there is no need to
rigorously prove the correspondence between the two families of
invariants, since the quasimap invariants already obstruct
displaceability.) One naturally expects an $A_\infty$-morphism from
the gauged Fukaya algebra to the Fukaya algebra of the quotient, which
we call the {\em open quantum Kirwan morphism}, and a relation between
the {\em bulk-deformed} potential on $X \qu G$ and the potential for
the gauged theory on $X$.  Carrying out this program would require not
only the compactness and gluing results above but a proper development
of virtual fundamental classes in this setting, which is why the
arguments of the last section are only conjectural.

We thank D. McDuff and M. Abouzaid for helpful suggestions and
encouragement, and G. Xu for pointing out several mistakes.

\section{Quasimap Floer cohomology}
\label{floer}

In this section we explain the definition of quasimap Floer
cohomology, which associates to a pair of Lagrangians $L_0,L_1$ in a
symplectic quotient $X \qu G$ an abelian group $HQF(L_0,L_1)$ by
counting strips in $X$ with boundary in the pre-images
$\ti{L}_0,\ti{L}_1$ of $L_0,L_1$ modulo the action of $G$. Actually,
none of the results of this section will be used for the main Theorem
\ref{main}, for the reason that the proof of Theorem \ref{main} uses
the cases that the two Lagrangians are equal (covered in Section 3)
together with the vanishing of the resulting cohomology when the
Lagrangian is displaceable (covered in Section 5).  However, it seemed
to the author that our Floer cohomology should be introduced before
\ainfty algebras for expositional reasons.

\subsection{Symplectic vortices}

The definition of quasimap Floer cohomology is motivated by a gauged
version of pseudoholomorphic curves introduced by Mundet and Salamon
and collaborators, see for example Cieliebak-Mundet-Gaio-Salamon
\cite{ci:symvortex}, which we now review.  Readers not interested in
the origin of quasimap Floer cohomology may skip to the following
section, with the caveat that without reading this section the
definition of quasimap Floer cohomology may seem rather miraculous.
Let $G$ be a compact connected group with Lie algebra $\g$ and $X$ a
Hamiltonian $G$-manifold with symplectic form $\omega \in \Omega^2(X)$
and proper moment map $\Phi: X \to \g^\dual$.  Let $X \qu G :=
\Phinv(0) /G$ denote the symplectic quotient.  We assume that the
action of $G$ on $\Phinv(0)$ is locally free, that is, has only finite
stabilizers, so that $X \qu G$ is a symplectic orbifold.  Furthermore
we assume that the action has trivial generic stabilizer.

Let $\J(X)$ denote the set of compatible almost complex structures $J
\in \End(TX)$ on $X$, $\J(X)^G$ the subset of invariants under the
action of $G$ by pull-back, and $\J(X \qu G)$ the set of compatible
almost complex structures on $X \qu G$.  There exists a map $\J(X)^G
\to \J(X \qu G)$, obtained by restricting $J$ to $T \Phinv(0) \cap
\g^\perp \cong \pi^* T(X \qu G)$, where $\pi: \Phinv(0) \to X \qu G$
is the projection.

Let $\Sigma$ be a compact Riemann surface.  Holomorphic maps from
$\Sigma$ to $X \qu G$ correspond to {\em gauged holomorphic maps} from
$\Sigma$ to $X$, as we now explain.  Let $P \to \Sigma$ be a principal
$G$-bundle.  Denote by $\A(P) \subset \Omega^1(P,\g)^G$ the space of
connections on $P$, and by $\G(P)$ the group of gauge transformations.
Any connection $A \in \A(P)$ and $J \in \J(X)^G$ induces an almost
complex structure on the associated fiber bundle $P \times_G X$.  Let
$\olp_A$ be the associated Cauchy-Riemann operator, so that if in
particular $\Gamma(P \times_G X)$ is the space of sections and $u \in
\Gamma(P \times_G X)$ then $\olp_A u \in \Omega^{0,1}(\Sigma, u^* T(P
\times_G X))$.  A {\em gauged holomorphic map} with bundle $P$ is a
pair $(A,u) \in \A(P) \times \Gamma( P \times_G X)$ satisfying $\olp_A
u = 0$.  Let $\cH(P,X)$ denote the space of gauged holomorphic maps
with bundle $P$; in general this is a singular subset of $\A(P) \times
\Gamma(P \times_G X)$.  If $J$ is integrable, then $\cH(P,X)$ admits a
formal symplectic structure depending on a choice of metric on $\g$
and area form $\Vol_\Sigma \in \Omega^2(\Sigma)$, given by as follows:
The pairing of two tangent vectors $(a_j,v_j) \in \Omega^1(\Sigma,P
\times_G \g) \oplus \Omega^0(\Sigma, u^*(P \times_G TX)), j = 0,1$ is
given by the integral over $\Sigma$
$$ (a_0,v_0),(a_1,v_1) \mapsto \int_\Sigma (a_0 \wedge a_1) + u^*
\omega ( v_0, v_1) \Vol_\Sigma $$
where the first term uses the metric on $\g$.  By formal, we mean that
each kernel of the linearized operator has a linear symplectic
structure given by the above formula, so that where $\cH(P,X)$ is
smooth it is symplectic.  The group $\G(P)$ acts on $\cH(P,X)$
preserving the formal symplectic structure and a formal moment map is
given by
$$ \cH(P,X) \to \Omega^2(\Sigma, P \times_G \g), 
\quad (A,u) \mapsto F_A + u^*(P \times_G \Phi)\Vol_\Sigma $$
where $\Omega^2(\Sigma, P \times_G \g)$ is identified with a subset of
the dual of the Lie algebra $\Omega^0(\Sigma, P \times_G \g)$ of the
group of gauge transformations via the pairing given by integration
and the metric on $\g$.  The formal symplectic quotient $M(P,X) :=
\cH(P,X) \qu \G(P) $ is the moduli space of {\em symplectic vortices}
$$ M(P,X) = \left\{ \begin{array}{c}
(A,u) \in \A(P) \times \Gamma(P \times_G X) \\ 
\olp_A u = 0  \\ \ F_A + u^*(P \times_G \Phi)\Vol_\Sigma = 0 
\end{array} \right\} / \G(P) .$$   
Define $M(\Sigma,X)$ to be the union of $M(P,X)$ over isomorphism
classes of bundles $P$.  Note the dependence on the choice of
$\Vol_\Sigma$.  In the infinite area limit the second equation becomes
$u^* (P \times_G \Phi) = 0 $ and $M(\Sigma,X)$ is then the moduli
space of holomorphic maps from $\Sigma$ to $X \qu G$.  Indeed any
solution descends to a holomorphic map to $X \qu G$.  Conversely any
holomorphic map to $X \qu G$ defines a pair $(A,u)$, by choosing a
connection on the bundle $\Phinv(0) \to X \qu G$ and taking $A$ to be
the pull-back connection.  In general one needs to compactify the
moduli space in order to define invariants but in certain
circumstances the moduli space is already compact, see for example
\cite{ci:symvortex}.  For example, J. Wehrheim \cite{jw:vi} shows that
if $\Sigma = \P^1, X = \C^n, G = S^1$ acting diagonally, then
$M(\Sigma,X)$ is diffeomorphic to $\bigcup_{d \ge 0} \P^{nd- 1} $.
Thus the moduli space of symplectic vortices is already compact while
the moduli space of maps to $X \qu G = \P^{n-1}$ has a natural
compactification, the moduli space of stable maps, whose boundary is
complicated.

Frauenfelder's thesis \cite{frau:thesis} exploited the better
compactness properties of the vortex equations to prove a version of
the Arnold-Givental conjecture.  More precisely, suppose that $L_0,L_1
\subset X$ are compact invariant Lagrangian submanifolds.  In good
cases Frauenfelder constructed a {\em gauged Floer cohomology} by
counting vortices on $\Sigma := \R \times [0,1]$.  Since any bundle
over $\Sigma$ is trivial, a symplectic vortex consists of a pair $A
\in \Omega^1(\Sigma,\g)$ of a connection $A$ on $\Sigma := \R \times
    [0,1]$ together with a map $u: \Sigma \to X$ such that
$$\olp_A u = 0, \quad u(s,j) \in L_j, j = 0,1, \forall s \in \R, 
\quad F_A +
u^* \Phi \Vol_\Sigma = 0 
 $$
modulo gauge transformation, where in this case $\Vol_\Sigma = \d s
\wedge \d t$.  The precise details will not concern us here, since we
work in a slightly different set-up.  

The gauged Floer cohomology (if everything is well-defined) is {\em
  independent of the choice of area form $\Vol_\Sigma$} by a standard
continuation argument, similar to the one giving independence of the
width of the strip in \cite{we:co}.  (In the case of gauged
Gromov-Witten invariants, the dependence on the choice of area form
was studied in Gonzalez-Woodward \cite{cross}.)  Therefore, one
expects an equivalent theory obtained by setting $\Vol_\Sigma = 0$,
the opposite limit from the one related to the Floer cohomology on the
quotient.  In this case the theory drastically simplifies: the
equation $F_A = 0$ implies that $A$ is gauge equivalent to the trivial
connection, in which case the other equation becomes the standard
Cauchy-Riemann equation.  However, since the trivial connection has
automorphism group $G$, the resulting moduli space is that of the
usual moduli space of holomorphic strips, modulo the action of $G$.
Since the gauge theory becomes somewhat trivial in this case, we call
the resulting Floer cohomology {\em quasimap Floer cohomology} in
cases where it is defined, by analogy with Givental's use of the term
{\em quasimaps}.  

\subsection{Holomorphic quasistrips} 

Having motivated the study of holomorphic disks modulo a group action,
we now develop a Floer theory for Lagrangians that are inverse images
of Lagrangians from the symplectic quotient.  Let $X$ be a Hamiltonian
$G$-manifold with $G$ acting locally freely on $\Phinv(0)$.  We say
that $L \subset X$ is {\em $G$-Lagrangian} if $\dim(L) = \dim(X)/2$
and the equivariant symplectic form vanishes on $L$, that is, the
restriction of the symplectic form and moment map vanish on $L$.  We
will always assume that the action of $G$ on $L$ is free, so that
$\ti{L}/G$ is contained in the smooth locus of $X \qu G$.  The map
$\ti{L} \mapsto \ti{L}/G$ defines a bijection between $G$-Lagrangians
in $X$ and Lagrangians in $X \qu G$.

This correspondence extends to Lagrangian branes as follows.  Suppose
that $X$ is equipped with a $G$-equivariant $N$-fold Maslov cover
$\Lag^N(X) \to \Lag(X)$.  One obtains an induced $N$-fold Maslov cover
$\Lag^N(X \qu G)$ on $X \qu G$ by taking the quotient $\Lag^N(X) |
\Phinv(0) / G$ and restricting to lifts of Lagrangian subspaces of
$T(X \qu G)$.  A {\em $G$-Lagrangian brane} is an oriented Lagrangian
submanifold equipped with a $G$-equivariant spin structure, a
$G$-equivariant flat $\Lambda$-line bundle, and a $G$-equivariant
grading, that is, a $G$-equivariant lift of the map $\ti{L} \to
\Lag(X) | \ti{L}$ to $\Lag^N(X) | \ti{L}$.  There is a one-to-one
correspondence between $G$-Lagrangian branes in $X$ and Lagrangian
branes in $X \qu G$, given by $\ti{L} \mapsto \ti{L} / G$, the induced
orientations and spin structure induced from a choice of orientation
and the left invariant spin structure on $G$, induced by the
trivialization $TG \cong G \times \g$ via the action via right
multiplication.  In our example, $G$ will be a torus, $N = 2$ and
$\Lag^2(X) \to \Lag(X)$ is the double cover given by the choice of
orientation.  This Maslov cover is $G$-equivariant since $G$ is
connected and so acts trivially on the orientations.

The line bundles in our brane structures will arise as follows.  Any
flat $\Lambda$-line bundle on $L = \ti{L}/G$ is determined by a {\em
  holonomy map} which, since the structure group is abelian, descends
to a map on the underlying homology classes:
$$ \Hol_L: H_1(L) \to \Lambda - \{ 0 \} .$$
In particular if $L$ is a torus then $H_1(L)$ is torsion-free and any
cohomology class $b \in H^1(L,\Lambda_0)$ gives rise to a flat
$\Lambda$-line bundle with holonomy around a loop representing a
homology class $a \in H_1(L)$ is given by the pairing $e^{\lan a,b
  \ran} \in \Lambda_0 \subset \Lambda$.  (Note that the
well-definedness of the exponential requires coefficients in
$\Lambda_0$.)

The quasimap Floer cochain complex is freely generated by generalized
intersections of transversally intersecting Lagrangians in the
quotient.  Let $L_0,L_1 \subset X \qu G$ be Lagrangian submanifolds,
and $\ti{L}_0,\ti{L}_1 \subset X$ their $G$-Lagrangian lifts to $X$.
Let $H \in C^\infty([0,1] \times X)^G$, let $H \qu G \in
C^\infty([0,1] \times X \qu G)$ the corresponding family of functions
on the symplectic quotient $X \qu G$ and let $(H \qu G)^\#_t \in
\Vect(X \qu G), t \in [0,1]$ denote the corresponding Hamiltonian
vector fields.  Let $\cI(L_0,L_1,H)$ denote the set of perturbed
intersection points in $X \qu G$,
\begin{equation} \label{genint}
 \cI(L_0,L_1,H) = \left\{x: [0,1] \to X \qu G, x(j) \in L_j, j = 0,1,
 \quad (\ddt x)(t) = (H \qu G)^\#_t(x(t)) \right\} .\end{equation}
Let $\phi_t \qu G \in \Aut(X \qu G)$ denote the flow of $(H \qu
G)^\#$.  We require that $(\phi_1 \qu G)(L_1) \cap L_0$ is transverse,
so that $\cI(L_0,L_1,H)$ is finite and the intersection $\ti{L}_0 \cap
\ti{L}_1$ is {\em clean}, that is, $T \ti{L}_0 \cap T\ti{L}_1 = T
(\ti{L}_0 \cap \ti{L}_1)$, so that $\ti{L}_0 \cap \ti{L}_1$ is a
finite union of orbits of $G$.

The differential in quasimap Floer cohomology is defined by counting
holomorphic strips modulo the group action.  Let $\Sigma := \R \times
[0,1]$. Given a Hamiltonian perturbation $\ti{H} \in
\Omega^1(\Sigma,C^\infty(X))$, we say that a map $\Sigma \to X$ is
      {\em $(J,\ti{H})$-holomorphic} if $\d u - \ti{H}^\#(u)$ is a
      holomorphic map from $T_z \Sigma \to T_{u(z)} X$, for each $z
      \in \Sigma$, where $\ti{H}^\# \in \Omega^1(\Sigma, \Vect(X))$ is
      the Hamiltonian vector field associated to $\ti{H}$.  For any
      $(J,\ti{H})$-holomorphic map $u: \Sigma \to X$ the symplectic
      area resp. energy
$$ A(u) := \int_\Sigma u^* \omega, \quad E_{\ti{H}}(u) := \hh
      \int_\Sigma | \d u - \ti{H}^\#(u) |^2 \d s \wedge \d t$$
are related by an identity involving the curvature of the connection
determined by $\ti{H}$ \cite[8.1.9]{ms:jh}: Let $\ti{H} = \ti{H}_s \d
s + \ti{H}_t \d t$.  The {\em curvature} of the connection on $\Sigma
\times X$ defined by $\ti{H}$ is
$$ R_{\ti{H}} = (\partial_s \ti{H}_t - \partial_t \ti{H}_s + \{
\ti{H}_s, \ti{H}_t \}) \d s \wedge \d t .$$
Then the area-energy identity of $(J,\ti{H})$-holomorphic maps is
\begin{equation} \label{energyarea}
 E_{\ti{H}}(u) = A(u) + \int_\Sigma R_{\ti{H}}(u) .\end{equation}
In particular, as long as the curvature $R_{\ti{H}}$ is bounded, then
any sequence of $(J,\ti{H})$-holomorphic maps has bounded energy iff
it has bounded symplectic area.  Given a function $H \in
C^\infty([0,1] \times X)$, a $(J,H)$-holomorphic strip is a
$(J,\ti{H})$-holomorphic strip for $\ti{H} = H \d t$.  For such strip,
the energy is equal to the symplectic area, and depends only on its
homotopy class.

\begin{definition}  Let $H \in C^\infty([0,1] \times X)^G$.  A
 {\em $(J,H)$-holomorphic quasistrip} with boundary in
 $\ti{L}_0,\ti{L}_1$ is a $(J,H)$-holomorphic map $\Sigma \to X$ with
 boundary in $\ti{L}_0,\ti{L}_1$.  An {\em isomorphism} of holomorphic
 quasistrips $u_0,u_1$ is an element $g \in G$ and an element $s_0 \in
 \R$ such that $u_1(s + s_0,t) = g u_0(s,t)$ for all $s \in \R, t \in
   [0,1]$.
\end{definition} 

\noindent That is, a quasistrip is the same as a strip, except that
the notion of isomorphism is different.  

\begin{remark} \label{integrable} 
If $X$ is K\"ahler (that is, the almost complex structure is
integrable) compact and the Hamiltonian $H$ vanishes then then any
holomorphic quasistrip defines a holomorphic strip in the quotient $X
\qu G$ as follows.  Let $G_\C$ be the complexification of $G$.  Since
$X$ is compact, the action of $G$ extends to an action of $G_\C$.  The
    {\em semistable locus} $X^{\ss}$ of $X$ is the smallest
    $G_\C$-invariant open set containing $\Phinv(0)$, and is equal to
    $G_\C \Phinv(0)$ if the action of $G$ on $\Phinv(0)$ has finite
    stabilizers, see for example Kirwan \cite{ki:coh}.  Furthermore
    $X^{\ss}$ is the complement of a finite union of $G$-stable
    subvarieties of positive codimension.  The composition of $u$ with
    $X^{\ss} \to X \qu G$, where defined, defines a map from $\R
    \times [0,1]$ to $X \qu G$ on the complement of a finite set, and
    extends over $\R \times [0,1]$ by removal of singularities.
    Conversely, any map $v: \R \times [0,1] \to X \qu G$ lifts to a
    quasistrip, since the holomorphic $G_\C$-bundle $v^* (X^{\ss} \to
    X \qu G)$ is trivial.  More generally, a similar discussion holds for
    non-compact $X$ under the assumption that the $G$ action extends
    to an action of $G_\C$.
\end{remark}

Let $M(L_0,L_1;H)$ denote the moduli space of isomorphism classes of
$(J,H)$-holomorphic quasimaps of finite energy with boundary in
$\ti{L}_0,\ti{L}_1$.

The following lemmas on holomorphic strips with clean intersection
Lagrangian boundary conditions were developed jointly with F. Ziltener
several years ago, and are probably known to experts.

\begin{lemma} \label{cleanquant}  Let $X$ be a compact or convex symplectic manifold equipped with a compatible almost complex structure $J$, and 
 $L_0,L_1 \subset X$ compact Lagrangians intersecting cleanly.  Then
  (i) there exists an open neighborhood $U$ of the intersection $L_0
  \cap L_1$ such that any finite energy $(J,H)$-holomorphic strip $u:
  \R \times [0,1] \to U$ with boundary in $L_0,L_1$ is trivial (ii)
  there exists a constant $\hbar > 0$ such that any
  $(J,H)$-holomorphic strip $u: \R \times [0,1] \to X$ with boundary
  in $L_0,L_1$ has energy $E(u)$ at least $\hbar$.
\end{lemma} 

\begin{proof} 
(i) By the local model for clean intersections \cite[Proposition
    C.3.1]{ho:an3}, there exists a neighborhood $U$ of $L_0 \cap L_1$
  and a strong deformation retract $\psi: [0,1]\x U \to U$ to $L_0
  \cap L_1$ preserving $L_0,L_1$.  Using the Cartan homotopy identity,
  one can construct $\alpha \in \Omega^1(U)$ with $\d \alpha = \omega$
  so that $\alpha$ vanishes on $L_0, L_1$: Let $V_r \in \Vect(U)$ be
  the time-dependent vector field generating $\psi$, $ V_r = (\d / \d
  r) \psi(x,r) .$
The Poincar\'e formula
$$ \alpha = \int_0^1 \psi_r^* \iota(V_r) \omega \d r $$
produces the required primitive since 
$$
 \d \alpha =
\int_0^1 \psi_r^* L_{V_r} \omega \d r 
= \int_0^1 \ddt \psi_r^*  \omega \d r= \psi_1^*  \omega - \psi_0^* \omega = \omega. $$
The pull-back $i_j^* L_j$ of $\alpha$ to $L_j,j = 0,1$ is
$$ i_j^* \alpha = \int_0^1 \psi_r^* i_j^* \iota(V_r) \omega \d r =
\int_0^1 \psi_r^* \iota(V_r) i_j^* \omega \d r = 0 $$
since $V_r$ is tangent to $L_0,L_1$.  By Stokes' theorem, 
$$
E(u)= \lim_{s \to \infty} \int_{[-s,s]\x[0,1]}u^*\om 
= \lim_{s \to \infty}
  \int_{\{s\}\x[0,1]}u^*\al-\int_{\{-s\}\x[0,1]}u^*\al. 
$$
%
The energy of $u$ restricted to $[\pm (s-1), \pm (s +1)] \times [0,1]$
goes to zero as $s \to \infty$, since $u$ is finite energy. It follows
by the mean value inequality that $\sup_{t \in [0,1]} |\d u(s,t)| \to
0$ as $s \to \pm \infty$, so $u$ has energy zero and must be trivial.
(ii) Suppose otherwise that there exists a sequence $u_\nu$ of
holomorphic strips with energy $E(u_\nu) \to 0$ but each $E(u_\nu)$
non-zero.  A standard argument using compactness shows that there
exists a number $ \eps > 0$ such that any point in $x$ within $\eps$
of both $L_0$ and $L_1$ lies in the open subset $U$ from part (i).  By
the mean value inequality, for $\nu$ sufficiently large the image of
$u_\nu$ is within distance $\eps$ of $L_0$ and $L_1$ (integrate the
derivative over the segments $ \{ s \} \times [0,t]$ and $\{ s \}
\times [t,1]$) and so is contained in $U$.  By part (i), $u_\nu$ is
trivial.  Hence $E(u_\nu)$ vanishes, which is a contradiction.
\end{proof}  

\begin{remark}  The energy quantization lemma in McDuff-Salamon 
\cite[4.1.4]{ms:jh} does not use a symplectic structure, while the
proof above does since it uses the energy-area relation for
holomorphic maps.
\end{remark}

\begin{lemma}   \label{expdecay}   Suppose that $L_0,L_1$
are compact Lagrangians with clean intersection in a symplectic
manifold $X$.  There exist constants $\eps, \delta, C > 0$ such that
if $u: [-S,S] \times [0,1] \to X$ is a holomorphic strip with boundary
conditions $L_0,L_1$ with $E(u) < \eps$ then $E (u |_{[-S + s, S - s]
  \times [0,1]}) < C e^{- \delta s} E(u)$ and $| \sup \d u |_{[-S + s,
    S - s] \times [0,1]} < C e^{-\delta s/2} \sqrt{E(u)}$ for $s \in
[1,S]$.  Furthermore, if $u: \R \times [0,1] \to X$ is holomorphic
with finite energy then $u$ converges exponentially fast to limits $u(
\pm \infty,t) \in \cI(L_0,L_1)$ as $s \to \pm \infty$: there exist
constants $C,\delta > 0$ such that $\dist(u(s,t), u(\pm \infty,t)) < C
e^{\mp \delta s}$ and $| \d u (s,t) | < Ce^{ \mp \delta s}$ for $\pm
s$ sufficiently large.
\end{lemma} 

\begin{proof}
Pozniak \cite[Lemma 3.4.5]{po:cl} proves a relative version of the
isoperimetric inequality for the relative action of paths for
Lagrangian clean intersections: The {\em length} of a path $x: [0,1]
\to X$ is $\ell(x) = \int_0^1 | \dot x | \d t $.  for sufficiently
small paths $x: [0,1] \to X, 0 \mapsto L_0, 1 \mapsto L_1$ the {\em
  relative action} of $x$ is
$$\cA_{L_0,L_1}(x) = - \int_{[0,1]^2} u^*\omega $$ 
where $u: [0,1]^2 \to X$ is a smooth map satisfying $u(0,t) \in L_0
\cap L_1, u(1,t) = x(t)$ for $t \in [0,1]$ and $ u(s,i) \in L_i$ for
$s \in [0,1], i \in \{ 0,1 \}$ such that for each $s$, the path
$u(s,\cdot)$ has sufficiently small length.  (For a precise discussion
of what sufficiently small means in the context of vortices, see
Ziltener \cite{zil:decay}.) One then has a {\em relative isoperimetric
  inequality}: there exist constants $\de,C>0$ such that the following
holds. If $x:[0,1]\to X$ is a path satisfying $x(i)\in L_i$, for
$i=0,1$ and $\ell(x)<\de$ then the action is defined and
$|\cA_{L_0,L_1}(x)|\leq C\Vert\dot x\Vert_2^2.$
Furthermore, after possibly shrinking $\de$, for every pair $s_-\leq
s_+$ and every smooth map $u:\Si:=[s_-,s_+]\x[0,1]\to X$ the following
holds. If $u(s,i)\in L_i$, for $i=0,1$, and $\ell(u(s,\cdot))<\de$,
for every $s\in[s_-,s_+]$, then the actions of $u(s_-,\cdot)$ and
$u(s_+,\cdot)$ are defined and one has an {\em area-action identity}:
$$\int_\Si
u^*\om=-\cA_{L_0,L_1}(u(s_+,\cdot))+\cA_{L_0,L_1}(u(s_-,\cdot)). $$
Then the same convexity argument in \cite[Lemma 4.7.3]{ms:jh} proves
the first claim for the energy.  Using the mean value inequality one
obtains an estimate for the first derivative $\d u$.  The final claim
follows by restricting $u$ to $\pm [0,S] \times [0,1]$ and taking $S
\to \infty$, deriving an estimate on the distance from the estimate on
the first derivative.
\end{proof} 

\begin{remark} The constant in exponential decay cannot be chosen
arbitrarily close to $1$ as in McDuff-Salamon \cite{ms:jh}; it depends
on the geometry of intersection of the Lagrangians. The lemma also
does not hold for non-compact Lagrangians in general, for a similar
reason.  
\end{remark} 

\begin{remark} 
In the case considered in this paper, an alternative argument is
possible: Suppose that $X$ is K\"ahler and $\ti{L}_0,\ti{L}_1$ are
Lagrangians in $\Phinv(0)$ that are inverse images of Lagrangians
intersecting transversely in $X \qu G$. Near any point $x \in \ti{L}_0
\cap \ti{L}_1$ we may write $X$ holomorphically as the product of an
open subset of $X \qu G$ and $G_\C$, so that the Lagrangians
$\ti{L}_0$ resp. $\ti{L}_1$ are the product of the Lagrangians $L_0$
resp. $L_1$ and $G$.  Then the exponential decay estimates in Lemma
\ref{expdecay} are a consequence of the corresponding exponential
decay estimates for the transversely intersecting pair $L_0,L_1$ in $X
\qu G$, and for holomorphic strips in $G_\C$ with boundary in $G$.
\end{remark} 

For any finite energy holomorphic map $u: \Sigma := \R \times [0,1]
\to X$ with Lagrangian boundary conditions in $\ti{L}_0,\ti{L}_1$, let
$\partial_j u$ denote the restriction of $u$ to $\R \times \{ j \}$
and for $\alpha > 0 $ define a {\em linearized Cauchy-Riemann operator}
\begin{multline} \label{du}
 D_u: \Omega^0(\Sigma, u^* TX, (\partial_0 u)^* T
 \ti{L}_0,(\partial_1u)^* T \ti{L}_1)_{1,p,\alpha} \to
 \Omega^{0,1}(\Sigma, u^* TX)_{0,p,\alpha}, \\ \xi \mapsto
 \nabla_H^{0,1} \xi - \hh (\nabla_\xi J) J \partial_{J,H} u
 \end{multline}
c.f. McDuff-Salamon \cite[p. 258]{ms:jh}.  The Sobolev spaces above
are defined as follows.  For integers $k \ge 0$ and $ p \ge 1 $ and
$\xi$ a $j$-form on $\Sigma$ with values in $u^* TX$ of class
$W^{k,p}_{\loc}$ set
\begin{equation} \label{weight} 
\Vert \xi \Vert_{k,p,\alpha} = \sum_{i + j \leq k} \Vert e^{\alpha
  \gamma(s) s} \nabla_s^i \nabla_t^j \xi \Vert_p
 \end{equation} 
where $\gamma(s) = -1, s < -1$ and $\gamma(s) = 1, s > 1$.  For $k \ge
0$ and $p \ge 1$ let $\Omega^j(\R \times [0,1], u^*
TX)_{k,p,\alpha}^{\pre}$ denote the space of $\xi$ with finite
$k,p,\alpha$ norm.  Let $\Omega^0(\R \times [0,1], u^*
TX)^{\on{const}}$ be the space of smooth sections that are covariant
constant in a neighborhood of infinity.  Then
$$ \Omega^0(\R \times [0,1],u^* TX)_{1,p,\alpha} := \Omega^0(\R \times
[0,1], u^* TX)_{1,p,\alpha}^{\pre} + \Omega^0(\R \times [0,1], u^*
TX)^{\on{const}} $$
is the space of sections of class $1,p$ that differ in a neighborhood
of infinity from a covariant constant section by a section of class
$1,p,\alpha$; a norm on this space in a neighborhood of each end is
given by the norm of the limit $\xi(\pm \infty)$ plus the norm of the
element $\xi - \xi(\pm \infty)$ of $\Omega^0(\R \times [0,1], u^*
TX)_{1,p,\alpha}^{\pre}$ obtained by subtracting off the limit: $
\Vert \xi \Vert_{1,p,\alpha} = \Vert \xi \Vert_{1,p,\alpha}^{\pre} +
\Vert \xi(\infty) \Vert $ for $\xi$ supported on some $[0,\infty)
  \times [0,1]$.  Patching together these norms with the $1,p$ norm on
  a compact subset of $\R \times [0,1]$ defines a norm on $\Omega^0(\R
  \times [0,1],u^* TX)_{1,p,\alpha}$, see e.g. \cite[4.7]{ab:ex}.  Let
$$ \Omega^0(\R \times [0,1], u^* TX, (\partial_0 u)^* T
  \ti{L}_0,(\partial_1u)^* T \ti{L}_1)_{1,p,\alpha} \subset
  \Omega^0(\R \times [0,1], u^* TX)_{1,p,\alpha}$$
denote the subspace with boundary values in $T \ti{L}_0,T \ti{L}_1$;
in particular this means that any element $\xi$ has exponential
convergence on the ends to an element of $T \ti{L}_0 \cap T \ti{L}_1$.
Let
$$\Omega^1(\Sigma, u^* TX)_{0,p,\alpha}:= \Omega^1(\Sigma,u^* TX)_{0,p,\alpha}^{\pre} $$
denote the space of one-forms with exponential decay of class
$0,p,\alpha$; this space does not contain forms that are constant but
non-zero on the ends.  Because $u$ has exponential decay, see Lemma
\ref{expdecay}, the map $D_u$ of \eqref{du} is well-defined for
sufficiently small $\alpha > 0$ and is a Fredholm operator by
combining standard estimates for compactly-supported sections with
totally real boundary conditions with arguments for manifolds with
cylindrical ends as in Lockart-McOwen \cite{loc:ell}.  See also
McDuff-Salamon \cite[Section 3.1]{ms:jh} and Abouzaid \cite{ab:ex} who
treats holomorphic strips with equal boundary conditions using
weighted Sobolev spaces; the equality of the boundary conditions is
only used to obtain exponential decay via removal of singularities,
which we have obtained instead via Pozniak's relative isoperimetric
inequality in Lemma \ref{expdecay}.  We say that $u$ is {\em regular}
if $D_{u}$ is surjective.

Let $M^{\reg}(L_0,L_1;H)$ denote the moduli space of isomorphism
classes of regular, finite energy $J,H$-holomorphic quasimaps with
boundary in $\ti{L}_0,\ti{L}_1$.

\begin{proposition} \label{mfd}  
The space $M^{\reg}(L_0,L_1;H)$ is a smooth finite dimensional
manifold with tangent space at the isomorphism class $[u]$ of a
quasimap $u$ given by $ T_{[u]} M^{\reg}(L_0,L_1;H) = \ker(D_{u})/(\g
+ \R)$ for sufficiently small $\alpha > 0$.
\end{proposition} 

\begin{proof} This is a standard implicit function theorem 
argument.   Consider the map
$$ \cF_u: \Omega^0(\R \times [0,1], u^* TX, (\partial_0 u)^* T
  \ti{L}_0,(\partial_1u)^* T \ti{L}_1 )_{1,p,\alpha} \to
  \Omega^{0,1}(\R \times [0,1], u^* TX)_{0,p,\alpha} $$
$$ \xi \mapsto \cT_u(\xi)^{-1} \olp_{J,H} \exp_u \xi  $$
where $\cT_u(\xi)$ denotes parallel transport along $\exp_u(\xi)$
using the complex-linear modification of the Levi-Civita connection
$\ti{\nabla} = \nabla - \hh (\nabla J) J$, and the exponential map is
defined using metrics $g_t$ so that $g_j$ is totally geodesic with
respect to $L_j, j = 0,1$, so that $TL_j$ maps to $L_j$.  This means,
however, that $g_j$ is {\em not} the metric corresponding to the
choice of almost complex structure.  This difference gives rise to
additional quadratic corrections in the map $\cF_u$, which are
explained in more detail in \cite[Remark 2.2]{ww:quilts},
\cite[Section 4.3]{mau:gluing}.  Sobolev multiplication and
exponential decay estimates for $u$ above imply that $\cF_u$ is a
smooth map of Banach manifolds.  Elliptic regularity
\cite[B.4.1]{ms:jh} implies that any solution $\xi$ to $\cF_u(\xi) =
0$ is smooth.  (The regularity theorem there only applies to the case
of a single Lagrangian, but the same proof holds
when a Lagrangian is assigned to each component of the boundary.)  If
$u$ is regular, the implicit function theorem implies that
$\cF_u^{-1}(0)$ is a smooth manifold modelled on $\ker D_{u}$. By the
exponential decay Lemma \ref{expdecay} for $\alpha $ sufficiently
small any nearby solution is of the form $\exp_u(\xi)$ for some $\xi
\in \Omega^0(\R \times [0,1], u^* TX, (\partial_0 u)^* T \ti{L}_0,
(\partial_1u)^* T \ti{L}_1)_{1,p,\alpha}$, so any nearby holomorphic
strip is represented by a point in $\cF_u^{-1}(0)$.  Since $J$ is
$G$-invariant, $G$ acts by pull-back on the moduli space of
holomorphic strips.  The action of $G$ is free and proper so the
quotient, the moduli space of parametrized quasistrips, is a smooth
manifold.  The action of $\R$ by reparametrization on the resulting
moduli space is also free and proper on the moduli space of
non-constant trajectories: if $g_i u(s + s_i, t)$ converges to some
$u(s,t)$ for some sequence $s_i \to \infty$ then we must have $u(s,t)$
constant, by exponential decay.  It follows that the quotient has a
smooth structure with the claimed tangent space.  (Note that the image
of $\R$ in $\ker(D_u)$ may be trivial, if the trajectory is constant,
since it is obtained by differentiating the trajectory in the $s$
direction.)
\end{proof} 

To obtain a complex structure for which the moduli space is regular we
follow the notation in Frauenfelder \cite[p. 49]{frau:thesis}, which
itself follows closely Floer-Hofer-Salamon \cite{fhs:tr}.  To set up
the notation, for each $\xi \in \g$ we denote by $\xi_X \in \Vect(X)$
the generating vector field.  Let $J : TX \to TX$ be a $G$-invariant
compatible almost complex structure.  The infinitesimal action
generates vector spaces
$$ \g(x) := \{ \xi_X(x) , \xi \in \g \} \subset T_x X, \quad \g_\C(x)
= \g(x) + J_x \g(x) .$$
In the case that the action of $G$ admits an extension to an action of
the complexified Lie group $G_\C$, the space $\g_\C(x)$ is the
subspace generated by the infinitesimal complexified group action.
Each tangent space admits a splitting
\begin{equation} \label{split} 
T_x X = \g_\C(x) \oplus \g_\C(x)^\perp  \end{equation}
using the metric on $X$ induced by $J$; denote by $\pi_x: T_x X \to
T_x X$ the projection onto the second factor.  At a point $x \in
\Phinv(0)$ we have $\g_\C(x)^\perp \cong T_{p(x)} (X \qu G)$ where $p:
\Phinv(0) \to X \qu G$ is the projection.  Note that in the case of
time-dependent $J$, the space $\g_\C(x)$ as well as the splitting
depend on $t$.

\begin{definition} (c.f. \cite[Definition 4.8]{frau:thesis})  A point 
$(s,t) \in \R \times [0,1]$ is {\em $G$-regular} for a Floer
  trajectory $u: \R \times [0,1] \to X$ joining Hamiltonian arcs
  $x^\pm(t)$ if
$$ \pi_{u(s,t)} \partial_s u(s,t) \neq 0, \quad u(s,t) \notin G
  x^{\pm}(t), $$
$$ u(s,t) \notin G u(\R - \{ s \}, t), \quad G_{u(s,t)} = \{ e  \} $$
where $e \in G$ is the identity.  Denote by $R(u) \subset \R \times
[0,1]$ the set of $G$-regular points.
\end{definition} 

\begin{theorem} (c.f. \cite[Theorem 4.9]{frau:thesis})     
Let $u: \R \times [0,1] \to X$ be a Floer trajectory for
$\ti{L}_0,\ti{L}_1$ such that $\pi_{u(s,t)} \partial_s u(s,t) \neq 0$
for some $s,t \in \R \times [0,1]$.  Then there exists $ S > 0 $ such
that $ \{ (s,t) \in R(u) | s \ge S \}$ is open and dense in $\{ (s,t)
\in \R \times [0,1] \ | \ s \ge S \}$.  In particular, $R(u)$ is open
and non-empty.
\end{theorem} 

\begin{proof} 
The proof follows from a sequence of lemmas, taken almost
word-for-word from Frauenfelder \cite[Theorem 4.9]{frau:thesis}.

\begin{remark} It suffices to consider the case that the Hamiltonian
perturbation $H$ vanishes, by an observation in Floer-Hofer-Salamon
\cite[Discussion following (7)]{fhs:tr}: There is a bijection between
Floer trajectories $u(s,t)$ for a pair $(J_t,H_t)$ with Floer
trajectories $\phi_t^{-1}u(s,t)$ for $( (\phi_t^{-1})^* J_t,0)$ where
$\phi_t$ is the flow of $\hat{H}_t$.  Indeed $\partial_s (
(\phi_t^{-1}) u(s,t)) = D \phi_t^{-1} \partial_s u(s,t) $ lies in $ D
\phi_t^{-1} \g_\C(u(s,t)))$ iff $\partial_s u(s,t)$ lies in
$\g_\C(u(s,t))$.  On the other hand, since $\phi_t$ is
$G$-equivariant, $ D \phi_t^{-1} \g_\C(\phi_t(x)) = D\phi_t^{-1}
\g(u(s,t)) + D \phi_t^{-1} J_t \g(u(s,t)) = \g(\phi_t^{-1}(u(s,t))) +
(D \phi_t^{-1})^* J_t \g(\phi_t^{-1}(u(s,t)))$.
\end{remark}

\begin{lemma}  $R(u)$ is open.
\end{lemma} 

\begin{proof} 
Suppose otherwise.  Then there exists a point $(s,t) \in R(u)$ which
can be approximated by a sequence $(s_\nu,t_\nu) \notin R(u)$.  Then
$\pi_{u(s_\nu,t_\nu)} \partial_s u(s_\nu,t_\nu) \neq 0, u(s_\nu,
t_\nu) \notin G x^{\pm}(t_\nu)$ and $G_{u(s_\nu,t_\nu)} = \{ e \}$ for
$\nu$ sufficiently large.  Since $(s_\nu,t_\nu) \notin R(u)$ there
exists a sequence $s_\nu' \in \R$ and $g(s'_\nu,t_\nu) \in G$ such that
$$ u(s_\nu, t_\nu) = g(s'_\nu, t_\nu) u(s_\nu',t_\nu), \quad s'_\nu
\neq s_\nu .$$
If the sequence $s_\nu'$ is unbounded then, passing to a subsequence
if necessary, we may assume $s'_\nu \to \pm \infty$.  Then by
exponential decay, $u(s_\nu',t_\nu) \to x^\pm$ hence
$g(s'_\nu,t_\nu)$ converges to some $h \in G$.  Hence $u(s,t) =
hx^\pm$ which contradicts the fact that $(s,t) \in R(u)$.  Hence the
sequence $s'_\nu$ is bounded and we may assume without loss of
generality that $s'_\nu \to s'$.  It follows that there exists $g \in
G$ such that $u(s,t) = gu(s',t)$.  Since $(s,t) \in R(u)$ we must have
$s' = s$.  Hence $s'_\nu$ and $s_\nu$ both converge to $s$ and this
contradicts the fact that $\pi_{u(s,t)} \partial_s u(s,t) \neq 0$.
This proves that the set $R(u)$ is open.
\end{proof} 

Next we show that $R(u)$ is not empty, and in fact $R(u) \cap \{ |s|
\ge S \}$ is dense in $\{ | s | \ge S \}$.  It follows from the assumption that $\pi_{u(s,t)}
\partial_s u(s,t)$ is somewhere non-zero that there exists a non-empty
open subset
$$ \Sigma := \{ (s,t) \in \R
\times [0,1], \ \pi_{u(s,t)}(\partial_s u(s,t)) \neq 0, u(s,t)
\notin Gx^{\pm}, G_{u(s,t)} = \{ e \} \}.$$
It remains to show that there exists $(s,t) \in \Sigma$ such that
$u(s,t) \notin G u(\R - \{s \},t )$.  To do that we show that the set
of $G$-regular points in $\Sigma$ is dense in $\Sigma$.  Assume
otherwise so that there exists $(s_0,t_0) \in \Sigma$ with $B_\eps(s_0,t_0) \cap
R(u) = \emptyset$ for some $\eps > 0$.

\begin{lemma} \label{sosmall} There exists $\eps > 0$ so small and $S > 0$ so large so
that the following holds:
\begin{enumerate} 
\item \label{first} $u(s,t) \notin Gu(B_\eps(s_0,t_0))$ for $|s| \ge
  S$ and $|t - t_0 | \leq \eps$; and
\item \label{second} $\pi_{u(s,t)}(\partial_s u(s,t)) \neq 0$ for
  every $(s,t) \in B_\eps(s_0,t_0)$, the map $u: B_\eps(s_0,t_0) \to
  X$ is an embedding and $u( \ol{B}_\eps(s_0,t_0 ) )\cap g u(
  \ol{B}_\eps(s_0,t_0 )) = \emptyset $ for every $g \in G \backslash
  \{ e \}$.
\end{enumerate} 
\end{lemma} 

\begin{proof} 
The condition $u(s,t) \notin Gu(B_\eps(s_0,t_0))$ for $|s| \ge S$ and
$|t - t_0 | \leq \eps$ is achievable since $u(s_0,t_0) \notin G
x^\pm$.  The condition $\pi_{u(s,t)}(\partial_s u(s,t)) \neq 0$ is
achievable since it is open and $\pi_{u(s_0,t_0)}(\partial_s
u(s_0,t_0)) \neq 0$.  The latter condition implies that $\d u$ is
injective and hence $u$ is an embedding near $(s_0,t_0)$.  Finally the
condition $\pi_{u(s,t)}(\partial_s u(s,t)) \neq 0$ implies that $u(
\cdot, t)$ is transverse to the $G$-orbits and so $u( \ol{B}_\eps(s_0,
t_0) ) \cap g u( \ol{B}_\eps(s_0, t_0 )) = \emptyset $ for every $g \in
G \backslash \{ e \}$ for $\eps$ sufficiently small.
\end{proof} 

The condition $B_\eps(s_0,t_0) \cap R(u) = \emptyset$ means that for
all $(s,t) \in B_\eps(s_0,t_0)$ there exists an $s' \in \R$ such that
$u(s,t) \in Gu(s',t)$ and $s' \neq s$.  

\begin{lemma}  There exists a point $(\ti{s}_0,\ti{t}_0) \in
B_\eps(s_0,t_0)$ such that the set
$$ C:= \{ s \in \R: u(s, \ti{t}_0) \in Gu (\ti{s}_0, \ti{t}_0) \} $$
is finite and for every $s \in C$ we have $\rank(\d u (s, \ti{t}_0)) =
2$ and $ \pi_{u(s,\ti{t}_0)} (\partial_s u ( s, \ti{t}_0)) \neq 0 $.
\end{lemma} 

\begin{proof}  Choose $T \subset M$ a $G$-equivariant tubular
neighborhood of $u(s_0,t_0)$ such that $G$ acts freely on $T$.  Define
the function 
$$\ti{u}: u^{-1}(T) \to N:= T/G, \quad (s,t) \mapsto G u(s,t) .$$
Since $\ti{u}$ is an immersion at $(s_0,t_0)$, by shrinking $\eps$ we
may assume without loss of generality that $\ti{u}(B_\eps(s_0,t_0))
\subset N$ and there exists an open set $\ti{u}(B_\eps(s_0,t_0))
\subset V \subset N$ and a chart $\phi: V \to \R^m$ such that
$$ \phi(\ti{u}(B_\eps(s_0,t_0))) \subset \R^2 \times \{ 0 \} \subset
\R^2 \times \R^{m-2} \cong \R^m .$$
Here $m:= \dim(X) - \dim(G)$ is the dimension of $N$.  It follows
from \eqref{first} of Lemma \ref{sosmall} that we can assume without
loss of generality that there exists a compact set $K$ such that
$\ti{u}^{-1}(V) \subset K \subset \R \times [0,1]$.  Abbreviate
$$ A:= \ti{u}(B_\eps(s_0,t_0)) = \ti{u}(u^{-1}(u(B_\eps(s_0,t_0)))).$$
Let $\rho: \R^m \to \R^m$ be the linear projection on $\R^2 \times \{
0 \}^{m-2} \subset \R^m$.  Define
$$ \ti{V} := \phi^{-1} 
 \circ \rho
|_{\phi(V)}^{-1} (\phi(A)) \subset V .$$
Note that $\ti{V}$ and hence $B := \ti{u}^{-1}(\ti{V})$ are open.  Moreover, 
$ B \subset K \subset \R \times [0,1] .$
We define further a map 
$$ v: B \to A, \quad v: = (\phi |_{B_\eps(s_0,t_0)}) ^{-1}
 \circ \rho \circ \phi \circ \ti{u} .$$
Observe that $u^{-1} (u(B_\eps(s_0,t_0))) \subset B$. Also, the maps $
v$ and $\ti{u}$ have the same restriction to
${u^{-1}(u(B_\eps(s_0,t_0)))}$.  Define the set
$$ C(\ti{u}) := \{ z \in B_\eps(s_0,t_0) : \# \{ \ti{z} \in B: \ti{u}(z) =
\ti{u}(\ti{z}) \} = \infty \} \subset B_\eps(s_0,t_0) .$$
Because $B$ is contained in the compact set $K$ we may write 
$$ C(\ti{u}) = \{ z \in B_\eps(s_0,t_0) : \exists \{ z_\nu \}_{\nu =
  1}^\infty, z_\nu \neq z, z_\nu \to \ti{z} \in B, \ti{u}(z_\nu) =
\ti{u}(z) = \ti{u}(\ti{z}) \} .$$
It follows that $\ti{u}(C(\ti{u}))$ is contained in the set of
critical points of $\ti{u}$.  Now the formula 
$$ \d v = (\d \phi)^{-1} \circ  \d \rho \circ \d \phi \circ \d \ti{u} $$
implies 
$ \ti{u} (C(\ti{u})) \subset C_v(B) $
where $C_v(B)$ denotes the set of critical values of $v$ in $A$.  By
Sard's theorem, the set
$  (A \backslash \ti{u}(C(\ti{u})) \supset (A \backslash C_v(B)) $
is dense in $A$ and in particular nonempty.  Choose $q \in A
\backslash \ti{u}(C(\ti{u}))$ and define
$$ (\ti{s}_0,\ti{t}_0) := \ti{u}^{-1}(q) \cap B_\eps(s_0,t_0) .$$
Then $(\ti{s}_0,\ti{t}_0)$ has the required properties.  This proves
the Lemma. \end{proof}
        
We investigate the failure of $G$-regularity  in more detail.  Fix a
point $(s_0,t_0) = (\ti{s}_0,\ti{t}_0)$ given by the Lemma.  Let
$s_1,\ldots,s_N \in [-S,S]$ be the points with
$ \ti{u}(s_0,t_0) = \ti{u}(s_1,t_0) = \ldots = \ti{u}(s_N, t_0) .$
Let $F_\delta$ be the set of domain values which fail to be
$G$-injective with one point near $(s_0,t_0)$:
$$ F_\delta := \{ (s',t) \in \R \times [0,1] \ | \ \exists (s,t) \in B_{2\delta}(s_0,t_0) 
\ | \ \ti{u}(s,t) = \ti{u}(s',t) \} \backslash B_{2\delta(s_0,t_0)} .$$

\begin{lemma}\label{rd} For every constant $r > 0 $ there exists a $\delta > 0$ such that 
$ F_\delta \subset \bigcup_{j=1}^N B_r (s_j , t_0) .$
\end{lemma} 

\begin{proof}
Otherwise, there would exist $\rho > 0 $ and a sequence $(s_\nu, t_\nu) \to (s_0, t_0)$ with 
$s'_\nu \neq s_\nu$ and $\ti{u}(s_\nu,t_\nu) = \ti{u}(s_\nu',t_\nu)$ such that 
$ (s'_\nu, t_\nu) \notin B_\rho(s_j,t_0)  $
for every $j \ge 1$.  By \eqref{second} of Lemma \ref{sosmall}, there
exists $\eps' > 0$ such that $|s_\nu - s'_\nu| > \eps'$.  By
\eqref{first} of Lemma \ref{sosmall}, we have $|s_\nu'| \leq S$.
Hence the sequence $s'_\nu$ has an accumulation point $s'$ which must
be distinct from all the points $s_0,\ldots,s_N$ but satisfies
$\ti{u}(s',t_0) = \ti{u}(s_0,t_0)$.  This contradiction proves the
claim in the Lemma. \end{proof} 

Locally the set of domain values can be partitioned according to
failure of $G$-regularity  as follows.  Fix an $r > 0$ and choose
$\delta$ as in Lemma \ref{rd}.  Let $\Sigma_j$ be the set of domain values for
which the value at a point near $(s_0,t_0)$ is repeated near
$(s_j,t_0)$, that is,
$$ \Sigma_j := \{ (s,t) \in \ol{B}_\delta(s_0,t_0) \ | \ 
\exists (s',t) \in \ol{B}_r(s_j,t_0), \ \ti{u}(s',t) = \ti{u}(s,t) \} $$
for $j = 1,\dots, N$.  The sets $\Sigma_j$ are closed and 
$ \ol{B}_\delta(s_0,t_0) = \Sigma_1 \cup \ldots \cup \Sigma_N .$
Hence at least one of the sets $\Sigma_j$ has a nonempty interior.  Assume
without loss of generality that $\on{int}(\Sigma_1) \neq \emptyset$.  Choose an open
set $U \subset \Sigma_1$ and note that 
$ U \cap B_{2r}(s_1,t_0) = \emptyset .$ Furthermore, $u:
B_{2r}(s_1,t_0) \to X$ is an embedding for $r$ sufficiently small.
Since $u(\cdot, t)$ is transverse to the $G$-orbits at $s_1$, we have
$$u(\ol{B}_{2r}(s_1,t)) \cap gu( \ol{B}_{2r}(s_1,t) ) = \emptyset,
\ \quad \forall g \in G \backslash \{ e \}$$
provided $r > 0$ was chosen sufficiently small.  On the other hand it
follows from the definition of $\Sigma_1$ that for every $(s,t) \in U$
there exists an $s' \in \R$ such that $(s',t) \in B_{2r}(s_1,t_0)$ and
$\ti{u}(s,t) = \ti{u}(s',t)$.

We show that there exist two regions of the domain on which the map
is, up to gauge transformation, related by a diffeomorphism.  Fix a
point $(\hat{s}_0,\hat{t}_0) \in U$.  Let $\sigma := \hat{s}_0' -
\hat{s}_0 \neq 0$.  Define
$$ w: \ti{W}:= B_{2r}(s_1 - \sigma,t_0) \to X, \quad (s,t)
\mapsto u(s + \sigma,t)  $$
and similarly $\ti{w} = \ti{u}(u + \sigma,t)$.  Note that $\ti{w}:
\ti{W} \to N$ is an embedding for $r$ sufficiently small.  Moreover,
$
\ti{w}(\hat{s}_0, \hat{t}_0) = \ti{u}(\hat{s}_0, \hat{t}_0)$ and $\ti{u}(U) \subset \ti{w}(\ti{W}) .$ 
Define
$$ W:= \ti{w}^{-1}(\ti{u}(U)) .$$
Then $\ti{u}^{-1} \circ \ti{w}: W \to U$ is a diffeomorphism.
Moreover, our assumptions assert that this map takes the form
$\ti{u}^{-1} \circ \ti{w}(s,t)  =: (\kappa(s,t),t)$.
Differentiating the formula $ \ti{w}(s,t) = \ti{u}(\kappa(s,t),t)$ we
obtain
\begin{eqnarray*}  0 &=& \pi_{w(s,t)} \partial_s w(s,t) + J(t,w(s,t)) \pi_{w(s,t)} 
\partial_t w(s,t) \\
&=& \pi_{u(\kappa,t)}  ( \partial_\kappa u(\kappa,t) \partial_s \kappa )  + J(t, u(\kappa,t))
  \pi_{u(\kappa,t)}  ( \partial_\kappa u(\kappa,t) \partial_t \kappa + \partial_t
  u(\kappa,t)) \\
 &=& 
\pi_{u(\kappa,t)}  \partial_\kappa u(\kappa,t) \partial_s \kappa +
  \pi_{u(\kappa,t)} (\partial_t u(\kappa,t) \partial_t \kappa - \partial_\kappa
  u(\kappa,t)) \\ &=& \pi_{u(\kappa,t)}  (\partial_\kappa u(\kappa,t) (\partial_s \kappa
  - 1)) + \pi_{u(\kappa,t)}  \partial_t u(\kappa,t) \partial_t \kappa.
\end{eqnarray*} 
Since $\pi_{u(\kappa,t)} \partial_\kappa u(\kappa,t)$ and
$\pi_{u(\kappa,t)} \partial_t u (\kappa,t)$ are linearly independent
we deduce that $\partial_s \kappa = 1$ and $\partial_t \kappa = 0$.
Hence $\kappa(s,t) = s + \hat{\sigma}$ for some $\hat{\sigma} \in \R$.
Since $\ti{u}(\hat{s}_0, \hat{t}_0) = \ti{w}(\hat{s}_0,\hat{t}_0)$, we
obtain $\hat{\sigma} = 0$ and hence $\kappa(s,t) = s$. This implies
that $\ti{w}$ and $\ti{u}$ agree in the neighborhood $U = W$ of
$(\hat{s}_0, \hat{t}_0)$.  So $w(s,t) = g(s,t) u(s,t)$ for some $g: U
\to G$.  By holomorphicity $J(t,u(s,t)) (\partial_t g)_X (u(s,t)) =
(\partial_s g)_X (u(s,t)) $, which is impossible unless $g$ is
constant.  By unique continuation, $g u(s,t) = u(s + \sigma,t)$ for
any $s,t$.  By induction $ g^ku(s,t)= u(s + k\sigma,t)$ for any $k \ge
0$.  Without loss of generality $\sigma > 0$, and since $u(s +
k\sigma,t)$ converges for $s \to \infty$ to $x^+$ this implies that
$u$ is constant in $s$.  This contradiction completes the proof of the
Theorem. \end{proof}

It remains to show that we may assume that any Floer trajectory has somewhere 
non-trivial differential modulo the complexified infinitesimal group action.  For
this we use the following modification of a result of Xu \cite{xu:gf},
which says that for a suitable lift of a Hamiltonian for the
symplectic quotient, any Floer trajectory has this property.
It suffices to consider the case that the Hamiltonian perturbation
on $X \qu G$ vanishes.  

\begin{theorem}   Suppose that $\dim(X \qu G) > 0$, 
$L_0,L_1 \subset X \qu G$ are transversally-intersecting Lagrangians
  and $(J_t) \in \J(X)^G$ is an invariant time-dependent almost
  complex structure.  There exists $H \in C^\infty([0,1] \times X)^G$
  equal to $0$ on $\Phinv(0)$ such that for every Floer trajectory $u:
  \R \times [0,1] \to X$ for $H$, there exists a point $(s,t) \in \R
  \times [0,1]$ with $\pi_{u(s,t)} \d u(s,t) \neq 0$.
\end{theorem} 

\begin{proof} 
Let $H \in C^\infty([0,1] \times X)^G$. Recall that any Floer
trajectory may be considered a gradient trajectory for the action
functional on the universal cover $\ti{P}(L_0,L_1)$ of the space of
paths $P(L_0,L_1)$ from $L_0$ to $L_1$: for $\ti{x} \in
\ti{P}(L_0,L_1)$ covering $x$,
$$ \A_H(\ti{x}) = - \int_{[0,1]^2} w^* \omega - \int_{[0,1]} x^* H \d
\t
$$
where $w: [0,1]^2 \to X$ is a homotopy of $x$ to a base point
determined by the choice of lift $\ti{x}$.  The Hessian of the action
functional is independent of the choice of lift and corresponds to the
linear map
\begin{equation} \label{hess} Q: \Omega^0(x^* TX) \mapsto \Omega^0(x^* TX), \quad \xi \mapsto J_t
(\nabla_t \xi - \nabla_\xi \hat{H}_t) .\end{equation}
According to a result of Robbin-Salamon \cite[Theorem B]{rs:asym}, any
Floer trajectory $u(s,t)$ is asymptotic to $\exp( - \lambda s) \xi(t)$
where $\xi$ is an eigenvector of the Hessian \eqref{hess} with
non-zero eigenvalue $\lambda$, in the sense that
$ \xi(t) = \lim_{s \to \infty} e^{\lambda s} u(s,t) .$
Furthermore there exist constants $s_0,c > 0$ such that $ (1/c)
e^{-\lambda s} \leq | \partial_s u | \leq ce^{- \lambda s}$ for all
$s$ such that $|s| \ge s_0$.  

For a carefully chosen Hamiltonian all of the eigenvectors have the
desired property.  Let $Gx_1,\ldots, Gx_s = \phi_{1,H} \ti{L}_0 \cap
\ti{L}_1 $ be the finite set of orbits lifting the intersection points
$\ol{x}_1,\ldots, \ol{x}_s$ of $\phi_{H,1} L_0 \cap L_1$.  The map
\begin{equation} \label{arise} \g \times ( T_{x_i(0)} (X \qu G) \times \g) \to X, \quad (\zeta,
v,\theta) \mapsto \exp(\zeta) \exp_{x_i} \left(v + J_t(x_i)
\theta_X(x_i) \right) \end{equation}
 a local diffeomorphism onto its image on a neighborhood of $0$.  We
 let $v_1,\ldots,v_k$ resp. $\theta_1,\ldots,\theta_d$
 resp. $\zeta_1,\ldots,\zeta_d$ be the coordinates near $x_i$ arising
 from \eqref{arise} with respect to orthonormal bases; these
 coordinates are invariant under the infinitesimal action of $\g$ and
 so extend to $G$-invariant functions in a neighborhood of $Gx_i$.
Choose cutoff-functions
$\rho_1,\ldots,\rho_s \in C^\infty(X)^G$ equal to one on a
neighborhood of $Gx_1,\ldots,Gx_s$ and with disjoint support.  Let
$\exp_{x_i}: T_{x_i} X \to X$ denote geodesic exponentiation. 

We first deal with the case that $d := \dim(G) \leq \dim(X \qu G)
=:2k$.  Define an invariant time-dependent function $F_i $ in a
neighborhood of $x_i$ by
$$ F_i(v,\zeta,\theta) = \rho_i \sum_{l=1}^d \theta_l v_l .$$
Take the Hamiltonian perturbation to be the sum of the functions
above, $ H = \sum_{i=1}^s F_i .$ We claim that the Hessian for $H$ in
\eqref{hess} has no eigenvectors which map to zero under the
projection to the tangent space to the symplectic quotient.  Indeed
suppose, by way of contradiction, there exists an eigenvector $\xi(t)$
with $\pi_{x_i} \xi(t) = 0$ for all $t \in [0,1]$ with eigenvalue
$\lambda$.  Then $\xi$ has the form $(0, \zeta(t),\theta(t))$ for some
functions $\zeta: [0,1] \to \g, \theta: [0,1] \to \g$.  The
eigenvector equation
$ J_t ( \nabla_t \xi - \nabla_\xi \hat{F_i} ) = \lambda \xi $
gives 
\begin{equation} \label{writing}
 J_t \left( {\zeta'(t)_X} + J_t {\theta'(t)_X} + 
[J_t {\theta(t)_X}, -
  \hat{F_i}] \right)
= \lambda {\zeta(t)_X}  + \lambda J_t {\theta(t)_X} .\end{equation} 
Discounting terms that lie in $\g_\C(x_i)$ one obtains that $[J_t
  {\theta(t)_X}, - \hat{F_i}]$ vanishes.  This implies that
$\theta(t)$ vanishes as well, by the explicit choice of $F_i$ above.
So
$ J_t {\zeta'(t)_X} = \lambda {\zeta(t)_X} .$
Since $\lambda\neq 0$, ${\zeta(t)_X}$ vanishes.  Since the action is
locally free near $x_i$, $\zeta$ vanishes as well.  This shows that
there do not exist eigenvectors with values contained in $\g_\C(x_i)$
for all $t \in [0,1]$.

Now we deal with the case $d > 2k$, in which we must choose the
Hamiltonian in a more complicated way.  Given a family of matrices
$a_{jl}(t)$ for $1 \leq l \leq d, 1 \leq j \leq 2k$, define using the
coordinates $v,\theta,\zeta$ above a function $F_i \in C^\infty([0,1] \times X)^G$
given by
$$ F_i(t,v,\theta,\zeta) = \rho_i \sum_{1 \leq l \leq d} \sum_{ 1 \leq j \leq
  2k} a_{jl}(t) \theta_l v_j .$$
Consider the map
\begin{equation} \label{imageof}
Q_{i,t}: \g \to T_{x_i}X, \quad \zeta \mapsto J_t[J_t \zeta_X,
  \hat{F}_i](x_i) .
\end{equation}
Its image satisfies 
\begin{equation} \label{image} \on{Im}(Q_{i,t}) \subset T_{\ol{x}_i} (X \qu G) \subset T_{x_i} X
\end{equation} 
in the splitting \eqref{split}, where $\ol{x}_i$ is the image of $x_i$
in $X \qu G$.  The matrix of $Q_{i,t}$ with respect to the given bases
of $\g, T_{\ol{x}_i}(X \qu G)$ is $a_{jl}(t)$.  Choose the matrices
$(a_{jl}(t))$ so that the kernel of \eqref{imageof} is spanned by
$\eps_l + \cos(2l \pi t) \eps_d, \ l = 1,\ldots, d - 2k .$
Suppose that $\xi (t) = {\zeta(t)_X} + J_t {\theta(t)_X}$ is an
eigenvector for the Hessian with non-zero eigenvalue $\lambda$.  Then
by \eqref{writing} and \eqref{image}
$ \theta(t) \in \on{span} \{ \eps_l(t) \ | \ l = 1,\ldots, d - 2k \}
.$
Hence there exist functions $\theta_l$ such that 
\begin{equation}\label{hence}
 \theta(t) = \sum_{l=1}^{d - 2k} \theta_l(t) (\eps_l + \cos(2 \pi l
t) \eps_d) .\end{equation}
The eigenvalue equation \eqref{writing} gives
$ \zeta'(t) = \lambda \theta(t), \  \theta'(t) = - \lambda \zeta(t) . $
Together with the Lagrangian boundary conditions this implies
$ \theta''(t) = - \lambda^2 \theta(t), \  \theta(0) = \theta(1) =
0 .$
Hence there exist $\alpha \in \g$ such that $\theta(t) = \alpha \sin
(2 \pi \lambda t) .$ By pairing \eqref{hence}
with $\eps_l$ we obtain 
$ \theta_l(t) = \lan \alpha, \eps_l \ran \sin(2 \pi \lambda  t), \ 
l = 1,\ldots, d - 2k .$
Pairing \eqref{hence} with $\eps_d$ gives
\begin{equation} \label{compare} \lan \alpha,\eps_d \ran \sin(2 \pi \lambda  t) 
= \sum_{l=1}^{d - 2k} 
 \lan \alpha, \eps_l \ran \sin(2 \pi \lambda  t) \cos(2 \pi l t)
.\end{equation}
Comparing Fourier coefficients in \eqref{compare} implies that all
coefficients $\lan \alpha, \eps_l \ran$ vanish.  It follows that
$\theta(t)$ and also $\zeta(t)$ vanish.
\end{proof}

We now show that for a generic choice of almost complex structure, the
moduli spaces of Floer trajectories are smooth manifolds.  The proof
is based on the Sard-Smale theorem.  We denote by
$$D_u: \Omega^0(\R \times [0,1], u^*TX; u_{\R \times \{ 0, 1\}}^* TL_0
\sqcup TL_1) \to \Omega^{0,1}(\R \times [0,1],u^* TX)$$
the linearized Cauchy-Riemann operator associated to the almost
complex structure $J$, acting on sections of the pull-back bundle
$u^*TX$ with boundary conditions in $TL_0,TL_1$.

\begin{theorem}    There exists a comeager subset $\J^{*} \subset
\J$ consisting of those $J \in \J$ such that $D_u$ is surjective for
every Floer trajectory $u$ with $\pi_{u(s,t)} \partial_s u$ not
identically zero.
\end{theorem}

\begin{proof}  Let $\cB = \Map_{k,p,\alpha}(\R \times [0,1],X,L_0,L_1)$ 
denote maps locally of Sobolev class $k,p$, for some $k \ge 1, p > 2$,
completed with respect to a metric with exponential decay with weight
$\alpha$ with boundary in $L_0,L_1$.  Let $\cE$ denote the vector
bundle over $\cB$ with fiber $\cE_u = T^{*,0,1} (\R \times [0,1])
\otimes u^* TX $ and
$$ \olp: \cB \to \cE , \quad u \mapsto \olp_J u  $$
the section given by the Cauchy-Riemann operator.  Let $\J^\ell$
denote the space of invariant compatible almost complex structures of
class $C^\ell$ for $\ell \ge 1$ and consider the space
$$ \widetilde{\M\W}^{\univ,\ell}(L_0,L_1) := \{ (u,J) \in \cB \times
\J^\ell : \olp_J u = 0 \} .$$
The universal moduli space
$$ \M\W^{\univ,\ell}(L_0,L_1) :=
\widetilde{\M\W}^{\univ,\ell}(L_0,L_1)/G $$
is a separable Banach manifold of class $C^l$.  To see this, one
verifies that the map
$$ D_{u,J}: T_u \cB \times T_J \J^\ell \to \cE_u, \quad 
(\zeta,Y) \mapsto  D_u(\zeta) + Y(u) \partial_t u  $$
is surjective.  Since $D_u$ is a Fredholm operator, it has closed
range and finite dimensional kernel.  Hence $D_{u,J}$ has closed range
and finite dimensional cokernel and it only remains to prove that its
range is dense.  To see this, let $\eta \in W^{-k+1,q}_{-\alpha}(\R
\times [0,1],u^*TX)$ for $1/p + 1/q = 1$ such that $\eta$ vanishes on
the range of $D_{u,J}$, that is,
\begin{equation} \label{adj}
\int_{\R \times [0,1]} \langle \eta, D_u \zeta \rangle \d s \d t = 0 ,
\quad \int_{\R \times [0,1]} \langle \eta, Y(u) \partial_t u \rangle
\d s \d t = 0 \end{equation}
for every $\zeta \in T_u \cB$ and for every $Y \in T_J \J^\ell$.
Equation \eqref{adj} says that $D_u^* \eta = 0$ where $D_u^*$ is the
formal adjoint of $D_u$.  An argument using elliptic regularity shows
that $\eta \in W_{loc}^{\ell+1,p}(\R \times [0,1],u^* TX)$. Let $(s,t)
\in \R \times [0,1]$ be a $G$-regular point for $u$.  We claim that
$\eta(s,t) = 0$.  Otherwise, one could construct as in \cite{fhs:tr} a
$Y \in T_J \J^\ell$ with small support around $(t, u(s,t))$ such that
$$ \int_{\R \times [0,1]}  \lan \eta, Y(u) \partial_t u \rangle \d s \d t > 0 $$
contradicting the second equation in \eqref{adj}.  This implies that
$\eta$ vanishes at every $G$-regular point.  Unique continuation for
the first order elliptic operator $D_u^*$ implies that $\zeta = 0$.
Hence $D_{u,J}$ is onto for every $(u,J) \in \cB \times \J^\ell$.  It
follows from the implicit function theorem that
$\M\W^{\univ,\ell}(L_0,L_1)$ is a Banach manifold of class $C^q$ for
any $q > \ell - k$.  The projection
$ p^\ell : \M\W^{\univ,\ell}(L_0,L_1) \to \J^\ell $
is a Fredholm map since the kernel and cokernel of $Dp^\ell$ at
$(u,J)$ may be identified with that of $D_u$.  By the Sard-Smale
theorem, after restricting to maps of bounded index, for $\ell$
sufficiently large the set $\J^{\ell,*}$ of regular values of $p^\ell$
is a countable intersection of open and dense sets in $\J^\ell$.  Note
that $J \in \J^\ell$ is a regular value of $p^\ell$ exactly if $D_u$
is surjective for every $u \in \cF_J^{-1}(0)$.  A standard Taubes
argument (see \cite[Chapter 3]{ms:jh}) implies that the space of
smooth almost complex structures such that the moduli space is regular
is also comeager.
\end{proof} 

\subsection{Compactification} 

In general, a sequence of quasistrips can have bubbling or breaking
behavior, namely disk bubbling with boundary values in
$\ti{L}_0,\ti{L}_1$, sphere bubbling, and breaking of Floer
trajectories.  Here we restrict to the case that $(X,\omega)$ is {\em
  aspherical}, that is, $\int_{S^2} u^* \omega= 0$ for any $u: S^2 \to
X$, which rules out sphere bubbling.

Recall that a {\em nodal disk} $\Sigma$ is a contractible space
obtained from a collection $D_1,\ldots, D_k$ of disks by identifying
points in a distinct set $ \{ \{w_1^+,w_1^- \} ,\ldots, \{ w_m^+,w_m^-
\} \}$ of {\em nodes} on the boundary.  An {\em isomorphism} of nodal
disks $\Sigma, \Sigma'$ is a homeomorphism $\Sigma \to \Sigma'$
holomorphic on each disk component.  A {\em nodal strip} is a nodal
disk $\Sigma$ with $2$ markings $z_+,z_-$ on the boundary, distinct
from each other and the nodal points. For each component between the
markings $z_-$ and $z_+$, the complement of the markings and the nodes
is biholomorphic to a strip $\R \times [0,1]$ uniquely up to
translation.  We call these components the {\em strip components} and
the other components {\em disk bubbles}.  See Figure \ref{nodalstrip},
where the vertical dotted line represents a node connecting strip
components.  The {\em combinatorial type} of a nodal strip is the
ribbon tree $\Gamma$ with two semi-infinite edges corresponding to
$z_\pm$, obtained by replacing each disk or strip component with a
vertex, each node with a finite edge, and each marking with a
semiinfinite edge, with the ribbon structure (cyclic order of edges at
each vertex) given by the ordering of the nodes and markings around
the boundary of each strip or disk component.

\begin{figure}[ht]
\includegraphics[height=1.5in]{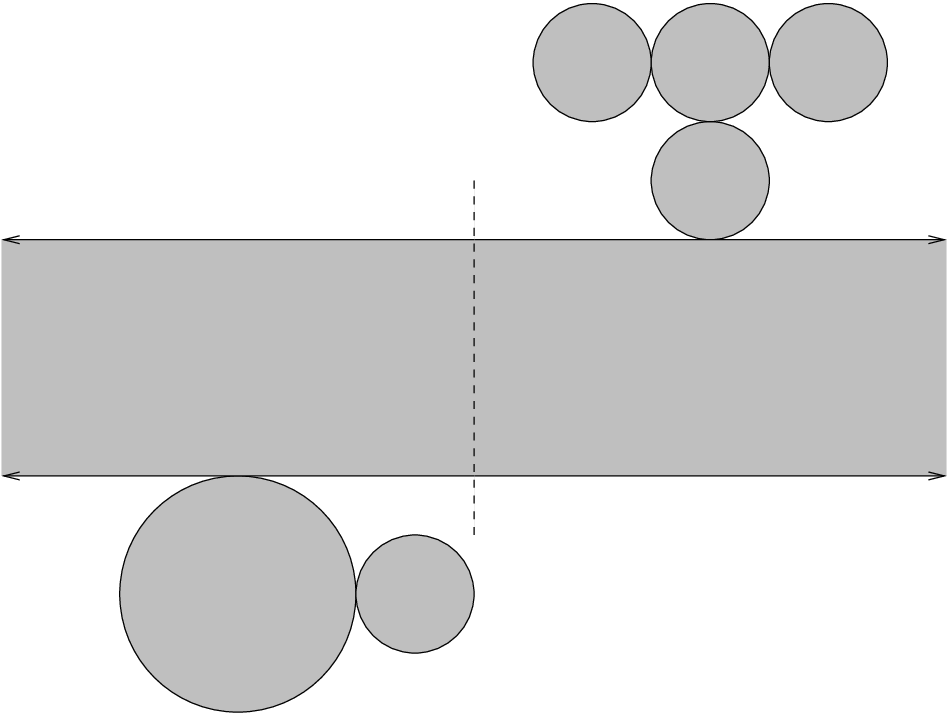}
\caption{A nodal strip}
\label{nodalstrip}
\end{figure} 

For any nodal strip $\Sigma$ let $\ti{\Sigma}$ denote the disjoint
union of the components, so that $\Sigma$ is obtained by identifying
pairs of points $w_j^\pm \in \ti{\Sigma}, j = 1,\ldots, m$.  The nodes
are of three types: they either connect strip components, disk
components, or a disk component to a strip component.  If $\Gamma$ is
the combinatorial type of $\Sigma$, let $\Def_\Gamma(\Sigma)$ denote
the direct sum of the tangent spaces $T_{w_j^\pm} \partial \Sigma$
over nodes that do {\em not} connect strip components.  Each element
$\zeta$ of $\Def_\Gamma(\Sigma)$ represents a deformation
$w^\pm_j(\zeta)$ of the attaching points used to construct $\Sigma$,
and so gives rise to a deformed nodal strip $\Sigma^\zeta$ for
sufficiently small $\zeta$.  A {\em collection of gluing parameters}
is an $m$-tuple $\delta = (\delta_1,\ldots, \delta_m)$ of non-negative
real numbers.  For any collection $\delta$, we denote by
$\Sigma^\delta$ the nodal strip obtained by removing half-disks in a
neighborhood of attaching nodes and gluing together via the map
$z_{+,j} \sim \delta_j/z_{-,j}, j = 1,\ldots, m$ where $z_{\pm,j}$ are
the coordinates on the half-disks near $w_j^\pm$; if the disks
represent strip components then this means that the strips are glued
together using a neck of length $-\log(\delta_j)$, each of which we
identify with a strip $[-2S_j,2S_j] \times [0,1]$.  The length $S_j$
is determined by radii of the small half-balls $\eps_j^+,\eps_j^-$
used in gluing and the gluing parameter $\delta_j$ by the relation $
4S_j = \log(\eps_j^- \eps_j^+/\delta_j)$.  The {\em deformation space}
$\Def(\Sigma)$ is the sum of the deformations $\Def_\Gamma(\Sigma)$
preserving the combinatorial type and the gluing parameters
$$ \Def(\Sigma) := \Def_\Gamma(\Sigma) \times [0,\infty)^m $$
with the last factor representing the deformation parameters.  For any
sufficiently small $(\zeta,\delta) \in \Def(\Sigma)$ we denote by
$\Sigma^{\zeta,\delta}$ the nodal strip obtained by applying the
gluing construction with parameters $\delta$ to the deformed nodal
strip $\Sigma^\zeta$.

\begin{definition}  A {\em nodal $(J,H)$-holomorphic quasistrip} 
with Lagrangian boundary conditions $L_0,L_1$ consists of a nodal
strip $\Sigma$ together with a continuous map $u : \Sigma \to X$ with
boundary in $\ti{L}_0$ resp.  $\ti{L}_1$ such that $u$ is
$(J,H)$-holomorphic with finite energy on each strip component and
$J$-holomorphic on any other disk component.  An {\em isomorphism} of
nodal holomorphic quasistrips $u_j : \Sigma_j \to X$ is an isomorphism
$\phi: \Sigma_0 \to \Sigma_1$ and an element $g \in G$ such that $gu_0
= \phi^* u_1$.  A nodal holomorphic quasistrip is {\em stable} if each
constant strip component has at least one nodal point on the boundary,
and each constant disk component has at least three nodal points on
the boundary, or equivalently, there are no non-trivial automorphisms.
\end{definition} 

The {\em energy} of a nodal holomorphic quasistrip is the sum of the
energies of the strip components and the energies of the holomorphic
disk components. Let $\ol{M}(L_0,L_1;H)$ denote the moduli space of
isomorphism classes of stable holomorphic quasistrips of finite
energy.  There is a natural notion of {\em Gromov convergence} of a
sequence of stable holomorphic quasistrips, generalizing the usual
notion for holomorphic maps but incorporating the action of $G$, which
we will not spell out.  

One possible set-up which guarantees compactness not only of spaces of
holomorphic maps with bounded energy but also of symplectic vortices
is that of Cieliebak-Mundet-Gaio-Salamon \cite{ci:symvortex} or
Frauenfelder \cite{frau:thesis}:

\begin{definition}  Let $X$ be a Hamiltonian $G$-manifold. 
A {\em convex structure} on $X$ is a pair $(f,J)$ where $J \in
\J(X)^G$ and $f \in C^\infty(X)^G$ satisfies
\begin{equation} \label{convex}  \langle \nabla_\xi \nabla f(x), \xi \rangle 
\ge 0, \quad \d f(x) J_x \Phi(x)^\# (x) \ge 0 \end{equation}
for every $x \in X$ and $\xi \in T_x X$ outside of a compact subset of $X$.
\end{definition} 

For example, if $X$ is a Hermitian vector space and $G$ is a torus
acting with proper moment map then the standard complex structure $J$
together with the function $f(z) = |z|^2/2$ defines a convex
structure.  In our situation, it suffices for the first equation in
\eqref{convex} to hold, by a standard argument involving the maximum
principle.

\begin{theorem} \label{trajcompact} 
If $X$ is compact or convex and $\ti{L}_0, \ti{L}_1 \subset X$ are
compact Lagrangian submanifolds with clean intersection, then any
sequence of stable holomorphic quasistrips with bounded energy has a
convergent subsequence.
\end{theorem} 

\begin{proof}  This is a version of Gromov compactness
which combines arguments in McDuff-Salamon \cite[Chapter 4]{ms:jh}
(which proves energy quantization for holomorphic disks with
Lagrangian boundary conditions) with energy quantization for
holomorphic strips with clean intersection Lagrangian boundary
conditions Lemma \ref{cleanquant} and exponential decay estimates
explained above in Lemma \ref{expdecay} for Floer trajectories for
clean intersection, used to show that bubbles connect.
  \end{proof}

With a little more work as in \cite{ms:jh}, one can show that for any
$C > 0$ the subset of the moduli space $\ol{M}(L_0,L_1,H)$ with energy
at most $C$ is compact; this requires showing that in the topology
defined by Gromov convergence, convergence is equivalent to that of
Gromov convergence.  However, we will not need this result.

We denote by $M_\Gamma(L_0,L_1;H)$ the stratum of stable holomorphic
quasistrips of combinatorial type $\Gamma$.  Near the equivalence
class defined by a map $u: \Sigma \to X$ the space
$M_\Gamma(L_0,L_1;H)$ is a quotient of the zero set of the Fredholm
map on weighted Sobolev spaces
\begin{multline}
 \cF_u: \Def_\Gamma(\Sigma) \times \Omega^0(\ti{\Sigma},u^*TX,
(\partial_0 u)^* T \ti{L}_0, (\partial_1 u)^* T \ti{L}_1)_{1,p,\alpha}
\\
\to \Omega^{0,1}(\ti{\Sigma},u^* TX)_{0,p,\alpha} \oplus
\bigoplus_{j=1}^m T_{u(w^\pm_j)} I_{\delta(j)}\end{multline}
where $I_{\delta(j)}$ is either $ \ti{L}_0$, for a node attaching to
the bottom of the strip, $ \ti{L}_1$, for a node attaching to the top
of the strip, or $\ti{L}_0 \cap \ti{L}_1$, for a node connecting two
strip components, and the map $\cF_u$ is given by
$ (\zeta,\xi) \mapsto \cT_u(\xi)^{-1} \olp_{J,H} \exp_u \xi $
on the strip components,  
$ (\zeta,\xi) \mapsto \cT_u(\xi)^{-1} \olp_J \exp_u \xi $
on the disk components and the differences
\begin{equation} \label{diffs}
 (\zeta,\xi) \mapsto \left( \exp_{u(w_j^\pm)}^{-1}
  \exp_{u(w_j^+(\zeta))} (\xi( w_j^+(\zeta))) - \exp_{u(w_j^\pm)}^{-1}
  \exp_{u(w_j^-(\zeta))} (\xi( w_j^-(\zeta)))
  \right)_{j=1}^m \end{equation}
at the nodes, where the position of the nodes $w_j^\pm(\zeta)$ depends
on the deformation parameter $\zeta$.  Here weighted Sobolev spaces
with weight $\alpha$ are used on the strip components and ordinary
Sobolev spaces on the disk components.  $M_{\Gamma}(L_0,L_1;H)$ is
given near $u: \Sigma \to X$ as the zero set of $\cF_u$, quotiented by
the action of $G \times \Aut(\Sigma)$.  Define the {\em linearized
  operator}
\begin{multline}
 D_u: \Def_\Gamma(\Sigma) \times
\Omega^0(\ti{\Sigma},u^*TX, 
(\partial_0 u)^* T \ti{L}_0, (\partial_1 u)^* T \ti{L}_1)_{1,p,\alpha} \\
\to
\Omega^{0,1}(\ti{\Sigma},u^* TX)_{0,p,\alpha} \oplus \bigoplus_{j=1}^m
T_{u(w^\pm_j)} I_{\delta(j)}\end{multline}
given by the linearized Cauchy-Riemann operator on the disk components
and the differential of the evaluation at the nodes,
\begin{equation}  \label{thirddif} 
 (\zeta,\xi) \mapsto \prod_{j=1}^m \xi( w_j^+(0)) - \xi( w_j^-(0)) +
  \ddt |_{t = 0} u(w_j^+(t\zeta_j^+)) - \ddt |_{t = 0}
  u(w_j^-(t\zeta_j^-)) .\end{equation}
We say that $u$ is {\em regular} if $D_u$ is surjective.  Let
$M_\Gamma^{\reg}(L_0,L_1;H)$ denote the locus of regular stable
holomorphic quasistrips of combinatorial type $\Gamma$.

\begin{theorem}  \label{trajreg}
$M_\Gamma^{\reg}(L_0,L_1;H)$ is a smooth manifold with tangent space
  at $[u]$ isomorphic to the quotient of $\ker(D_u)$ by $\aut(\Sigma)
  \times \g$.
\end{theorem} 

\begin{proof}  The zero set $\cF_u^{-1}(0)$ is a smooth
manifold with tangent space $\ker(D_u)$ at $0$ by the implicit
function theorem for Banach manifolds.  Since $G$ acts freely on
$\ti{L}_0,\ti{L}_0$ and is compact, it acts properly on
$\cF_u^{-1}(0)$ and the quotient $\cF_u^{-1}(0)/G$ is a smooth
manifold. It remains to check that $\Aut(\Sigma)$ acts freely and
properly on $\cF_u^{-1}(0)/G$.  It suffices to consider the case that
$\Sigma$ is a disk with one or two markings, by considering each
component separately.  (There are no automorphisms permuting
components, because there are no automorphisms of a disk permuting the
points on the boundary.)  The automorphism group of any disk $\Sigma$
with one resp. two markings on the boundary can naturally be
identified with the group of translations and dilations
resp. dilations.  Suppose that $\phi$ is a non-trivial automorphism
fixing $Gu$, and let $\psi$ denote its (non-trivial) restriction to
the boundary.  Since $\psi$ fixes as least one point, for any other
point $z \in \partial \Sigma$ the sequence $\phi^n (z)$ converges to
some limit $z_\infty$.  Then $g^{-n} u(z) = u(\phi^n(z)) \to
u(z_\infty)$ for some element $g \in G$.  Since $G$ acts freely on
$\ti{L}_k$, $g$ is the identity.  If $\phi$ is a non-trivial
translation or dilation and $\phi^*u = u$ then $u$ must be constant,
so $\Aut(\Sigma)$ acts freely.  Similarly if $g_n \phi_n^* u$
converges to some map $v$ for some sequence $\phi_n \in \Aut(\Sigma)$
of automorphisms going to infinity and some sequence $g_n \in G$, then
$u,v$ are constant.  This shows that $\Aut(\Sigma)$ acts properly.
\end{proof} 

We denote by $M_\Gamma(L_0,L_1;H)_d$ the subset of stable
holomorphic quasistrips $u$ with $\Ind(D_u) - \dim(\aut(\Sigma)) -
\dim(\g) = d$.  Thus if $u$ is regular, then the
$M_\Gamma(L_0,L_1;H)_d$ of stable holomorphic quasistrips with the
same combinatorial type $\Gamma$ as $u$ is a smooth manifold of
dimension $d$ in a neighborhood of $u$, by Theorem \ref{trajreg}.

Let $M_1(\ti{L}_j)$ denote the moduli space of holomorphic disks $u:
(D,\partial D) \to (X,\ti{L}_j)$, modulo automorphisms fixing a point
$1 \in \partial D$.  For any such $u$ we denote by $D_u$ its usual
linearized Cauchy-Riemann operator, and by $M_1(\ti{L}_j)_d$ the
subset of $M_1(\ti{L}_j)_d$ with $\Ind(D_u) = d + 2$.  For $k = 0,1$
we denote by $M_{k,1}(\ti{L}_0,\ti{L}_1;H)$ the moduli space of
$(J,H)$-holomorphic strips for $\ti{L}_0,\ti{L}_1,H$ with a single
marking on the boundary component $\R \times \{ k \} \subset \R \times
[0,1]$, and $M_{k,1}(\ti{L}_0,\ti{L}_1;H)_d$ the subset whose quotient
$M_{k,1}(\ti{L}_0,\ti{L}_1;H)_d/G$ has formal dimension $d$.  

 \begin{theorem}  \label{bound} If $X$ is convex aspherical and $J,H$ are such 
that every non-constant stable holomorphic quasistrip with index $1$
or $0$ is regular then the one-dimensional component
$\ol{M}(L_0,L_1;H)_1$ is a compact one-manifold with boundary given by
\begin{multline} \label{boundl01}
\partial \ol{M}(L_0,L_1;H)_1 
= M(L_0,L_1;H)_0 \times_{\cI(L_0,L_1)}
M(L_0,L_1;H)_0 \\ \cup \bigcup_{k = 0,1}
(M_{k,1}(\ti{L}_0,\ti{L}_1;H)_{0} \times_{\ti{L}_k} M_1(\ti{L}_k)_0)/G
. \end{multline}
\end{theorem} 

\begin{proof} Except for the quotient by the group action and the clean
rather than transversal intersection, this is a standard combination
of compactness and gluing theorems, c.f. Oh \cite{oh:fl1}.
Compactness was proved in Theorem \ref{trajcompact}.  We sketch the
proof of the gluing theorem for clean intersections which follows
Fukaya et al \cite{fooo} and Abouzaid \cite{ab:ex}, who deal with
holomorphic strips with equal Lagrangian boundary conditions $\ti{L}_0
= \ti{L}_1$.  However, equality of the Lagrangian boundary in these
references is only used to obtain exponential decay, which we proved
in Lemma \ref{expdecay}.  Let $\delta = (\delta_1,\ldots, \delta_m)$
be a collection of gluing parameters, and $\Sigma^\delta$ the glued
strip.  We assume in order to simplify notation that all gluing
parameters are non-zero.  We denote by $\nu^\pm_j: \pm [-\infty,2S_j]
\times [0,1] \to \Sigma$ the embeddings given by the local coordinates
near the nodes, so that $ \{ 0 \} \times [0,1]$ maps to the
half-circle of radius $\sqrt{\delta_j}$.  (That is, take the logarithm
of the local coordinate and shift by $\log(\delta_j)/2$.)  The glued
surface is obtained by cutting off $\nu_j^\pm( \pm [2S_j,\infty)
  \times [0,1])$ and identifying the remaining finite strips.  Let
  $\nu_j: [-2S_j, 2S_j] \times[0,1] \to \Sigma^\delta$ be the
  coordinates described above on the $j$-th neck given by this
  procedure.  Using cutoff functions one defines an {\em approximate
    trajectory} $ G^{\on{approx}}_\delta u: \Sigma \to X $ given by
  $u$ away from the neck, and on the $j$-th neck as follows: Let
  $\rho: \R \to \R$ be a cutoff function with $\rho(s) = 1$ for $s \ge
  1$ and $\rho(s) = 0$ for $s \leq 0$, and $\xi_j^\pm \in \Omega^0(
  \pm [-\infty, 2S_j], u^* TX)$ are defined by requiring $
  \exp_{u(w_j)^{\pm}}( \xi_j^\pm) = (\nu^\pm_j)^* u$ on the component
  containing $w_j^\pm$.  Then we define
$$ \nu_j^* G^{\on{approx}}_\delta u 
(s,t) = \exp_{u(w_j)^{\pm}}( \rho(
1 - s) \xi_j^-(s,t) + \rho(s + 1) \xi_j^+(s,t)) .$$
An {\em approximate right inverse} to $D_{G^{\on{approx}}_\delta u}$
is constructed on Sobolev spaces obtained by weighting the standard
Sobolev norm on the neck.  More precisely, choose a function $\kappa :
\Sigma^\delta \to [1,\infty)$ such that $\kappa$ is the previously 
defined weight function \eqref{weight} outside of the
  neck and $\kappa(s,t)= \exp ((2S_j - |s|) p \alpha)$ on the neck
  $\nu_j$.  Then
$$ \Vert \eta \Vert_{0,p,\delta,\alpha}^p = \int_{\Sigma^\delta} \kappa
  \Vert \eta \Vert^p $$
defines a $\delta$-dependent norm on $\Omega^{0,1}(\Sigma^\delta, u^*
TX)$.  Similarly for any $\xi \in \Omega^0(\Sigma^\delta, u^* TX)$ we
define a norm as follows: Let $z_j \in \Sigma^\delta$ be the point in
the middle of the neck given by $\nu_j(0,1/2)$.  Let $\xi^{\on{const}}_j$
denote the section obtained by parallel transport of $\xi(z_j)$ using
the Levi-Civita connection from $u(0,t)$ to $u(s,t)$. Let
$$\Vert \xi \Vert_{1,p,\delta,\alpha}^p = \Vert \xi - \xi^{\on{const}}
\Vert_{0,p,\delta,\alpha}^p + \Vert \nabla( \xi - \xi^{\on{const}})
\Vert_{0,p,\delta,\alpha}^p + \Vert \xi^{\on{const}}(z) \Vert^p .$$
Then a patching argument \cite[(5.50)]{ab:ex} constructs from a
one-form $\eta$ on $\Sigma^\delta$ a zero-form $\xi = T^\delta_u \eta$
which satisfies $\Vert D_{G^{\on{approx}}_\delta u} T^\delta_u \eta -
\eta \Vert_{0,p,\delta,\alpha} \leq \sum_j e^{ -2 S_j \alpha} \Vert
\eta \Vert_{0,p,\delta,\alpha}$.  We denote by $\Omega^0(\Sigma, u^*
TX, (\partial_0)u^* T\ti{L}_0, (\partial_1 u)^* T
\ti{L}_1)_{1,p,\alpha,\delta}$ the subset of $\Omega^0(\Sigma, u^*
TX)_{1,p,\alpha,\delta}$ with boundary values on $\ti{L}_0,\ti{L}_1$.
The explicit formula for the patching of the element $\xi$ obtained by
applying the right inverse to $D_u$ is in terms of the coordinates
given by $\nu_j$
\begin{multline}
 \nu_j^* \xi^\delta(s,t) = \cT \xi(w_j^\pm) + \rho(S_j - s) ( \cT
 \xi_-(s,t) - \cT \xi(w_j^\pm)) \\ + \rho(s + S_j) ( \cT \xi_+(s,t) -
 \cT \xi(w_j^\pm) ) \end{multline}
where $\cT$ denotes suitable parallel transports. The proof that it is
an approximate right inverse requires uniform quadratic estimates, see
\cite[p. 89-96]{ab:ex}, that are the same in this situation as in
\cite{ab:ex} (except for the additional corrections discussed in
\cite[Section 4.3]{mau:gluing}) and we will not reproduce here; these
depend on the fact that the weighted Sobolev norms introduced above
control the usual Sobolev norms, and these in turn control by $C^0$
norm by estimates that are uniform in the gluing parameter $\delta$,
because the cone angle of the metric on $\Sigma^\delta$ is uniformly
bounded from below.  Because $ u $ decays exponentially on the ends by
some constant $\alpha_0$ given by Lemma \ref{expdecay}, there exists a
constant $C$ so that
$$\Vert \cF_{G^{\on{approx}}_\delta u }(0) \Vert_{0,p,\delta,\alpha} <
C e^{- S (\alpha_0 - \alpha)} .$$
The implicit function theorem then produces for $0 < \alpha <
\alpha_0$ and $S$ sufficiently large, a solution $(\zeta,\xi)$ to
$\cF_{G^{\on{approx}}_\delta(u)} (\xi) =0 $ with $\zeta \in
\Def(\Sigma^\delta)$ in the image of $\Def_\Gamma(\Sigma)$, and we set
$G_\delta(u) = \exp_{G^{\on{approx}}_\delta(u)}(\xi)$.  One expects
that the map $(\delta,u) \mapsto G_\delta(u)$ is actually injective,
but this is not needed: rather one shows that each configuration on
the right-hand side of \eqref{boundl01} is the limit of a unique
one-parameter family of elements of $M(L_0,L_1;H)_1$, using the
uniqueness of the solution given by the implicit function theorem.
\end{proof} 

Orientations on the moduli spaces of stable quasistrips can be
constructed as follows.  
Recall \cite{fooo},\cite{orient} that a {\em
  relative spin structure} on an oriented Lagrangian submanifold $L
\subset X$ is a lift of the class of $TL$ defined in the first
relative \v{C}ech cohomology group for the inclusion $i: L \to X$ with
values in $\on{SO}(\dim(L))$ to first relative \v{C}ech cohomology
with values in $\on{Spin}(\dim(L)))$.  The set of such objects
naturally forms a category, equivalent to the category of
trivializations of the image of the second Stiefel-Whitney class
$w_2(TL)$ in the relative cohomology $H^2(X,L;\Z_2)$.  In particular,
the set of isomorphism classes $\on{Spin}(TL,X)$ of relative spin
structures is non-empty iff $w_2(TL) \in H^2(L;\Z_2)$ is in the image
of $H^2(X;\Z_2)$, and if non-empty has a faithful transitive action of
$H^1(X,L;\Z_2)$.  Any relative spin structure on $L$ induces
orientations on the moduli spaces of stable holomorphic quasistrips as
in, for example, \cite{orient}, by deforming each linearized operator
so that index is identified with a complex vector space plus a tangent
space to the Lagrangian; the relative spin structure implies that the
resulting orientation is independent of the choice of deformation.

\begin{proposition} \label{orient} 
Relative spin structures on $L_0,L_1$ induce orientations on the
regular parts of the moduli spaces so that the inclusion of the
components in \eqref{boundl01} is orientation preserving for the
broken strips, and orientation preserving resp. reversing for bubbles
in $L_0$ resp. $L_1$.
\end{proposition} 

The proof is similar to that for nodal disks and strips given in
\cite{orient}, and omitted.  We denote by
$$\eps: M(L_0,L_1;H)_0 \to \{ \pm 1 \}$$ 
the map comparing the constructed orientations with the canonical
orientation of a point and by $\eps(u)$ its value at $[u] \in
M(L_0,L_1;H)_0$.

\subsection{Quasimap Floer cohomology} 

For any $x_\pm \in \cI(L_0,L_1;H)$ we denote by $\ti{x}_\pm \subset
\Map(I, X)$ the set of trajectories of $H^\#_t$ covering $x_\pm: I \to
X \qu G$ (any two lifts are related by the $G$-action) and by
$M(x_+,x_-) := M(x_+,x_-,L_0,L_1;H)$ the subset of $M(L_0,L_1;H)$ the
set of $(J,H)$-holomorphic quasistrips such that
$$ (t \mapsto \lim_{s \to \pm \infty} u(s,t)) \in \ti{x}_\pm .$$
For each $x \in \cI(L_0,L_1;H)$ we fix a trivialization of the
$\Lambda$-line bundle along $x$ and denote by
$$ \Hol_{L_0,L_1} : M(L_0,L_1;H) \to \Lambda $$
the product of parallel transport maps, that is, parallel transport in
$L_0$ from $x_+$ to $x_-$ along $u |_{\R \times \{ 0 \}}$ and then
parallel transport from $x_-$ to $x_+$ along $u |_{\R \times \{ 1
  \}}$; we think of this as a holonomy around the loop in $L_0 \cup
L_1$, although it is not because of the Hamiltonian
perturbation.  Let $CQF(L_0,L_1;H)$ the chain complex generated by
$\cI(L_0,L_1;H)$,
$$ CQF(L_0,L_1;H) = \bigoplus_{x \in \cI(L_0,L_1;H)} \Lambda \bra{x} .$$
Define 
$$ \partial_{L_0,L_1,H} \bra{x_+} = \sum_{[u] \in M(x_+,x_-)_0}
\eps(u) \Hol_{L_0,L_1}(u) q^{A(u)} \bra{x_-} .$$
For $k = 0,1$ we denote by $\mu_0(L_k)$ the count
$$ \mu_0(L_k) = \sum_{[u] \in M_1(\ti{L}_k,\ti{x}_k)_0} \eps(u)
q^{A(u)} \Hol_{L_k}(u)$$
of holomorphic disks $u:D \to X$ with boundary in $\ti{L}_k$ passing
through a generic point $\ti{x}_k \in \ti{L}_k$, where $\eps(u) \in \{
\pm 1 \}$ is the orientation induced by the relative spin structure;
assuming every such disk is regular these disks are all Maslov index
two.

\begin{theorem} \label{qfh} (c.f. \cite[Section 5.2]{frau:thesis}) 
Suppose that $X$ is convex aspherical, $L_0,L_1,J,H$ are such that
every stable holomorphic quasistrip is regular and every non-trivial
holomorphic disk has positive Maslov index.  Then
$\partial_{L_0,L_1,H}^2 = \mu_0(L_1)- \mu_0(L_0) $.
\end{theorem} 

\begin{proof} By Theorem \ref{bound}, the additivity of the 
energy, the multiplicativity of the holonomies, and orientations
Proposition \ref{orient}.
\end{proof} 

If $\mu_0(L_1) = \mu_0(L_0) $ then the {\em quasimap Floer cohomology}
is 
$$HQF(L_0,L_1) := HQF(L_0,L_1;H) := H(\partial_{L_0,L_1,H}) ;$$
if $L_0,L_1$ are equipped with $N$-Maslov gradings then this is
$\Z_{2N}$-graded group.  In our situation we will only have $N=1$, so
our Floer cohomologies will be only $\Z_2$-graded.

\begin{theorem}  In the situation of Theorem \ref{qfh},
if $J$ is such that every non-trivial stable holomorphic disk is
regular and has positive Maslov index, and $H_0,H_1 \in
C^\infty_c([0,1] \times X)^G$ are two Hamiltonians so that every
stable holomorphic quasistrip is regular, then $HQF(L_0,L_1;H_0)$ is
isomorphic to $HQF(L_0,L_1;H_1)$.  If $L_1 \cap \phi(L_0)$ is empty
for some Hamiltonian diffeomorphism $\phi$ on $X \qu G$ then
$HQF(L_0,L_1;H)$ is trivial for any $H$.
\end{theorem}  

\begin{proof}  The first statement is a standard continuation argument counting
$(J,\ti{H})$-holomorphic strips where $\ti{H}$ is a generic homotopy
  between $H_0 \d t $ and $H_1 \d t$.  Since we give a more general
  argument in the \ainfty setting in Section \ref{bimoduleiso}, we
  omit the proof.  To prove the second statement, let $H_0 \in
  C^\infty_c([0,1] \times X \qu G) $ be a time-dependent Hamiltonian
  whose flow $\phi$ satisfies $L_1 \cap \phi(L_0) = \emptyset$.  If $H
  \in C^\infty_c([0,1] \times X)^G$ is any lift of $H_0$ then we have
  $\cI(L_0,L_1;H) = \emptyset$ hence $HQF(L_0,L_1;H)$ vanishes.
\end{proof} 

The homotopy argument works with $\Lambda$ coefficients but not
$\Lambda_0$-coefficients: because of the energy-area identity
\eqref{energyarea}, the symplectic area of a $(J,\ti{H})$-strip is
possibly negative.  That is, quasimap Floer cohomology, as well as
Floer cohomology, is defined using $\Lambda_0$ coefficients but as
such is not an invariant of Hamiltonian isotopy.

Notice that we have not said anything about independence from $J$.
There is probably no hope of achieving regularity for an arbitrary
family of almost complex structures, and so the independence from $J$
falls outside of the techniques considered in this paper.  

\section{Quasimap \ainfty algebra for a Lagrangian} 
\label{algebra}

In order to understand the structure of the Floer cohomology groups
defined in the previous section, it is helpful to understand how they
arise via \ainfty algebras and bimodules.  This will take the next
several sections; in this section we use a standard trick which
combines Morse theory and Floer cohomology via {\em treed disks}, and
which in particular gives an \ainfty algebra associated to a
Lagrangian without using Kuranishi structures if every stable
holomorphic disk is regular and positive Maslov index.  While this is
well-known (see for example Seidel \cite{seidel:genustwo}) there is
unfortunately no complete writeup in the literature, and the version
we need replaces disks in the quotient with quasidisks.  We begin with
some generalities about \ainfty algebra.

\subsection{\ainfty algebras}

An {\em \ainfty algebra} over $\Lambda$ consists of a $\Z$-graded
$\Lambda$-module $A$ and a sequence of {\em higher composition maps}
$(\mu_n: A^{\otimes n} \to A[2-n])_{ n \ge 0} $
satisfying the {\em \ainfty associativity relation}
\begin{equation} \label{assoc}
  0 = \sum_{i + j \leq n} (-1)^{\aleph}
\mu_{n-i+1}(a_1,\ldots,a_{j},\mu_i(a_{j+1},\ldots,a_{j+i}),
a_{j+i+1},\ldots,a_n) \end{equation}
where $a_1,\ldots,a_n$ are homogeneous elements of degree
$|a_1|,\ldots,|a_n|$,
\begin{equation} \label{aleph}
 \aleph = \sum_{k=1}^j (|a_k| - 1) \end{equation}
is the sum of the {\em reduced degrees} of the elements to the left of
the inner operation, see Seidel \cite{seidel:sub}, and we adopt the
standard convention of writing commas instead of tensor products to
save space.  The map $\mu_0: \Lambda \to A$ is the {\em curvature} of
$A$ and is determined by a single element $\mu_0(1) \in \Lambda$.  If
$\mu_0 = 0$ then $A$ is {\em flat}. The signs in the \ainfty
associativity relation only depend on the induced $\Z_2$-grading, and
this means that one can replace the assumption of a $\Z$-grading with
a $\Z_2$-grading, which will be the case in our example.  We also work
with \ainfty algebras defined over the Novikov ring $\Lambda_0$, which
means that $A$ is a $\Lambda_0$-module and the maps $\mu_n$ are
$\Lambda_0$-module morphisms.

A {\em morphism} of \ainfty algebras $(A_1,
(\mu_{1,i})),(A_2,(\mu_{2,i}))$ is a collection
$ (\phi_n :A_1^{\otimes n} \to A_2 )_{n \ge 0} $
satisfying a relation 
\begin{multline}\label{ainftymorphism} 0 = \sum_{i,j} (-1)^\aleph \phi_{n- j  + 1}( a_1,\ldots, a_i,
\mu_{j,1}(a_{i+1},\ldots,a_{i+j}),a_{i+j+1},\ldots, a_n) +
\\ \sum_{i_1,\ldots,i_r} \mu_{r,2}(\phi_{i_1}(a_1,\ldots,a_{i_1}),
\ldots, \phi_{i_r}(a_{i_1 + \ldots + i_{r-1} + 1},\ldots, a_n)) .\end{multline}
The map $\phi_0$ is the {\em curvature} of the morphism; if $\phi_0 =
0$ then the morphism is {\em flat}.

\subsection{Treed disks} 
\label{associahedron}

Stasheff's {\em associahedron} $K_n$ is a cell complex whose vertices
correspond to total bracketings of $n$ variables $x_1,\ldots,x_n$
\cite{st:ho}.  For example, $K_3$ is an interval, with vertices
corresponding to the expressions $(x_1x_2)x_3$ and $x_1(x_2x_3)$.  The
associahedron has several geometric realizations.  The first,
well-known description is as the moduli space of stable marked nodal
disks.  A nodal disk is {\em stable} if it has no automorphisms, or
equivalently, each disk component has at least three nodal or marked
points.  Let $\ol{M}_n$ denote the moduli space of stable nodal
$n+1$-marked disks.

A second and older description of the associahedron due to
Boardman-Vogt \cite{boardman:hom} is the moduli space of stable ribbon
metric trees: A {\em ribbon metric tree} is a ribbon tree 
$$\Gamma =
(V(\Gamma),E(\Gamma),O(\Gamma))$$ 
consisting of a ribbon tree $(V(\Gamma),E(\Gamma))$ equipped with a
labelling of the semiinfinite edges by integers $0,\ldots, n$, a
cyclic ordering $O(\Gamma)$ of the edges at each vertex together with
a {\em metric} $l: E(\Gamma) \to [0,\infty]$.  The semiinfinite edges
are required to have infinite length.  Two metric trees are {\em
  isomorphic} if the ribbon metric trees obtained by collapsing all
edges of length zero are isomorphic, that is, there is a bijection
between the vertex sets preserving the edges, the cyclic orderings,
and their lengths.  Equivalently, a metric tree can be taken to be a
tree with a metric on each one-dimensional segment, and an isomorphism
of metric trees is required to preserve the metric on each edge.  A
metric ribbon tree is {\em stable} if each vertex has valence at least
two, any edge containing a vertex of valence two has infinite length,
and any edge of infinite length contains at least one vertex of
valence three.  (This is a somewhat obscure way of saying that we
allow an edge of finite length to degenerate to a broken edge of
infinite length; the standard method would just be to allow edges of
infinite length but this convention would cause notational problems
when we talk about gradient trees later on.)  There is a natural
topology on the set of isomorphism classes $\ol{W}_n$ of stable ribbon
metric trees with $n + 1$ leaves, which allows each edge with length
approaching infinite to degenerate to a broken edge, that is, a pair
of infinite length edges joined by a vertex.  We refer to
\cite{mau:mult} for a review.  The edges of each tree are oriented so
that the positive orientation is in the direction of the zero-th
semiinfinite edge.

Both $\ol{W}_n$ and $\ol{M}_n$ have natural structures as manifolds
with corners, isomorphic as such to Stasheff's associahedron $K_n$.
In the case of $\ol{M}_n$ the boundary components have interiors
consisting of configurations with exactly two components containing
prescribed markings.  In the case of $\ol{W}_n$ the boundary
components have interiors whose trees have a single edge of infinite
length.  The cell structure of $\ol{M}_n$ is identical to that of
$K_n$ while the cell structure of $\ol{W}_n$ is not.

The following definition combines the two constructions: A {\em treed
  disk} is a collection of a collection $D_1,\ldots, D_k$ of
holomorphic disks, segments-with-metric $S_i$ which for simplicity we
take to be intervals $S_i = [\eps_i^-,\eps_i^+], i = 1,\ldots, l$
where $\eps_i^- < \eps_i^+$ and $\eps_i^\pm \in [-\infty,\infty]$,
together with a collection 
$$\{ \{ w_1^+,w_1^-\},\ldots, \{ w_m^+,w_m^-\} \}$$ 
of {\em nodes} in the disjoint union $\partial D_1 \sqcup \ldots
\sqcup \partial D_k \sqcup \partial S_1 \sqcup \ldots \sqcup \partial
S_l $ such that the set $w_1^+,\ldots,w_m^-$ is distinct and {\em
  markings} $z_1,\ldots, z_n$ contained in $\partial S_1 \sqcup \ldots
\sqcup \partial S_l $ disjoint from the nodes.  Gluing the nodes gives
a topological space
$$ \Sigma := \left( D_1 \sqcup \ldots \sqcup D_k \sqcup S_1 \sqcup \ldots
\sqcup S_l \right) / (w_j^+ \sim w_j^-)_{j=1}^m $$
which is required to be contractible. The ordering of the markings is
required to agree with the ordering of the leaves of the underlying
ribbon tree, that is, the tree obtained by collapsing the disks to
vertices and taking the ribbon structure induced from the ordering of
the nodes on the boundary of the disks.  We say that a node $w_j^\pm =
\eps_j^\pm$ contained in the boundary of $S_j = [\eps_j^-,\eps_j^+]$
is {\em infinite} if $\eps_j^\pm = \pm \infty$, and {\em finite}
otherwise.  An {\em isomorphism} of treed nodal disks $\Sigma,
\Sigma'$ is a homeomorphism $\Sigma \to \Sigma'$ holomorphic on each
disk component and preserving the metric on each line segment, mapping
the markings $z_1,\ldots, z_n$ to $z_1',\ldots,z_n'$.  A treed nodal
disk is {\em stable} if it has no automorphisms, or equivalently, each
disk component has at least three nodal or marked points, and in
addition each node connecting two line segments is infinite.  More
explicitly, this means that each disk component has at least three
nodal points, and each sequence of line segments connecting two disk
components is a {\em broken line segment}: all nodes of finite type
attach to disk components.  Let $\ol{MW}_n$ denote the moduli space of
connected stable tree disks with $n+1$ semiinfinite edges.  See Figure
\ref{MW} for the example $n = 3$.
\begin{figure}[ht]
\includegraphics[height=.7in]{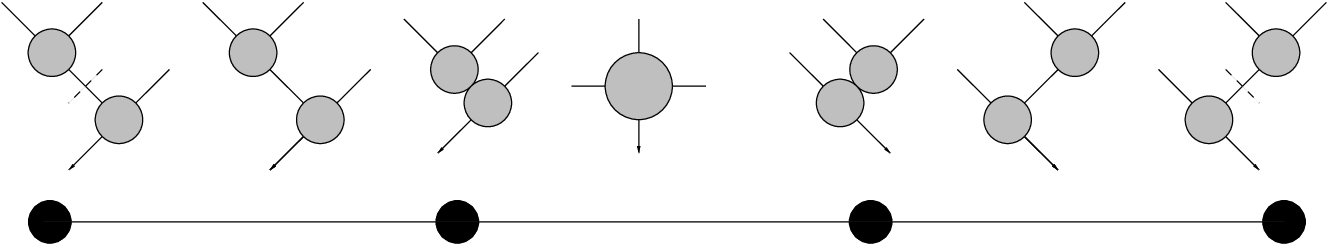}
\caption{Moduli space of stable treed disks}
\label{MW}
\end{figure} 
\noindent The combinatorial type of a treed disk is a ribbon tree $\Gamma =
(V(\Gamma),E(\Gamma),O(\Gamma))$ with a partition of the set of
vertices $V(\Gamma)$ into two subsets $V^1(\Gamma),V^2(\Gamma)$,
corresponding to the two-dimensional and one-dimensional components
being glued together.  We denote by $MW_{\Gamma,n}$ the subset of
$\ol{MW}_n$ of combinatorial type $\Gamma$.  A codimension one stratum
$MW_{\Gamma,n}$ is a {\em true boundary component} if it consists of
points in the topological boundary of $\ol{MW}_n$, considered as a
topological manifold with boundary, and is a {\em fake boundary
  component} otherwise.  For any (not necessarily stable) treed disk
$\Sigma$ we denote by $\ti{\Sigma}$ denote the disjoint union of the
disk components of $\Sigma$, together with a copy of the real line for
each one-dimensional component of $\Sigma$.  Denote by
$\Def_\Gamma(\Sigma)$ the space of infinitesimal deformations of
$\Sigma$ preserving the combinatorial type, that is, the direct sum of
the tangent spaces of the boundaries of the disks or tangent spaces to
the lines at the nodes and markings (representing deformations of the
attaching points and markings)
\begin{equation} \label{defsig}
 \Def_\Gamma(\Sigma) = \bigoplus_{j=1,\pm}^m T_{w^\pm_i} \partial
 \ti{\Sigma} \oplus \bigoplus_{j=1}^n T_{z_i} \partial \ti{\Sigma}
\end{equation}
where in the case of one-dimensional components we define the boundary
to be the entire component.  Over a neighborhood of $0$ in
$\Def_\Gamma(\Sigma)$ we have a family of treed disks, given by
varying the attaching points using a collection of exponential maps
$T_{w^\pm_i} \partial \ti{\Sigma} \to \partial \Sigma$ defined using a
metric on $\partial \ti{\Sigma}$.  This is a smoothly trivial family
in the sense that the fibers are all smoothly diffeomorphic, and any
choice of identification identifies the deformations with deformations
of the complex structure on the disk components and deformations of
the metric on the line segments, but our construction of the gluing
map will not use such a trivialization.  Let $\Aut(\Sigma) =
\Aut(\ti{\Sigma})$ denote the product of the automorphism groups
$SL(2,\C)$ for each disk component together the automorphism groups
$\R$ of each one-dimensional component, and $\aut(\Sigma)$ the Lie
algebra of $\Aut(\Sigma)$.  The action of $\Aut(\Sigma)$ on
$\ti{\Sigma}$ induces a map $\aut(\Sigma) \to \Def_\Gamma(\Sigma)$.
If $\Sigma$ is stable, then a neighborhood of $[\Sigma]$ in
$MW_{\Gamma,n}$ is homeomorphic to a neighborhood of $0$ in
$\Def_\Gamma(\Sigma)/\aut(\Sigma)$.  The full deformation space is
obtained by taking the direct sum
$$ \Def(\Sigma) = \Def_\Gamma(\Sigma) \oplus ((-\infty,0) \cup \{ 0 \}
\cup (0,\infty))^z \oplus [0,\infty)^i $$
where $z$ is the number of edges of zero length (that is, nodes
connecting disks) and $i$ is the number of infinite nodes connecting
segments of infinite length.  A neighborhood of $[\Sigma]$ in
$\ol{MW}_n$ is isomorphic as a cell complex to a neighborhood of $0$
in $\Def(\Sigma)/\aut(\Sigma)$, by a map $(\zeta,\delta) \mapsto
[\Sigma^{\zeta,\delta}]$ obtained by combining the deformations above
with, in the case of a negative gluing parameter associated to an edge
of zero length, carrying out the gluing procedure for disks mentioned
above, see e.g. \cite{mau:mult}.  Thus in particular the true boundary
components of $\ol{MW}_n$ correspond to configurations with an edge of
infinite length.

$\ol{MW}_n$ admits {\em two} natural forgetful maps, one which forgets
the disk components $f_M: \ol{MW}_n \to \ol{W}_n$ and takes the ribbon
structure to be the one induced from the ordering of the points on the
boundaries of the disks, and one that forgets the edges $f_W:
\ol{MW}_n \to \ol{M}_n$. The fiber of $f_M$ over each stratum
$M_{n,\Gamma}$ is a product of intervals $[0,\infty]$ corresponding to
the lengths of the edges, while the fiber over a stratum
$W_{n,\Gamma}$ is a product of associahedra corresponding to the
valences at the vertices.  The product $f_M \times f_W: \ol{MW}_n \to
\ol{M}_n \times \ol{W}_n$ defines an injection into a product of
compact spaces, so that $\ol{MW}_n$ has a natural topology for which
it is compact.  With respect to this topology each stratum
$MW_{\Gamma,n}$ has a neighborhood in $\ol{MW}_n$ homeomorphic to the
product of $MW_{\Gamma,n}$ with a product of intervals $[0,\infty)$,
  one for each node connecting segments of infinite length, together
  with a product of intervals $(-\infty,\infty)$, one for each node
  connecting disk components.  From the first forgetful map one sees
  that $\ol{MW}_n$ is again homeomorphic to Stasheff's associahedron,
  but with a more refined cell structure.

The strata of $\ol{MW}_n$ can be oriented as follows.  First, $M_n$ is
oriented via the identification with $\{ 0 < z_2 < z_2 < \ldots <
z_{n-1} < 1 \} \subset \R^{n-2}$ which maps the first marked point to
$\infty$, the second to $0$, the last to $1$, and the remaining points
to the interval $(0,1) \cong \R$.  The orientations of the boundary
strata can be determined as follows.  Any stratum $M_{\Gamma,n}
\subset \ol{M}_n$ corresponding to a tree with $m$ nodes has a
neighborhood homeomorphic to a neighborhood of $M_{\Gamma,n}$ in
$M_{\Gamma,n} \times [0,\infty)^m$.  In the case $m = 1$ so that
  $M_{\Gamma,n} \cong M_i \times M_{n-i+1}$ for some $i$, this
  homeomorphism has an inverse given by a {\em gluing map} of the form
\begin{multline} \label{gluemap}
 \R \times M_i \times M_{n-i+1},
   (\delta,(w_2,\ldots,w_{i-1}), (
   z_2,\ldots,z_{n-i})) \to
   \\ (z_2,\ldots,z_j,z_j + \delta w_2,\ldots, z_j +
   \delta w_{i-1}, z_j + \delta,z_{j+1},\ldots, z_{n-i}) .\end{multline}
It follows that the inclusion $M_i \times M_{n-i+1} \to M_n$ has
orientation sign $ij + 1 - j - i $.  Next, each stratum
$MW_{\Gamma,n}$ corresponding to a combinatorial type $\Gamma$ of
marked stable disks is isomorphic to
\begin{equation} \label{orienttreed}
 M_{\Gamma,n} \times \prod_{e \in E_{< \infty}(\Gamma)}
 \R \end{equation}
where $M_{\Gamma,n}$ is the corresponding stratum of $\ol{M}_n$ and
$E_{< \infty}(\Gamma)$ is the set of finite edges.  Here the order
is taken first by distance to the root edge of $\Gamma$, and then by
cyclic order determined by the ribbon structure on $\Gamma$.  Thus
$MW_{\Gamma,n}$ inherits an orientation from that on $M_{\Gamma,n}$,
which in turn is induced by an orientation on $\ol{M}_n$ induced from
that on $M_n$.  On the other hand, the boundary of $\ol{MW}_n$ may be
identified with the union over types $\Gamma$ of trees with one
infinite edge of the product $\ol{MW}_{n-i+ 1} \times \ol{MW}_i$ where
$n+1$ is the number of semi-infinite edges of the component tree not
containing the root.  Thus
$$ \partial \ol{MW}_n \cong \bigcup_{i+j \leq n+1} MW_{n-i+ 1}
\times MW_i $$
where the sum is over edges $i$ of the tree with $n-i + 1$ leaves
other than the root.  Then by construction, the inclusion of $MW_{n-i
  + 1} \times MW_i $ has sign that of the inclusion of $M_{n-i + 1}
\times M_n$.

Over $\ol{MW}_n$ there is a {\em universal treed disk} $\ol{UMW}_n \to
\ol{MW}_n$ with the property that the fiber over $[\Sigma]$ is
isomorphic to the treed disk $\Sigma$. This is again a cell complex,
with cells indexed by pairs consisting of a combinatorial type of tree
disk together with a cell in the corresponding treed disk.

\subsection{Unstable trees} 

Standard perturbation schemes for Fukaya categories, as in Seidel
\cite{se:bo}, use systems of perturbations depending on the underlying
stable disk.  We need a different perturbation scheme which depends on
the underlying (possibly unstable) tree, similar to that used for the
construction of the Morse \ainfty category via perturbations in
e.g. Abouzaid \cite[Section 2.2]{ab:top}, which has non-vanishing
perturbations on edges that would normally be collapsed by the
forgetful morphism to the moduli space of stable trees.  This section
may be skipped at first reading.

Let $\ol{W}_{n,v}$ denote the moduli space of connected rooted metric
ribbon trees with $n+1$ leaves, at least one vertex, and at most $v$
vertices, and edges of possibly zero or infinite length, where
equivalence is generated by collapsing edges of zero length.  Given a
fixed underlying ribbon tree $\Gamma$, the corresponding stratum
$W_{n,v,\Gamma}$ is homeomorphic to a product of intervals
$(0,\infty)$ describing the lengths of the edges, and so the
decomposition
$$\ol{W}_{n,v} = \bigcup_\Gamma W_{n,v,\Gamma} $$ 
gives $\ol{W}_{n,v}$ the structure of a finite cell complex, hence a
compact Hausdorff space.  A neighborhood of $W_{n,v,\Gamma}$ is
homeomorphic to a neighborhood of $W_{n,v,\Gamma} \times \{ 0 \}$ in
the product of $W_{n,v,\Gamma}$ with a subset of $\prod_{v \in
  V(\Gamma)} W_{n(v),v}$ with at most $v$ vertices in total, where
$n(v)$ is the valence of $v$; the homeomorphism is obtained by
replacing each vertex $v$ by the corresponding tree.
Over $\ol{W}_{n,v}$ there is a {\em universal tree} $\ol{UW}_{n,v} \to
\ol{W}_{n,v}$ with the property that the fiber over $[\Sigma]$ is
isomorphic to the treed disk $\Sigma$. The moduli space $\ol{W}_{1,3}$
is the square with edges identified shown in Figure \ref{m30}.

\begin{figure}[ht]
\includegraphics[height=1.5in]{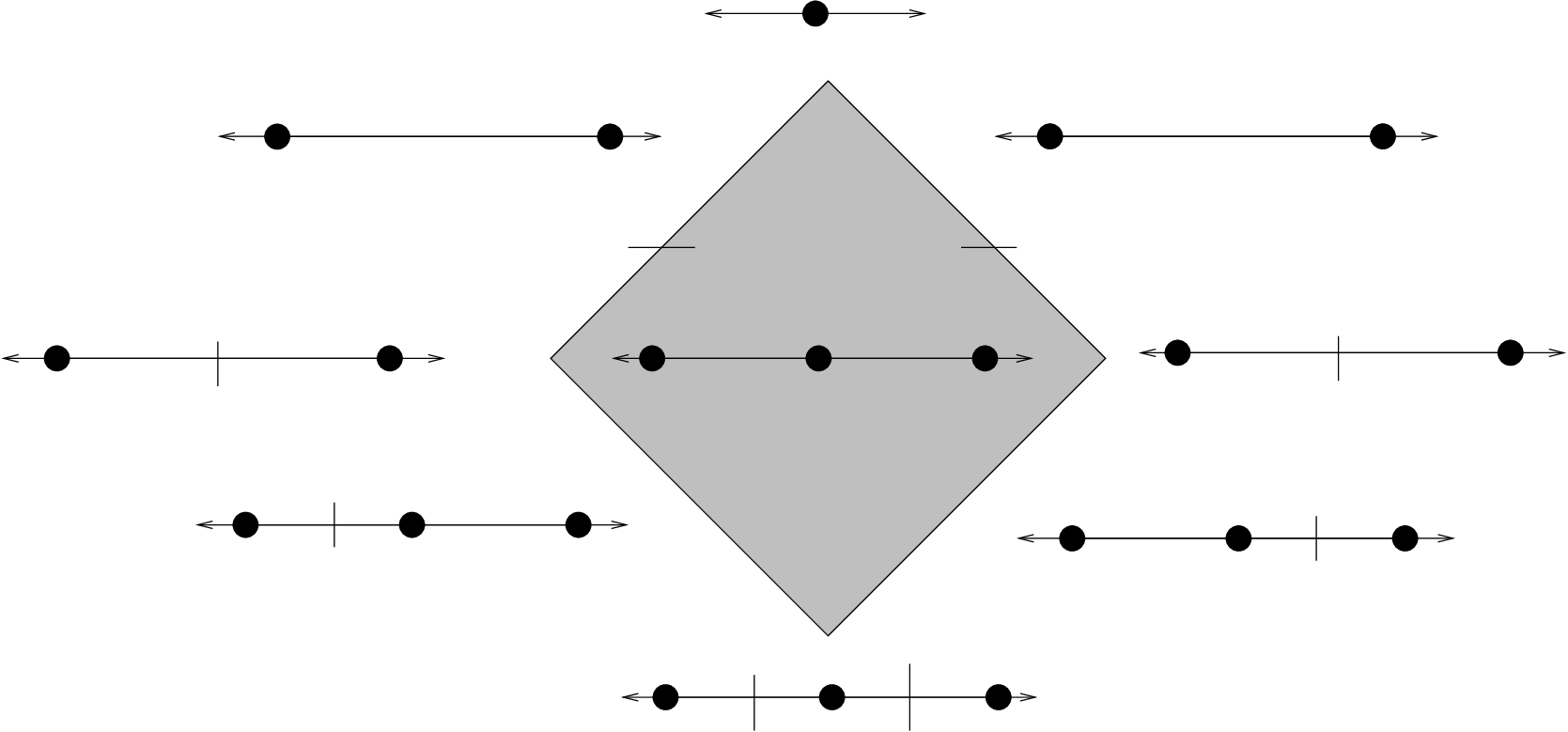}
\caption{Strata of $\ol{W}_{1,3}$}
\label{m30}
\end{figure} 
\noindent Note that $\ol{W}_{1,3}$ is not combinatorially a manifold with
corners.  

\subsection{Holomorphic treed quasidisks} 
\label{quasidisks}

Let $X$ be a Hamiltonian $G$-manifold equipped with a compatible
almost complex structure $J \in \J(X)^G$, as above.  Let $L \subset X
\qu G$ be a compact oriented Lagrangian submanifold equipped with a
brane structure, that is, a relative spin structure and grading.  Let
$(B: TL^{\otimes 2} \to \R, F: L \to \R)$ be a Morse-Smale pair, that
is, $B$ is a Riemannian metric, $F$ is Morse, and for each pair of
critical points $x_\pm \in \crit(F)$, the intersection of the stable
resp. unstable manifolds $W^+(x_+) \cap W^-(x_-)$ is transverse.  In
particular, the moduli space of gradient trajectories from $x_-$ to
$x_+$ is naturally isomorphic to $(W^+(x_+) \cap W^-(x_-) )/\R$, which
is a smooth finite-dimensional manifold.  The stable resp. unstable
manifolds $W^\pm(x)$ of $x \in \crit(F)$ are diffeomorphic to the
positive resp. negative parts $T^\pm _xL$ of the tangent space $T_x L$
with respect to the Hessian of $F$; we denote by $I(x) = \dim(T^-_x
L)$ the {\em index} of $x$ so that the dimension of the space of
gradient trajectories $(W^+(x_+) \cap W^-(x_-) )/\R$ is $I(x_+) -
I(x_-) - 1$.  For simplicity, we assume that $L$ is connected and
there is a unique critical point $x_{\min} \in \crit(F)$ of index $0$. In
addition, we will need to lift gradient trajectories of $F$ in $L$ to
the pre-image $\ti{L}$ in $X$.  For this purpose we assume that we
have fixed a $G$-invariant metric on $\ti{L}$ so that the induced
metric on $F$ is obtained from the quotient $F = \ti{F}/G$.  We may
then identify the gradient trajectories of $F$ on $L$ with those of
its $G$-invariant lift $\ti{F}$ on $\ti{L}$, up to the action of $G$.

\begin{definition} A {\em holomorphic treed quasidisk} for $L$ 
consists of a treed disk $\Sigma$ together with a map $u: \Sigma \to
X$ such that on each disk component in $\Sigma$, $u$ is holomorphic
and maps the boundary to $\ti{L}$, and on each edge in $\Sigma$, $u$
is a gradient flow line of the function $\ti{F}: \ti{L} \to \R$.  An
{\em isomorphism} of treed quasidisks $u_j: \Sigma_j \to X, j = 0,1$
consists of an isomorphism $\phi: \Sigma_0 \to \Sigma_1$ and an
element $g \in G$ such that $\phi^* u_1 = g u_0$.  A holomorphic treed
quasidisk is {\em stable} if $u$ has no automorphisms (equivalently,
any disk on which $u$ is constant has at least three marked or nodal
points and $u$ is non-constant on each segment of each broken line)
and if two line segments are joined at a node, then the node maps to a
critical point of $F$.  (This means that any node of finite type is an
attaching point to a disk.)
\end{definition} 

Because the semiinfinite edges have infinite length, the limit of $u$
along the $j$-th semiinfinite edge maps to a critical point $x_j \in
\crit(F)$ under the projection $\ti{L} \to L$.  A picture of a
holomorphic treed disk is shown in Figure \ref{holtreed}.

\begin{figure}[ht] 
\begin{picture}(0,0)%
\includegraphics{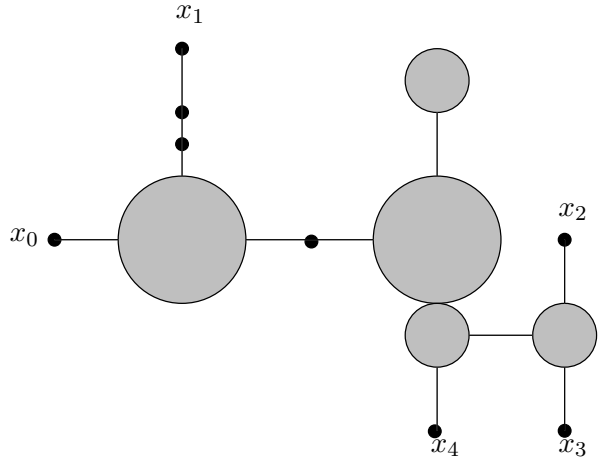}%
\end{picture}%
\setlength{\unitlength}{4144sp}%
\begingroup\makeatletter\ifx\SetFigFontNFSS\undefined%
\gdef\SetFigFontNFSS#1#2#3#4#5{%
  \reset@font\fontsize{#1}{#2pt}%
  \fontfamily{#3}\fontseries{#4}\fontshape{#5}%
  \selectfont}%
\fi\endgroup%
\begin{picture}(3533,2694)(1606,-4456)
\put(1621,-3182){\makebox(0,0)[lb]{$x_0$}}%
\put(2613,-1847){\makebox(0,0)[lb]{$x_1$}}%
\put(4902,-3030){\makebox(0,0)[lb]{$x_2$}}%
\put(4902,-4441){\makebox(0,0)[lb]{$x_3$}}%
\put(4139,-4441){\makebox(0,0)[lb]{$x_4$}}%
\end{picture}%
\caption{A holomorphic treed disk}
\label{holtreed}
\end{figure} 

Our perturbation scheme depends on choosing generic perturbations of
the gradient flow equations on the edges, as in the construction of
the Morse \ainfty algebra by perturbation.  For any treed disk
$\Sigma$ and $v$ sufficiently large we denote by $\Sigma_{(1)} \in
\ol{W}_{n,v}$ the (possibly unstable) metric ribbon tree obtained by collapsing the disk
components of $\Sigma$ to vertices.  Let $F_{n,v} \in
C^\infty(UW_{n,v} \times L)$ resp.  $B_{n,v} \in C^\infty( UW_{n,v}
\times TL^{\otimes 2})$ be a function equal to $F$ resp. metric equal
to $B$ on the complement of a compact subset of the union of open
edges.  An {\em $(F_{n,v},B_{n,v})$-perturbed treed disk} is a treed
nodal disk $\Sigma$ together with a continuous map $u: \Sigma \to X$
such that $u$ satisfies the Cauchy-Riemann equation on the disks and
the gradient equation
\begin{equation} \label{Fpert}
 (\ddt u)(t) = \grad(\ti{F}_{n,v}(\Sigma_{(1)},t,u(t))) \end{equation}
for the lift $\ti{F}_{n,v}$ of $F_{n,v}$ from $L$ to $\ti{L}$ at the
points $t$ in the segments of $\Sigma$, which are identified with the
segments of $\Sigma_{(1)}$.  That is, $(\Sigma_{(1)},t)$ represents a
point on the universal tree $\ol{UW}_{n,v}$.  The parameters $t$ on
the segments are well-defined up to translation, because of the choice
of metric.  The {\em energy} of a holomorphic treed disk is the sum of
the energies of the disk components. In particular, there is no
contribution to the energy from the gradient flow lines.  Given a
bound on the energy, we obtain a bound on the number of disk
components in any holomorphic treed disk by energy quantization for
disks, hence a bound on the number of vertices in the underlying tree.
We always suppose that $v = v(E)$ has been chosen sufficiently large
so that if $u: \Sigma \to X$ is a holomorphic treed disk of energy at
most $E$ then $\Sigma_{(1)}$ has at most $v$ vertices, and so defines
a point in $\ol{W}_{n,v}$.

\begin{theorem} \label{compactthm}  If $X$ is aspherical convex and $L$ is
compact, then any sequence of stable holomorphic treed disks with
bounded energy and number of semiinfinite edges has a convergent
subsequence.
\end{theorem} 

\begin{proof}  The number of segments in any broken gradient trajectory
is bounded by $\dim(X)$ by the Morse-Smale condition.  Furthermore the
number of disk components is bounded by the energy divided by the
constant in energy quantization for holomorphic disks, proved in
McDuff-Salamon \cite[Proposition 4.1.4]{ms:jh}. It follows that the
number of possible combinatorial types is finite, and so we may assume
after passing to a subsequence that the combinatorial type is
constant.

By Gromov compactness for marked disks, each sequence of disk
components converges to a nodal disk, after passing to a subsequence.
Similarly after passing to a subsequence the sequence of gradient
trees converges to a limiting gradient tree.  The limiting treed disk
is obtained by attaching the limiting gradient segments to their
attaching points in the limiting holomorphic disks, treating the
attaching points as markings. \end{proof}

For each homology class $\mdeg \in H_2(X,\ti{L})$ let $v(\mdeg)$ be
such that any holomorphic treed disk $u : \Sigma \to X$ has at most
$v(\mdeg)$ vertices or disk components.  Choose a function
$F_{n,v(\mdeg)} \in C^\infty({UW}_{n,v(\mdeg)} \times L)$ and let
${MW}_n(L,\mdeg)$ denote the moduli space of stable perturbed
holomorphic treed quasidisks with boundary in $L$ of class $\mdeg$
using the perturbation $F_{n,v(\mdeg)}$.  Given two moduli spaces
${MW}_{n-i+1}(L,\mdeg), {MW}_i(L,\mdeg')$ defined using perturbations
$F_{n-i+1,v(\mdeg)}, F_{i,v(\mdeg')}$, a gluing procedure produces a
perturbation system for ${MW}_n(L,\mdeg+\mdeg')$ in a neighborhood of
${MW}_{n-i+1}(L,\mdeg) \times_L {MW}_i(L,\mdeg')$, with the function
on the collapsing segment equal to $F$ in a neighborhood of the
infinite node in the case that the length goes to infinity.
Similarly, for any stratum $\Gamma$ corresponding to an edge of zero
length, any perturbation datum for $MW_{\Gamma,n}(L)$ produces a
perturbation datum in a neighborhood by taking the perturbation to
vanish on the edge whose length is going to zero.  We say that a
system of perturbations $F_* = ( F_{n,v(\mdeg)})_{n \ge 0, \mdeg \in
  H_2(X,\ti{L})}$ has been {\em chosen compatibly} if the function
$F_{n,v(\mdeg'')}$ is equal to the function induced by gluing in a
neighborhood of the image of ${UW}_{n-i+1,v(\mdeg)} \times
{UW}_{i,v(\mdeg')}$ in ${UW}_{n,v}$ whenever $\mdeg'' = \mdeg +
\mdeg'$, and similarly in the case of perturbing an edge length from
zero.

Each stratum $MW_{\Gamma,n}(L)$ is given locally as the quotient of a
zero set of a Fredholm map of Banach manifolds.  For each $\zeta \in
\Def_\Gamma(\Sigma)$ let $(w^\pm_j(\zeta) \in \ti{\Sigma})_{j=1}^m$
denote the corresponding nodes.  Let $\ti{\Sigma}_{(d)}, d=1,2$ the
union of $1$ resp. $2$-dimensional components of $\ti{\Sigma}$, and by
$u_d, d = 1,2$ the restriction of $u$ to $\ti{\Sigma}_d$.  Let
\begin{multline} \cF_u: \Def_\Gamma(\Sigma) \oplus 
\Omega^0(\ti{\Sigma}_{(2)},(u_2)^*TX, (\partial u_2)^* T\ti{L})_{1,p}
\oplus \Omega^0(\ti{\Sigma}_{(1)}, (u_1)^* T \ti{L})_{1,p,\alpha} \\ \to
\Omega^{0,1}(\ti{\Sigma}_{(2)},(u_2)^* TX)_{0,p} \oplus
\Omega^1(\ti{\Sigma}_{(1)}, (u_1)^* T \ti{L})_{0,p,\alpha} \oplus
\bigoplus_{j=1}^m T_{u(w^\pm_j)} \ti{L} \end{multline}
denote the map given by
$ (\zeta,\xi) \mapsto \cT_u(\xi)^{-1} \olp_J \exp_u \xi $
on the components $\ti{\Sigma}_{(2)}$ of dimension $2$, the gradient
operator
$$ (\zeta,\xi) \mapsto \cT_u(\xi)^{-1}( \ddt \exp_u(\xi) -
\grad(\ti{F}_{n,v})(\Sigma_{(1)}(\zeta),t,\exp_u(\xi))) \d t $$
on the components $\ti{\Sigma}_{(1)}$ of dimension 1, together with
the differences in \eqref{diffs} where the position of the nodes
$w_j^\pm(\zeta)$ depends on the deformation parameter $\zeta$; the
weighted Sobolev spaces are defined as in \eqref{weight} but on $\R$
rather than $\R \times [0,1]$.  The space $MW_{\Gamma,n}(L)$ is given
near $u: \Sigma \to X$ as the zero set of $\cF_u$, quotiented by the
action of $G \times \Aut(\Sigma)$.  The {\em linearized operator}
associated to $\cF_u$ is the operator formed by combining \eqref{du}
with the linearized gradient operator
$$ (\zeta,\xi) \mapsto (\nabla_t \xi - \nabla_\xi \grad(\ti{F}_{n,v}))
\d t $$
on the one-dimensional components, and finally for the third factor
the differential of the third component of $\cF_u$:
\begin{equation}  \label{third} 
 (\zeta,\xi) \mapsto \prod_{j=1}^m \xi( w_j^+(0)) - \xi( w_j^-(0)) +
  \ddt |_{t = 0} u(w_j^+(t\zeta_j^+)) - \ddt |_{t = 0}
  u(w_j^-(t\zeta_j^-)) .\end{equation}
The first two components depend only on $\xi \in \Omega^0(\Sigma,u^*
TX)$.  We say that $u$ is {\em regular} if $D_u$ is surjective.  Let
$MW_{\Gamma,n}^{\reg}(L)$ denote the locus of regular maps of type
$\Gamma$.

\begin{theorem}  \label{mwsmooth}
$MW_{\Gamma,n}^{\reg}(L)$ is a smooth manifold in a neighborhood of
  $[u]$ with tangent space at $[u]$ given by the quotient of
  $\ker(D_u)$ by the subspace generated by the Lie algebra $\g$ and
  the infinitesimal automorphisms $\aut(\Sigma)$ of $\Sigma$.
\end{theorem} 

\begin{proof}   An argument similar to that used in the proof of 
Theorem \ref{trajreg} shows that $\cF_u$ is a smooth map of Banach
manifolds.  The linearized Cauchy-Riemann resp. gradient operator is
elliptic so the first two components of $D_u$ define a Fredholm
operator.  Then $D_u$ is itself Fredholm, since $\Def_\Gamma(\Sigma)$
is finite dimensional.  The claim of the theorem then follows by the
implicit function theorem for Banach manifolds, and properness of the
action of $\Aut(\Sigma) \times G$, which holds for the same reasons as
in the proof of Theorem \ref{trajreg}.
\end{proof} 

\noindent Denote the component of $\ol{MW}_n(L)$ with homotopy class
$\mdeg \in \pi_2(X,\ti{L})$ by $\ol{MW}_n(L,\mdeg)$.  We suppose that
a system $B_{n,v}$ of tree-dependent metrics on $L$, agreeing with the
given metrics on the ends, has been fixed.

\begin{theorem} \label{regpert} 
Suppose that $F_{n,v(\mdeg)}^0 \in C^\infty(UW_{n,v(\mdeg)} \times L)$
is such that every treed disk with boundary in $\ti{L}$ is regular,
over an open neighborhood $V$ of the boundary $\partial
W_{n,v(\mdeg)}$.  Let $\rho \in C^\infty_c( UW_{n,v(\mdeg)} )$ be a
compactly supported function which on the complement of $V$ is
non-vanishing somewhere on each edge.  Then there exists a comeager
subset $\PP^{\reg}(L) \subset \PP(L) := C^\infty_b(UW_{n,v(\mdeg)}
\times L)$ with the property that if $F_{n,v(\mdeg)}^1 \in
\PP^{\reg}(L)$ and $F_{n,v(\mdeg)} = F_{n,v(\mdeg)}^0 + \rho
F_{n,v(\mdeg)}^1$ then every perturbed holomorphic treed quasidisk of
class $\mdeg$ is regular.
 \end{theorem} 

\begin{proof}   The reader may wish to compare with the argument in 
Seidel's book \cite{se:bo} of the Fukaya \ainfty algebra in the exact
case, or in the construction of the Morse \ainfty algebra, see
Abouzaid \cite[Section 2.2]{ab:top}.  The {\em universal moduli space}
\begin{multline} 
MW^{\on{univ}}_n(L) = \{ (F_{n,v(\mdeg)}, u, (w_j^\pm)_{j=1}^m) \in
C^l_b(UW_{n,v(\mdeg)} \times L) \times \Map(\ti{\Sigma}, X,
\ti{L})_{1,p} \times \\  (\partial \ti{\Sigma}_{(2)} \cup
\ti{\Sigma}_{(1)})^{2m} \ |  \olp u = 0, \ \eqref{Fpert} \ ,
\ u(w_j^+) = u(w_j^-), j = 1,\ldots, m \} \end{multline}
is a smooth Banach manifold in a neighborhood of any tree disk, since
the linearized operators on the disk components are surjective by
definition, the linearized operators on the one-dimensional components
of $\ti{\Sigma}$ are surjective, and the evaluation maps at the nodes
are transversal.  To explain the last point, for each segment $S_i =
[\eps_i^-,\eps_i^+]$ we can find a perturbation $f_{n,v(\mdeg)}$ whose
image under the linearized operator on $S_i$ has arbitrary end points
in $T_{u(\eps_i)^\pm} L$, and hence the image of $T_{u(\eps_i^-)} L
\oplus T_{u(\eps_i^+)}L \oplus \g$ in $T_{u(\eps_i^-)} \ti{L} \oplus
T_{u(\eps_i^+)} \ti{L}$ is contained in the image of the linearized
operator; the first factor is omitted if the segment is a
semi-infinite edge of the tree.  Keeping in mind that each node is
contained in two components, the extra factor of $\g$ is also in the
image by an inductive argument, beginning with the semi-infinite edges
corresponding to the markings $z_1,\ldots,z_n$.  The claim now follows
from Sard-Smale applied to each component with fixed index for
sufficiently large $l$, and a Taubes-type argument to pass to smooth
perturbations.
 \end{proof}

It follows from Theorem \ref{regpert} that we can construct
perturbations $F_{n,v(\mdeg)}$ so that every treed disk is regular, by
induction on $n$.  We denote by
$$MW_n(L)_d = \{ [u] \in MW_n(L) | \Ind(D_u) - \dim(\on{aut}(\Sigma)
\oplus \g) = d  \} $$  
the component of dimension $d$.  


\begin{theorem} \label{mwbound} Suppose that $X$ is aspherical convex and 
compatible perturbations have been chosen so that every holomorphic
treed disk is regular.  Then $\ol{MW}_n(L)_1$ has the structure of a
compact one-manifold with boundary
\begin{equation} \label{boundary} \partial \ol{MW}_n(L)_1 = \cup_{k + i \leq n} 
 \left( {MW}_{n-i + 1}(L)_0 \times_{\ev_k \times \ev_0} {MW}_i(L)_0 \right)
.\end{equation}
\end{theorem} 

\begin{proof}   Compactness is Theorem \ref{compactthm}, 
while smoothness of the strata is Theorem \ref{mwsmooth}.  The local
structure at the boundary follows from a gluing construction: For any
collection $\delta$ of {\em gluing parameters}, that is, real numbers
attached to the edges of zero length and nodes of infinite type,
denote by $\Sigma^\delta$ the treed disk obtained by (in the case of
edges of zero length) removing half-disks in a neighborhood of
attaching nodes and gluing together via the map $z_{j,+} \sim
\delta_j/z_{-,j}$ where $z_{j,\pm}$ are the coordinates on the
half-disks near $w_j^\pm$.  Using cutoff functions one defines an {\em
  approximate treed holomorphic map} $ G^{\on{approx}}_\delta u:
\Sigma \to X $, see the proof of Theorem \ref{bound}.  The implicit
function theorem then produces a solution $(\zeta,\xi)$ to
$\cF_{G^{\on{approx}}_\delta(u)} (\zeta,\xi) =0 $ with $\zeta \in
\Def_\Gamma(\Sigma)$, and we set $G_\delta(u) =
\exp_{G^{\on{approx}}_\delta(u)}(\xi)$, using uniform error estimates,
uniformly bounded right inverse, and uniform quadratic estimates that
combine those used in the construction of the Morse and Fukaya \ainfty
categories given in Abouzaid \cite{ab:ex}.  These estimates involve
the gluing at clean intersection ends which we already described in
the proof of Theorem \ref{bound}.  The estimates for the Morse \ainfty
algebra are special cases of the corresponding estimates for those of
Floer trajectories, in the special case that the Floer trajectory
comes from a Morse trajectory.  Uniformity of the estimates and the
exponential decay Lemma \ref{expdecay} then implies that for every
sequence $u_t: \Sigma_t \to X$ converging to $u$, there exists a $T$
such that for $t > T$, $u_t$ is in the image of the gluing map.
\end{proof} 

\noindent In fact, for a regular system of perturbations,
$\ol{MW}_n(L)$ is a topological manifold with corners, with a
non-canonical $C^1$-structure, by using arguments as in \cite{deform}.
However, this requires a very complicated study of the gluing
construction and we will not need it.

Any relative spin structure on $L$ induces orientations on $MW_n(L)$
by combining the orientations on the moduli spaces of stable treed
disks $MW_n$ (if $n \ge 3$) in \eqref{orienttreed} with those on the
linearized operators $D_u$ described in \cite{orient}, see also
Katic-Milinkovic \cite{ka:coh}.  Let $CQF(L)$ denote the formal sum
over critical points (that is, the same vector space appearing in
Morse complex with $\Lambda$ coefficients)
$$ CQF(L) = \bigoplus_j CQF^j(L), \quad CQF^j(L) = \bigoplus _{x \in
  \crit(F)^j} \Lambda \bra{x} $$
where $\crit(F)^j \subset \crit(F)$ denotes the subset of critical
points of index $j$.  We denote by $1_L \in CQF^0(L)$ the class of the
critical point $x_{\min}$ of degree zero (or more generally, the sum
of critical points of degree zero, if the Morse function $F$ has more
than one.)  For any disk $u: (D,\partial D) \to (X,L)$, we denote by
$\Hol_L(u) = \Hol_L([\partial u]) \in \Lambda$ the holonomy of the
flat line bundle $\Lambda$ around the boundary $\partial u$ of $u$.
Denote by $ \ol{MW}_{n}(x_0,\ldots,x_n) \subset \ol{MW}_{n}(L)$ the
subset with markings mapping to $x_0,\ldots, x_n$.

\begin{theorem} \label{ainfty}  For any choice of perturbation system
so that all treed holomorphic disks are regular, the formula
$$ \mu_n(\bra{x_1},\ldots, \bra{x_n}) = \sum_{[u] \in
  \ol{MW}_{n}(x_0,\ldots,x_n)_0 } (-1)^{\heartsuit} \eps(u) q^{A(u)}
\Hol_L(u) \bra{x_0} $$
where 
$ \heartsuit = {\sum_{i=1}^n i|x_i|} $
define the structure of a $\Z_2$-graded strictly $A_\infty$-algebra on
$CQF(L)$.
\end{theorem} 

\begin{proof} The \ainfty axiom up to signs follows from Theorem \ref{mwbound},
the additivity of the area and the multiplicativity of the holonomies.
The signs are painfully justified as follows.  The determinant line
$\det(D_u)$ of the linearized operator for the holomorphic disks
admits a canonical orientation, by deforming to a connect sum of
holomorphic spheres and constant disks as in \cite{fooo},
\cite{orient}.  Consider the gluing map
$$MW_{i}(y,x_{j+1},\ldots, x_{j+i})_0 \times
MW_{n-i+1}(x_0,x_1,\ldots,y,\ldots,x_n)_0 
\to MW_n(x_0,\ldots,x_n)_1 .$$
$$TMW_{i}(y,x_{j+1},\ldots, x_{j+i}) \to TMW_{i} \oplus 
TL \oplus
T_yL^+ \oplus T_{x_{j+1}}L^- \oplus \ldots \oplus T_{x_{j+i}}L^- $$
plus a complex vector space.  Here $T_{x_i}L^-$ is the negative part
of the tangent space $T_{x_i}L$ with respect to the Hessian of $F$.
Similarly the orientation on
the second factor 
 is determined by an
isomorphism
$$ TWM_{n-i+1}(x_0,x_1,\ldots,y,\ldots,x_n) \to TMW_{n-i+1} \oplus 
TL
\oplus TL_{x_0}^+ \oplus TL_{x_1}^- \oplus \ldots \oplus TL_{x_n}^-
$$
plus a complex vector space.  A sign arises from a switch of
$TMW_{n-i+1}$ with $TL \oplus \ldots \oplus
T_{x_{j+i}}L^-$, that is, a sign of $(-1)$ to the power $i(n-i)$,
plus the sign arising from a switch of $TL \oplus T_yL^+ \oplus
T_{x_{j+1}}L^- \oplus \ldots \oplus T_{x_{j+i}}L^- $ with $TL
\oplus TL_{x_0}^+ \oplus TL_{x_1}^- \oplus \ldots \oplus TL_{x_j}^-$,
giving $(-1)$ to the power $i \left(|y| + \sum_{k=1 }^j |x_j|\right)
$, for a total of $(-1)$ to the power $ i \left(\sum_{k=j+i+1 }^n
|x_k|\right) $.  The gluing map $MW_i \times MW_{n-i+1} \to MW_n$ is
the same as that of $M_i \times M_{n-i+1} \to M_n$ by construction,
which has sign $(-1)$ to the power $ij + 1 - j - i $, see
\eqref{gluemap}.  
Comparing the contributions from $(-1)^\heartsuit$
with an overall sign of $(-1)^\square$, where $\square = \sum_{k=1}^n
k |x_k|$, contributes $(-1)$ to the power
\begin{multline} \sum_{k=1}^n k | x_k|  + 
\sum_{k = 1}^i k |x_{j+k}| + \sum_{k=1}^j 
k |x_k| + \sum_{k=1}^{n-i-j} (j + k ) |x_{j+i+k}| + (j+1) |y| \\
\equiv_2  j(|y| +
i) + (i-1) (\sum_{k=1}^{n-i-j} |x_{j+i+k}|) + (j+1) |y| \\
\equiv_2 |y| + ji + (i-1)( \sum_{k=1}^{n-i-j} |x_{j+i+k}|)
 \end{multline}
while the sign in the \ainfty axiom contributes $\sum_{k=1}^j (|x_k| -
1)$.  Combining the signs one obtains in total
\begin{multline} 
 i \left( \sum_{k=i+j+1 }^n |x_k|\right)
 + ij + 1 - j - i +
 ji 
  + (i-1) \left(\sum_{k=i+j+1}^{n} |x_{k}| \right) + |y| +  
\sum_{k=1}^j (|x_k| - 1)
 \\
= \sum_{k=1}^j (|x_k| - 1) + |y|  
 + 1 + \sum_{k=i+j+1 }^n |x_k| 
\equiv_2 1 + \sum_{k=1}^n |x_k|
  \end{multline}
which is independent of $i,j$.   The \ainfty-associativity relation
follows. 
 \end{proof} 

If the line bundle is defined over $\Lambda_0$ resp. is trivial then
the \ainfty algebra is defined over $\Lambda_0$ resp. $\Z$.  The
combined gradient trees/holomorphic disks have appeared in many
places, for example the Piunikhin-Salamon-Schwarz \cite{pss},
Biran-Cornea \cite{bc:ql}, Albers \cite{alb:lag}, and Seidel
\cite{seidel:genustwo}.  See also the homological perturbation lemma
in Kontsevich-Soibelman \cite{ks:tf}, which involves a sum over trees.

\subsection{Quasimap Floer cohomology} 

We wish to achieve units in the quasimap Fukaya algebra, in a suitable
weak sense.  Let $1_L = \bra{x_{\min}}$ denote the generator
corresponding to the critical point of index zero $x_{\min}$ of $F:L
\to \R$, assuming it is unique.  We denote by ${\M\W}_\Gamma(L)$ the
moduli space of stable treed holomorphic disks of combinatorial type
$\Gamma$ and by $\PP^*_\Gamma(L)$ the space of perturbation data for
which every element of ${\M\W}_\Gamma(L)$ is regular.

\begin{proposition}   Suppose that $J$ is such that every 
non-trivial stable holomorphic disk in $L$ is regular and has positive
Maslov index.  Then for sufficiently small perturbation data,
\label{mu0}
$$ \mu_0(1) = \sum_{ I(u) = 2} q^{A(u)} \Hol_L(u) 1_L $$
where the sum is over equivalence classes of Maslov index two
holomorphic disks with boundary in $\ti{L}$ such that the boundary of
$u$ maps to a generic point under the projection $X \to X \qu G$.
Furthermore, the element $1_L$ satisfies
$$(-1)^{|c|} \mu_2(1_L,c) = \mu_2(c,1_L) = c, \quad \forall c\in
CF(L). $$
\end{proposition}  

\begin{proof}   
The assertion on $\mu_0(1)$ follows from a dimension count and the
assumption of no disks of non-positive Maslov index.  By the
transversality assumption $\ol{M}_1(x_{\min})$ consists of a gradient
trajectory to $x_{\min}$ attached to a holomorphic disk of index two.
Since $x_{\min}$ is index zero, the gradient trajectory must be length
zero, so that the disk actually passes through $x_{\min}$; the
perturbation has the effect of replacing this with a generic point.

To prove the second assertion, let $[u:C \to X] \in \M\W_\Gamma(L)$ be
an isolated element with limit along an incoming semi-infinite edge
mapping to the maximum $x_0 \in \crit(f)$, the other incoming edge
mapping to $x_1 \in \crit(f)$ and the outgoing edge mapping to $x_2
\in \crit(f)$.  The domain $C$ of $u$ is obtained by replacing the
vertices $v$ of a tree $T$ with disks $D(v)$.  By assumption all
non-constant stable disks are assumed regular and positive Maslov
index.  By transversality, all edges in the domain of $u$ have finite
length.  We claim that $C$ consists of a single disk $D(v)$ on which
$u$ is constant, and three semi-infinite edges attached to the
boundary.  We first show that there are no disks $D(v)$ in the domain
of $u$ with a single special point.  Indeed if there existed such a
disk $D(v) \subset C$ then by varying the length of the connecting
edge $d$ to $D(v)$ one would obtain a one-dimensional family of
configurations as follows.  Let $t(d), h(d) \in C$ denote the
endpoints of $d$.  We have $t(d) = \varphi_{\ell(d)} h(d) $ where
$\ell(d)$ is the length of $d$ and $\phi_{\ell(d)}$ is the time
$\ell(d)$ perturbed gradient flow.  By transversality the moduli space
$\M_1^2(L)$ of holomorphic disks with boundary on $L$ and one marked
point on the boundary of Maslov index $2$ is smooth and evaluation on
the boundary is a local diffeomorphism.  Replacing $\ell(d)$ by a
small perturbation and adjusting $t(d)$ produces a one-dimensional
family in $\M\W_\Gamma(L)$ containing $[u]$ which is a contradiction.
Thus disks with only one special point are involved only in the
definition of $\mu_0$, and not in the composition maps $\mu_n, n > 0$.

Secondly, any disk $D(v)$ meeting the semi-infinite edge labelled
$x_{\min}$ must be constant.  Let $z \in D(v)$ be the point on the boundary
connecting to the semi-infinite edge labelled $x_{\min}$.  We have $\lim_{t
  \to -\infty} \varphi_t(u(z)) =x_{\min}$ which is an open condition.  If
$u$ is non-constant on $D(v)$ one could vary $z$ to obtain a
one-parameter family of elements of $\M\W_\Gamma(L)$ containing $[u]$,
which is a contradiction.

Thirdly, the domain has a single disk component.  Let $D(v)$ denote
the first disk component attached to the semi-infinite edge labelled
$x_{\min}$.  By the previous paragraph, $u$ is constant on $D(v)$ with
value $u(D(v))$. After replacing $D(v)$ with a point, one obtains a
perturbed treed holomorphic disk with only one incoming edge with
$\Ind(D_u) - \dim(\aut(C)) = -1 $ 
where $D_u$ is the associated Fredholm operator and $\aut(C)$ the
space of infinitesimal automorphisms of the domain.  By assumption we
have chosen $(f,g)$ so that the moduli space of holomorphic treed
disks with a single incoming edge is regular, that is, there are no
non-constant holomorphic treed disks $u:C \to X$ with $\Ind(D_u) -
\dim(\aut(C)) = -1$.  The same is true for perturbations in a
neighborhood of $(f,g)$, in particular, for $(f_\Gamma,g_\Gamma)$ in a
neighborhood of $(f,g)$, there are no such configurations.  Hence all
disks in $C$ are constant.  Since each constant disk between the
semi-infinite edges labelled $x_1$ and $x_2$ corresponds to a
semi-infinite edge different from the one labelled $x_1$ and $x_2$,
and there are only three semi-infinite edges, there must be a single
disk component.

Collapsing the disk $D(v)$ and forgetting the incoming trajectory from
$x_{\min}$, one obtains a perturbed gradient trajectory $u'$ joining $x_1$
to $x_2$ with a distinguished point $z \in C$.  For any critical
points $x_1,x_2$ of equal index, there is a parametrized Morse
trajectory for $(f,g)$ of index zero connecting $x_1,x_2$ if and only
if $x_1,x_2$ are equal.  The same holds after any sufficiently small
perturbation $f_\Gamma,g_\Gamma$.  It follows that $\M\W_\Gamma(L)$ is
a point if $x_1=x_2$, and is empty otherwise.  The orientation is
induced by the orientation on the moduli space of parametrized
trajectories from $x_1$ to $x_2$, which is positive.  Thus
$$ (-1)^{|x_1|} \mu_2(\bra{x_{\min}},\bra{x_1}) =
\mu_2(\bra{x_1},\bra{x_{\min}}) = \bra{x_1}, \quad \forall x_1 \in \crit(f)
.$$
The statement for Floer cochains follows by taking linear
combinations.
\end{proof} 

\begin{remark} 
If one wishes an \ainfty algebra with strict units may be obtained by
the method of Fukaya-Oh-Ohta-Ono \cite{fooo}, see also Ganatra
\cite{ganatra:duality} and Sheridan \cite{sh:hmsfano} as we now
explain.  Let $A$ be an \ainfty algebra over $\Lambda$.  A {\em
  homotopy unit} for $A$ is an \ainfty structure on the
$\Lambda$-module $A^+$ generated by $A$ and additional generators
$1_L^-,1_L^+$, that is,
$ A^+ = A \oplus \Lambda 1_L^-[1] \oplus \Lambda 1_L^+ $
so that the \ainfty structure maps $\mu_n, n \ge 1$ coincide with the
previous \ainfty structure on $A$ and in addition the relations on the
maps $\mu_n, n \ge 0$ hold as follows:
\begin{eqnarray*} 
\mu_1(1_L^-) &=& 1_L^+ - 1_L \\
\mu_1(1_L^+) &=& 0 \\
(-1)^{|a|}
 \mu_2(1_L^+,a) &=& \mu_2(a,1_L^+) = a \\ 
\mu_n(\ldots, 1_L^+,\ldots ) &=& 0 \ \text{for} \ 
n \ge 3 \end{eqnarray*}
for all homogeneous elements $a \in A$.  In particular $A^+$ is a
strictly unital \ainfty category with strict unit $1_L^+$.  

A homotopy unit on the quasimap Fukaya \ainfty algebra may be obtained
following Fukaya-Oh-Ohta-Ono \cite[Remark 10.3]{fooo}, Ganatra
\cite{ganatra:duality} and Sheridan \cite{sh:hmsfano} by homotoping
the perturbation data for insertions involving the unit to
perturbation data for which one has forgetful maps.  Since the
composition maps involving the generators of $CF(L)$ and the strict
unit only are already determined, the problem is to define the
composition maps involving at least one input $1_L^-$.  These are
defined by counts of {\em weighted treed disks}, by which mean a treed
disks equipped with a subset $I \subset \{ 1,\ldots, n \}$ of the
semi-infinite edges corresponding to the inputs marked $1_L^-$, and for
each $i \in I$ a {\em weight} $\rho_i \in [0,1]$.  A weighted treed
disk is stable if the underlying treed disk is stable.  A collection
$(f_\Gamma^+, g_\Gamma^+)$ of perturbation data for weighted treed
disks is {\em coherent} if
\begin{enumerate} 
\item in a neighborhood of any boundary stratum $ \T_{\Gamma'} \subset
  \T_\Gamma$, $f_\Gamma^+,g_\Gamma^+$ are pulled back from the
  boundary perturbation data $f_{\Gamma'}^+,g_{\Gamma'}^+$ under the
  gluing construction;
\item whenever a weight parameter $\rho_i$ is equal to $0$,
  $f_{\Gamma}^+, g_{\Gamma}^+$ are pulled back under the forgetful map
  forgetting the $i$-th semi-infinite edge and stabilizing; and
\item whenever a weight parameter $\rho_i$ is equal to $1$, then the
  perturbation data $f_{\Gamma}^+,g_{\Gamma}^+$ is equal to the
  perturbation data $f_\Gamma,g_\Gamma$ used to define the composition
  maps for $CF(L)$.
\end{enumerate} 
Given coherent collections of perturbation data, counting isolated
points in the moduli spaces of stable weighted treed holomorphic disks
defines the higher composition maps for inputs involving $1_L^-$.  In
particular, the \ainfty relations involving $\mu_1(1_L^-) = 1_L^+ - 1_L$
hold for the following reason: Weighted disks with some $\rho_i \in \{
0, 1 \}$ are boundary components of the moduli space of weighted disks
for which the weight $\rho_i$ is allowed to vary freely. When $\rho_i
= 0$, the perturbation data is pulled back under the forgetful map.
So the higher compositions vanish after insertion of $1_L^+$ while for
$\rho_i = 1$ one obtains the corresponding higher composition map with
insertion of $1_L$.  The \ainfty relations are proved in the same way as
before and yield a strictly unital \ainfty algebra $CF^+(L)$.  This ends
the remark.
\end{remark}

The \ainfty algebra $CQF(L)$ might be called the {\em gauged} or {\em
  equivariant Fukaya} algebra of $\ti{L}$.  However, the definition of
the structure maps do not involve any connection.  Note that the
vortex equations are not conformally invariant, that is, they depend
on a choice of area form.  In the next section we investigate the
\ainfty bimodule associated to a pair of Lagrangians given by counting
holomorphic strips.  Since the automorphism group of a strip {\em
  does} preserve the standard area form, one should expect an \ainfty
bimodule defined using the vortex equations.  The rigorous
construction of such a bimodule, however, has not been carried out.

Because of Proposition \ref{mu0} we have $\mu_0(1) = w1_L$ for some $w
\in \Lambda_0$ and the square of the first structure map $\mu_1:
CQF(L) \to CQF(L)$ satisfies
$$ (\mu_1)^2(a) = - \mu_2(a,w 1_L) + (-1)^{|a|} \mu_2(w 1_L,a) = w( -
(-1)^{|a|} a + (-1)^{|a|} a) = 0 .$$
It follows that the {\em quasimap Floer cohomology} 
$$HQF(L) := H(\mu_1) $$ 
is well-defined.

The higher composition maps have an obvious expansion by energy, such
that the leading order term is the higher composition map in the
\ainfty Morse algebra defined by the perturbations $F_{n,v}$.  In
particular the leading order term in $\mu_1$ is the Morse-Smale-Witten
operator counting isolated gradient flow lines, and one has a
Morse-to-Floer spectral sequence (compare e.g. Oh \cite{oh:fc}, Fukaya
et al \cite{fooo}, and especially  Buhovsky \cite{buh:mul}, who uses the
same combined Morse-Fukaya framework.)

\begin{proposition}  \label{spectral} 
Suppose that perturbations have been chosen above so that every
holomorphic treed quasidisks are regular.  There is a spectral
sequence $(E^j,\mu_1^j)_{j \ge 1}$ converging to $HQF(L,\Lambda_0)$
with first page $E^1 = HM(L,\Lambda_0)$ the Morse cohomology.
\end{proposition}  

\begin{proof}  The holomorphic treed disks with zero energy
have all holomorphic disks constant.  These can be isolated only if
each disk has exactly three markings, in which case they can be
collapsed to stable gradient trees.  If any disk occurs then the
resulting tree has at least three semiinfinite edges.  Hence no disks
occur in the zero energy contributions to $\mu_1$, which are therefore
those of the Morse differential.  Filtering the chain complex by
energy intervals
$$ CQF(L,\Lambda_0) = \bigcup_{n \in \Z_{\ge 0}} CQF(L)_{\geq n\delta},
\quad CQF(L)_{\geq n \delta} = \bigoplus_{x \in \crit(F)} q^{n\delta}
\Lambda_0 \bra{x}
 $$
for sufficiently small $\delta >0 $ and taking the associated spectral
sequence gives the result.
\end{proof} 

Note that here we have {\em assumed} that all stable holomorphic
quasidisks are regular, and explicitly constructed perturbations
achieving transversality inductively using moduli spaces of
not-necessarily-stable metric ribbon trees.

\subsection{The divisor equation} 

A feature of the combined Morse-Fukaya moduli spaces is that there is
{\em not} a forgetful morphism of the form $ f_i: \ol{MW}_{n}(L) \to
\ol{MW}_{n-1}(L) $ forgetting the $i$-th marking.  If there was such a
morphism, then one would have an analog of the divisor equation in
Gromov-Witten theory, as explained by Cho \cite{cho:products}: if $b
\in CQF(L)$ represents a cycle of codimension one in Morse cohomology
then
$$\sum_{i=0}^{n} \mu_{n+1}(a_1,\ldots,a_i,b,a_{i+1},\ldots,a_n;\mdeg)
\stackrel{?}{=}  \lan [\partial \mdeg],[b] \ran \mu_n(a_1,\ldots,a_n;\mdeg) $$
where $\mu_n( \cdots; \mdeg)$ is the contribution to $\mu_n(\cdot)$
from disks of class $\mdeg \in \pi_2(X,\ti{L})$, and $[\partial \mdeg]
\in H_1(L)$ is the boundary homology class of $\mdeg$.  Indeed, given
a holomorphic disk with the first $n$ markings mapping to
$a_1,\ldots,a_n$ and class $\mdeg$, the boundary of the disk traces
out a circle of class $[\partial \mdeg] \in H_1(L)$, and so there are
$ \lan [\partial \mdeg], [b] \ran$ possibilities for the placement of
an additional marked point on the disk so that it lies on the cycle
given by $b$. 

Unfortunately the equations on the segments are {\em perturbed}
gradient flows, and it is not possible to choose the perturbation so
that it vanishes on every segment that might be collapsed after
forgetting a marked point and stabilization.  So there is no forgetful
map, for the perturbation system constructed above.  However, we have
the following special case of the divisor equation:

\begin{proposition}   \label{divisor}  Suppose that all non-trivial 
stable holomorphic treed disks are regular and have positive Maslov
index.  For any Morse cocycle $b \in CQF^1(L)$ and $\mdeg \in
\pi_2(X,\ti{L})$ with $I(\mdeg)= 2$, we have $ \mu_{1}(b;\mdeg) = \lan
   [b],[\partial \mdeg] \ran \mu_0(1;\mdeg) .$
\end{proposition} 

\begin{proof}   The only contributing configuration is that of 
a single disk of index two and two gradient trajectories: all
configurations with disks whose degrees sum to more than two vanish
for reasons of degree, while the underlying tree has only one vertex
since the Morse product of $b$ and a point vanishes, so there are no
contributing configurations with constant disks.  Suppose that $b =
\sum_{x \in \crit(F)^1} n_x \bra{x}$ for some coefficients $n_x$.  The
contributions to $\mu_1(b;\mdeg)$ correspond to choices for the
position of the point on the boundary of the disk component connecting
via a gradient trajectory to some $x$ appearing in $b$, that is, with
$n_x \neq 0$, since by assumption the moduli spaces are regular.

To check that the sign is the expected one, note that for any $x_1$
such that $W_{x_1}^-$ meets $u(\partial D)$, the orientation on
$\ol{MW}_1(x_0,x_1)$ is $\pm 1$ depending on whether $\partial_\theta
u (e^{i \theta})$ maps positively or negatively onto $T^-_{x_1}
\ti{L}$: The orientation on $T_u MW(x_0,x_1)_0$ is induced from an
identification of the determinant line with the determinant line of
$\ominus \R \oplus \C \ominus T^-_{x_1} \ti{L}$ where $\C$ is a factor
arising from the complex linear kernel of a Cauchy-Riemann operator on
a sphere, with value prescribed at a marked point.  If $u$ is an
element of the kernel of the Cauchy-Riemann operator vanishing at the
marking, then the orientation on the complex line is determined by
$\partial_v u \wedge \partial_w u$, where $v,w$ are vector fields on
the sphere $\C \cup \{ \infty \}$ such that $v \wedge w$ is positive
on $\C$. This property transfers via the gluing construction, so that
the orientation is given as follows.  Consider an identification $D -
\{ -1,1 \} \to \R \times [0,1]$ mapping $\{-1,1 \}$ to the point at $
\mp \infty$.  Translation on the strip induces vector fields $v,w \in
\Vect(D)$.  Since $v \wedge w$ is positive, $\partial_v u \wedge
\partial_w u$ is the orientation on $\ev(1)^{-1}(0) \cap \ker(D_u)
\cong \C$.  Since the translation $\R$ identifies with span of
$\partial_v u$, the sign of the contribution is given by $1$
resp. $-1$ if $- \partial_\theta u ( e^{i\theta}) |_{\theta = \pi}$
gives the orientation resp. minus the orientation of $T^-_{x_1}
\ti{L}$.  Combining with the factor $(-1)^{|x_1|} = - 1$ in the
definition of $\mu_1$ gives the claim.  Thus the total contribution
from each such disk is the sum of these contributions times $n_x$, for
each $x$ with $n_x \neq 0$, or equivalently, $\lan [b], [\partial
  \mdeg] \ran q^{A(u)} \Hol_L(u)$.  Summing over such disks proves the
Proposition.
\end{proof} 

Combining Proposition \ref{mu0} with Proposition \ref{divisor} we
obtain:

\begin{corollary}  \label{mu1}
 Suppose that the moduli space of treed quasidisks is regular and
 $\beta \in CQF^1(L)$ is a Morse cocycle.  Then
$$ \mu_1(\beta) = \sum_{ I(u) = 2} 
\lan [\beta] ,[\partial \mdeg] \ran q^{A(u)} \Hol_L(u) 1_L $$
where the sum is over equivalence classes of index two holomorphic
disks with boundary in $\ti{L}$ and a boundary point mapping to a
generic point in $L$.
\end{corollary} 

Suppose that $L$ is a torus and consider $H^1(L,\Lambda_0)$ as
parametrizing a family of brane structures on a compact oriented
Lagrangian $L$ equipped with the standard spin structure.  For any $b
\in H^1(L,\Lambda_0)$ we denote by $\mu^b_n$ the structure
coefficients for $L^b$, that is, the Lagrangian $L$ with brane
structure given by $b$.  The {\em potential} for $L$ is the function
$$ W: H^1(L,\Lambda_0) \to \Lambda_0, \quad \mu_0^b(1) = W(b) 1_L .$$
Any $\beta$ in $CQF^1(L)$ on which the Morse differential vanishes
defines a class $[\beta] \in H^1(L,\Lambda_0)$.

\begin{proposition} \label{partials}
For  $\beta \in CQF^1(L)$ with $\mu_1^b(\beta;0) = 0$
we have $\mu_1^b(\beta) = \partial_{[\beta]} W(b)$.
\end{proposition} 

\begin{proof} By Proposition \ref{mu1}, Corollary \ref{mu0}, and the formula 
$\Hol_L(u) = \exp \lan [b],[\partial \mdeg] \ran. $
\end{proof}

\subsection{\ainfty homotopy invariance}

The homotopy type of the quasimap Fukaya algebra defined above is
independent of all choices.  The argument uses moduli spaces of {\em
  quilted treed disks}, which are a particular realization of
Stasheff's {\em multiplihedron} $L_n$ \cite{st:ho}.  This is a cell
complex whose vertices correspond to total bracketings of
$x_1,\ldots,x_n$, together with the insertion of expressions $f( \cdot
)$ so that every $x_j$ is contained in an argument of some $f$.  For
example, $L_2$ is an interval with vertices $f(x_1)f(x_2)$ and
$f(x_1x_2)$.  A geometric realization of this polytope can be given as
follows: A {\em quilted metric ribbon tree} is a rooted metric ribbon
tree 
$$\Gamma = (V(\Gamma),E(\Gamma), O(\Gamma), l: E(\Gamma) \to [0,\infty])$$
together with a subset $V^{\on{col}}(\Gamma)$ of the vertices
$V(\Gamma)$, satisfying the condition the length of the finite part of
the path $P(z_0,v)$ from the root vertex $z_0$ to any {\em colored
  vertex} $v \in V^{\on{col}}$ given by $\prod_{e \in P(z_0,v)} l(e)$
is independent of the choice of colored vertex $v \in
V^{\on{col}}(\Gamma)$, see Ma'u-Woodward \cite{mau:mult}.  The set of
finite resp. semiinfinite edges is denoted $E_{< \infty}(\Gamma)$
resp. $E_\infty(\Gamma)$; the latter are equipped with a labelling by
integers $0,\ldots, n$.  A quilted tree is {\em stable} if each
colored vertex has valence at least two, any non-colored vertex has
valence at least three or connects two edges of infinite length, and
each edge contains either a trivalent vertex or a colored vertex.

There is a natural notion of convergence of quilted trees, in which
edges whose length approaches zero are contracted and edges whose
lengths go to zero are replaced by broken edges.  Let $\ol{W}_{n,1}$
denote the moduli space of quilted metric ribbon trees; this is
naturally homeomorphic to $L_n$, see \cite{mau:mult}, although not
isomorphic as a cell complex.  

There is a different realization of the multiplihedron given in
Ma'u-Woodward \cite{mau:mult} which gives the correct cell structure.
Namely in \cite{mau:mult} a {\em quilted disk} was defined as a marked
disk $(D,z_0,\ldots, z_n \in \partial D)$ (the points are required to
be in cyclic order) together with a circle $C \subset D$ tangent to
the $0$-th marking $z_0$.  An {\em isomorphism} of quilted disks from
$(D,C,z_0,\ldots,z_n)$ to $(D',C',z_0',\ldots,z_n')$ is an isomorphism
of holomorphic disks $D \to D'$ mapping $C$ to $C'$ and $z_0,\ldots,
z_n$ to $z_0',\ldots, z_n'$.  The moduli space $M_{n,1}$ of quilted
disks admits a compactification $\ol{M}_{n,1}$, isomorphic to $L_n$ as
a cell complex, which allows the interior circle $C$ to ``bubble out''
into the extra disk bubbles, or disk bubbles without interior circles
to form when the points come together.  (The homotopy invariance of
the \ainfty algebra of a Lagrangian is proved in Fukaya et al
\cite{fooo} using a moduli space of {\em weighted} stable disks; we
find the used of quilted stable disks more natural because it
reproduces exactly Stasheff's cell structure, although the quilting
has no geometric meaning in the current paper.)  The open stratum
$M_{n,1}$ may be identified with the set of sequences $0 = w_1 <
\ldots < w_n$; the bubbles form either when the points come together,
in which case a disk bubble forms, or when the markings go to
infinity, in which case one rescales to keep the maximum distance
between the markings constant.  One has a quilted disk bubbles for
each group of markings that come together after re-scaling, but not
before.

There is a combined moduli space $\ol{MW}_{n,1}$ that combines both
objects: A {\em quilted treed disk} $\Sigma$ is given by (i) a quilted
rooted metric ribbon tree (ii) for each non-colored vertex of the tree
with valence at least three, a disk with markings whose number is the
valence of the given vertex (iii) for each colored vertex, a quilted
disk with number of markings equal to the valence that of the colored
vertex.  From this datum one defines a space obtained by attaching the
endpoints of the segments of the quilted metric tree to the marked
points on the disks corresponding to the vertices (except the
uncolored vertices of valence two, which are not attached to disks).
A quilted treed disk is {\em stable} if it has no automorphisms, that
is, each quilted disk resp. unquilted disk has at least two
resp. three marked or nodal points, and in addition any
one-dimensional component is attached to at least one disk.

Thus in particular each disk in a stable quilted treed disk is
connected by a sequence of one resp. two edges, which are of finite
resp. infinite length.  Let $\ol{MW}_{n,1}$ denote the moduli space of
stable quilted treed disks.  See Figure \ref{MWc} for a picture of
$\ol{MW}_{2,1}$.  The quilted disks are those with two shadings; while
the ordinary disks have either light or dark shading depending on
whether they can be connected to the zero-th edge without passing a
colored vertex.  The hashes on the line segments indicate nodes
connecting segments of infinite length, that is, broken segments.
\begin{figure}[ht]
\includegraphics[height=.6in]{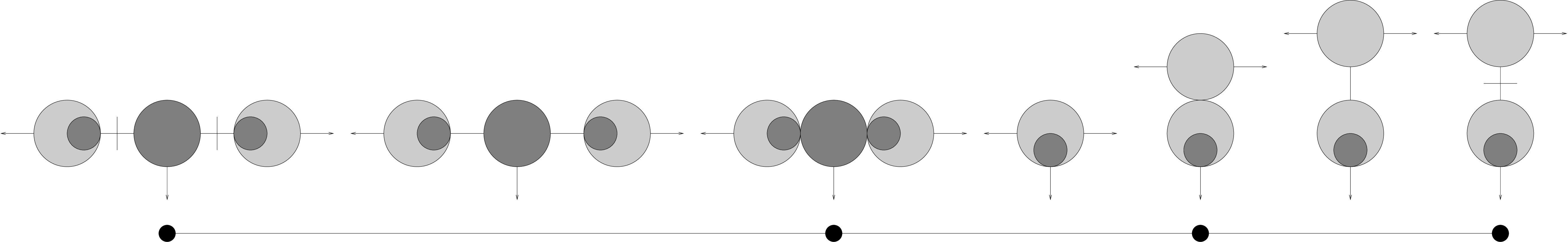}
\caption{Moduli space of stable quilted treed disks}
\label{MWc} 
\end{figure} 
The space $\ol{MW}_{n,1}$ admits two forgetful maps
$ f_M : \ol{MW}_{n,1} \to \ol{M}_{n,1}, \quad 
f_W: \ol{MW}_{n,1} \to \ol{W}_{n,1} $
obtained by forgetting the edges resp. disks and taking the ribbon
structure induced by the order of the markings on the boundary of the
disk.  The product $f_M \times f_W$ defines an injection into
$\ol{M}_{n,1} \times \ol{W}_{n,1}$, defining a topology on
$\ol{MW}_{n,1}$ for which it is compact and a sequence converges iff
the underlying sequence of quilted trees and quilted disks converges.
If a stratum $MW_{n,1,\Gamma}$ is contained in the closure of another
stratum $MW_{n,1,\Gamma'}$, then there is an explicit description of
the normal cone as a real toric singularity, explained in
\cite{mau:mult} for quilted trees or disks.

The cells of $\ol{MW}_{n,1}$ can be oriented by identifying
each with a corresponding cell of $\ol{M}_{n,1}$ times a product of
real lines corresponding to the edges.  The boundary is naturally
isomorphic to a union of moduli spaces with lower numbers of markings:
$$ \partial \ol{MW}_{n,1} \cong \bigcup_{i,j} \left( \ol{MW}_{n-i + 1,1}
\times \ol{MW}_i \right)
\cup \bigcup_{i_1,\ldots, i_r} \left( \ol{MW}_r \times
\prod_{j=1}^r \ol{MW}_{i_j,1} \right) .$$
By construction the sign of the inclusions of boundary strata are the
same as that for the corresponding inclusion of boundary facts of
$\ol{M}_{n,1}$, that is, $(-1)^{ij + n - j - i} $ for the facet of the
first type.  For facets of the second type, the gluing map is
$$ \R \oplus M_r \oplus \bigoplus_{j=1}^r M_{|I_j|,1} \to M_n$$
\begin{multline} \label{gluemap2} (\delta,z_1 = -1, z_2,\ldots, z_{r-1}, z_r= 0, 
(w_{1,j} = 0, w_{2,j},\ldots, w_{|I_j|,j} )_{j=1}^r) \\ \to (-
  \delta^{-1} , - \delta^{-1} + w_{2,1},\ldots, -\delta^{-1} +
  w_{|I_1|,1}, \delta^{-1} z_2,\delta^{-1} z_2 + \delta w_{2,2},
  \ldots , w_{|I_r|,1}, \ldots, \delta w_{|I_r|,r}) \end{multline}
and so changes orientations by $\sum_{j=1}^r (r-j) (|I_j| - 1).$
Finally, there is a {\em universal quilted tree disk} $\ol{UMW}_{n,1}
\to \ol{MW}_{n,1}$ which is a cell complex whose fiber over $\Sigma$
is isomorphic to $\Sigma$.

\begin{theorem}   Suppose that $J \in \J(X)^G, L \subset X \qu G$
are such that every stable disk with boundary in $\ti{L}$ is
regular. Then the \ainfty algebra $CQF(L)$ is independent up to
\ainfty homotopy of the choice of perturbation system used to
construct it.
\end{theorem} 

To prove this, one considers two systems of perturbations
$(B_{*,0},F_{*,0}),(B_{*,1},F_{*,1})$ and extends them to a set of
perturbations resp. connections for the moduli space of the following:

\begin{definition}  A  
{\em holomorphic quilted treed quasidisk} to $X$ is a quilted treed
disk $\Sigma$ together with a map $u: \Sigma \to X$ such that $u$ is
holomorphic on each disk component, and $u$ is a gradient trajectory
for $\ti{F}_{n,v} \in C^\infty(UW_{n,v} \times \ti{L})$ on each
edge, where $v = v(\mdeg)$ depends on the homology class $\mdeg =
\mdeg(u)$ of $u$.  An {\em isomorphism} of holomorphic quilted treed
quasidisks $u_j: \Sigma_j \to X$ is an isomorphism $\phi: \Sigma_0 \to
\Sigma_1$ together with an element $g \in G$ such that $\phi^* u_1 = g
u_0$; in particular this means that $\phi$ preserves the quilting on
each quilted disk.  A holomorphic quilted treed quasidisk $u: \Sigma
\to X$ is {\em stable} if it has no infinitesimal automorphisms (each
disk component on which $u$ is constant has at least three marked or
nodal points, and $u$ is non-constant on each edge of positive length)
and each node connecting two line segments maps to a critical point.
\end{definition}  

\noindent The additional data of the quilting on each quilted disk is
not used to modify the Cauchy-Riemann equation (as opposed to the
situation in \cite{ainfty} where the quilting has a geometric
meaning.)  Let $\ol{MW}_{n,1}(L)$ denote the moduli space of
holomorphic quilted treed quasidisks with boundary in $L$, and
perturbation system $(B_*,F_*)$.  Perturbations can be constructed
inductively in a neighborhood of the boundary components, by taking
the perturbation system $(B_{*,0},F_{*,0})$ above the quilted disk
components (that is, on the components not connected to the root by
unquilted components) and $(B_{*,1},F_{*,1})$ below, and zero on the
edges of small length.  Assuming that the perturbations have been
constructed inductively, the boundary components of $\ol{MW}_{n,1}(L)$
are the combinatorial types with at least one segment of infinite
length.  (There are also ``fake boundary component'' with edges of
zero length, but these are not part of the ``true boundary''.)  In
particular, the boundary components of dimension one consist of two
types: configurations where a collection of quilted treed disks has
broken off, or an unquilted treed disk has broken off as shown in
Figure \ref{MWc}:
\begin{multline}  \label{morbound}
 \partial \ol{MW}_{d,1}(L) = \bigcup_{i_1 + \ldots + i_r = d}
 \ol{MW}_{r,1}(L) \times_{\crit(F)^r} \prod_{j=1}^r
 \ol{MW}_{i_j}(L;B_{*,1},F_{*,1}) \\ \cup \bigcup_{i+k \leq d}
 \ol{MW}_{i}(L;B_{*,0},F_{*,0}) \times_{\crit(F)} \ol{MW}_{d - i +
   1,1}(L) .\end{multline}


\begin{theorem} \label{morreg}  If every stable holomorphic disk is regular then there
exist compatible systems of perturbations so that every quilted treed
quasidisk is regular.
\end{theorem}

\begin{proof} 
This is similar to the argument given for the unquilted case above in
Theorem \ref{regpert}, and is left to the reader.
\end{proof} 

Assuming that a perturbation system has been chosen as in Theorem
\ref{morreg} define
\begin{multline}
 \phi_n: CQF(L;B_{*,0},F_{*,0})^{\otimes n} \to CQF(L;B_{*,1},F_{*,1})
 \\ (\bra{x_1},\ldots,\bra{x_n}) \mapsto \sum_{[u] \in
   \ol{MW}_{n,1}(L;x_0,\ldots,x_n)_0} (-1)^{\heartsuit} \eps(u)
 q^{A(u)} \Hol_L(u) \bra{x_0} .\end{multline}

\begin{theorem} 
$\phi = (\phi_n)_{n \ge 0}$ is an \ainfty morphism.
\end{theorem} 

\begin{proof}    For $\Z_2$ coefficients, this follows from 
the description of the boundary in \eqref{morbound}.  The signs for
the terms \eqref{ainftymorphism} of the first type is similar to those
for the \ainfty axiom and will be omitted.  For terms of the second
type we need to determine the sign of the isomorphism
$$ \R \oplus T \ol{MW}_r(y_0,\ldots,y_r) \oplus \bigoplus_{j=1}^r T
\ol{MW}_{|I_j|,1}(y_i,x_{I_j}) \to T \ol{MW}_n(y_0,x_1,\ldots,x_n) .$$
The former is determined by an isomorphism with
$$ \R \oplus TL \oplus T_{y_0}^+ \oplus T \ol{MW}_r \oplus T_{y_1}^- \ldots
\oplus T_{y_r}^- \oplus \bigoplus_{j = 1}^r \left( TL \oplus T_{y_j}^+
\oplus T \ol{MW}_{|I_j|,1} \oplus \bigoplus_{k \in I_j} T_{x_k}^-
\right) $$
(where $T_{y_0}^+$ denotes $T_{y_0}^+L$ etc.)  except that this
differs by signs
\begin{equation} \label{clubs}  |y_0|(r-2) + \sum_{j=1}^r (|I_j|-1)|y_j|\end{equation}
from our previous convention.  Equivalently since each moduli space
has formal dimension zero the determinant line is given by
$$ \R \oplus TL \oplus T_{y_0}^+ \oplus T \ol{MW}_r \oplus \bigoplus_j
\left( T_{y_j}^- \oplus TL \oplus T_{y_j}^+ \oplus T \ol{MW}_{|I_j|,1}
\oplus \bigoplus_{k \in I_j} T_{x_k}^- \right) .$$
Using $T_{y_j}^- \oplus T_{y_j}^+ \cong TL$ these three factors
disappear.
%
Moving each $T \ol{MW}_{|I_j|,1}$ past $T_{x_k}^-$ for $k < \min I_j$ 
contributes a number of signs $(|I_j| - 1) \sum_{k < \min I_j} |x_k| $.
Moving $TL \oplus T_{y_0}^+$ past $\R$ gives $|y_0|$ additional signs. 
Finally we have a contribution from the signs in the definition of
$\phi_{|I_j|}$ and the sign from the definition of $\mu_r$
$$ \sum_{j = 1}^r \sum_{i=1}^{|I_j|} i|x_i| + \sum_{j=1}^r j |y_j|  $$
and the overall sign used in the proof the \ainfty axiom,
$1 + \sum_{k=1}^n (k + 1) |x_k|  .$
The gluing map has sign \eqref{gluemap2}.  In total, the number of
signs is
\begin{multline} \label{total}
\sum_{j=1}^r \left( (|I_j| - 1) \sum_{k
  < \min I_j} |x_k| \right) + |y_0| + \sum_{j=1}^r \sum_{i=1}^{|I_j|}
i|x_i| \\
+ \sum_{j=1}^r j |y_j| + 1 + \sum_{k=1}^n (k + 1) |x_k| 
+ \sum_{j=1}^r (r-j) (|I_j| - 1).\end{multline}
The difference is 
$$\sum_{k=1}^n k |x_k| - \sum_{j=1}^r
\sum_{i=1}^{|I_j|} i|x_i| = \sum_{j=1}^r \sum_{i \in I_j} (|I_1| +
\ldots + |I_{j-1}|) |x_i| ,$$
so \eqref{total} equals
\begin{multline}
\sum_{j=1}^r (|I_j| - 1) \sum_{k
  < \min I_j} |x_k| + |y_0| + \sum_{j=1}^r |I_j| \left( \sum_{k >
  \max(I_j)} |x_k| \right) + \\
1 + \sum_{k=1}^n |x_k| + \sum_{j=1}^r j |y_j| 
+ \sum_{j=1}^r (r-j) (|I_j| - 1)
.\end{multline}
Now 
$$\sum_{j=1}^r |I_j| \sum_{k \notin I_j} |x_k| = \sum_{j=1}^r |I_j|
(\sum_k |x_k| - \sum_{k \in I_j} |x_k|) = 
n( |y_0| +
n- 2) - \sum_{j=1}^r |I_j| ( |y_j| + |I_j| - 1) .$$
So \eqref{total} equals
$$ \left( \sum_{j=1}^r \sum_{k < \min I_j} |x_k| \right) + |y_0| + 1 +
(n-2 + |y_0|) + \left( \sum_{j=1}^r j |y_j| \right) + n( |y_0| + n- 2)$$
$$ - \sum_{j=1}^r \left( |I_j| ( |y_j| + |I_j| - 1) \right)
+ \sum_{j=1}^r (r-j) (|I_j| - 1).$$
The first term is 
$$ \sum_{j=1}^r \sum_{k \in I_j} (r-j)|x_k| = 
\sum_{j=1}^r (|y_j| +
|I_j| - 1)(r-j) = 
\sum_{j=1}^r (|y_j|r - |y_j|j  +
(|I_j| - 1)(r-j)
.$$
So \eqref{total} equals
\begin{multline} 
\left( \sum_{j=1}^r |y_j|r - |y_j|j + (|I_j| - 1)(r-j) \right) +
|y_0| + 1 + (n-2 + |y_0|) + \left( \sum_{j=1}^r j |y_j| \right) +
\\  n(
|y_0| + n- 2) - \sum_{j=1}^r \left( |I_j| ( |y_j| + |I_j| - 1)
\right) + \sum_{j=1}^r (r-j) (|I_j| - 1)
.\end{multline}
Cancelling and simplifying gives
$$ (|y_0| + r)r + n |y_0| + 1 - \sum_{j=1}^r |I_j| (
|y_j| + |I_j| - 1)  $$
which is congruent mod $2$ to 
$ (n-r)|y_0| + (1-r) - \sum_{j=1}^r |I_j| 
|y_j|  .$
Combining with \eqref{clubs} and switching $T_{y_j}^+ \oplus TL$ with
$MW_n$ we obtain
$$ (n-r)|y_0| + (1-r) - \sum_{j=1}^r |I_j| |y_j| + |y_0|(r-2) +
\sum_{j=1}^r (|I_j|-1)|y_j| + (n-1)|y_0|$$
which equals $ (1-r) + |y_0| + \sum_{j=1}^r |y_j| = (1-r) + (r- 1) = 0
$ as claimed.
\end{proof} 

If two perturbation systems are equal then we may take as perturbation
system for quilted strips the one given by pullback under the map that
forgets the quilting.  In this case the only isolated holomorphic
quilted treed disks are the constant ones, since then
$\ol{MW}_{n,1}(L)$ is a $[0,1]$-fiber bundle over the subset of
$\ol{MW}_n(L)$ consisting of non-constant holomorphic treed disks.
The morphism $\phi$ is then the identity morphism.


Returning to the general case, the morphism $\phi$ is a homotopy
equivalence of \ainfty algebras, by an argument using {\em
  twice-quilted disks} similar to the argument for \ainfty bimodules
given later in Section \ref{twice}.  We show in Figure \ref{twicecol}
the moduli space of twice-quilted stable disks $\ol{M}_{n,2}$, in the
case $ n = 2$ which is a pentagon whose vertices correspond to the
expressions
$$f(g(x_1x_2)), f(g(x_1)g(x_2)), f(g(x_1))f(g(x_2)),
((fg)(x_1)) ((fg)(x_2)), (fg)(x_1x_2) .$$  
There is a similar treed version $\ol{MW}_{n,2}$, whose description we
omit.  
\begin{figure}[ht]
\includegraphics[height=2in]{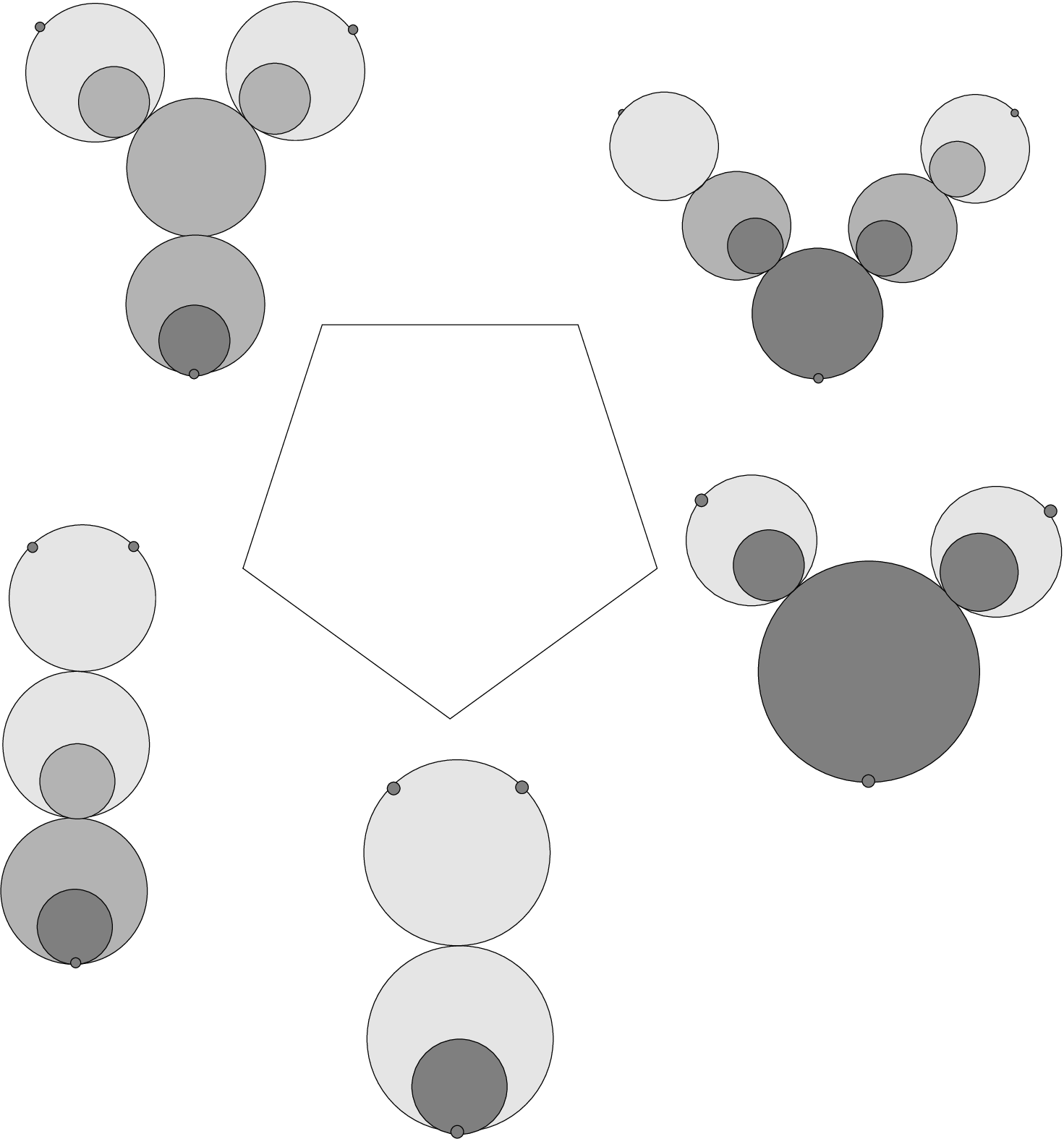}
\caption{Twice-quilted disks}
\label{twicecol} 
\end{figure} 
Given a triple of perturbation systems
$(B_{*,0},F_{*,0}),(B_{*,1},F_{*,1}),(B_{*,2},F_{*,2})$ and morphisms
$$\phi_{ij}: CQF(L;B_{*,i},F_{*,i}) \to CQF(L;B_{*,j},F_{*,j}), \quad
0 \leq i < j \leq 2 $$
defined by extension of these over the moduli space of quilted disks,
we wish to compare the composition $\phi_{12} \circ \phi_{01}$ with
$\phi_{02}$.  Over $\ol{MW}_{n,2}$ we consider the moduli space
$\ol{MW}_{n,2}(L)$, using a perturbation system $(B_*,F_*)$ equal to
the given perturbation systems equal to the given perturbation systems
in a neighborhood of the boundary.  (That is, $(B_{*,0}, F_{*,0})$ on
the edges above the medium shaded region, $(B_{*,1},F_{*,1})$ on edges
between the beginning of the medium shaded region and the darkly
shaded region, and $(B_{*,2},F_{*,2})$ after the beginning of the
darkly shaded region.)  The facets of $\ol{MW}_{n,2}$ correspond to
either to terms in the definition of composition of \ainfty maps
$\phi_{12} \circ \phi_{01}: CQF(L; B_{*,0}, F_{*,0}) \to CQF(L;
B_{*,2},F_{*,2})$, to the components contributing to $\phi_{02} :
CQF(L; B_{*,0},F_{*,0}) \to CQF(L; B_{*,2}, F_{*,2})$, or to terms
corresponding to the bubbling off of some markings on the boundary
which define a homotopy operator for the difference $\phi_{12} \circ
\phi_{01} - \phi_{02}$.  Since we do not need homotopy invariance for
any of our results (we do need homotopy invariance of the \ainfty
bimodule) we omit the proof; see \cite{ainfty} for a related argument
for \ainfty functors associated to Lagrangian correspondences.  In
case $(B_{*,0},F_{*,0}) = (B_{*,2}, F_{*,2})$ this produces a homotopy
between $\phi_{12} \circ \phi_{01}$ and the identity, and hence
$\phi_{12},\phi_{01}$ define a homotopy equivalence between $CQF(L;
B_{*,0},F_{*,0})$ and $CQF(L; B_{*,1},F_{*,1})$.  In particular
homotopy invariance of $CQF(L)$ implies that the quasimap Floer
cohomologies $HQF(L)$ defined using different perturbation systems are
isomorphic.

\section{Quasimap \ainfty bimodule for Lagrangian pair} 
\label{bimodule}

We describe here a version of Floer theory for pairs $(L_0,L_1)$ of
Lagrangian submanifold, which counts {\em perturbed holomorphic treed
  strips} giving rise to an {\em \ainfty bimodule}, invariant up to
\ainfty homotopy of the choice of perturbation.  In particular, if
$L_1$ is displaceable from $L_0$ by a Hamiltonian diffeomorphism then
the cohomology of this \ainfty bimodule vanishes.

\subsection{\ainfty bimodules} 
\label{bimod}

Let $A_0,A_1$ be \ainfty algebras.  An {\em \ainfty bimodule} is a
$\Z$-graded vector space $M$ equipped with operations
$$ \mu_{d|e} : A_0^{\otimes d} \otimes M \otimes A_1^{\otimes e} \to
M[1-d-e] $$
satisfying the relations 
$$ \sum_{i,k} (-1)^{\aleph} \mu_{d-i+1|e}( a_{0,1}, \ldots,
\mu_{0,i}(a_{0,k},\ldots, a_{0,k+i-1}), a_{0,k+i}, \ldots, a_{0,d},
m,a_{1,1},\ldots, a_{1,e})) $$
$$ + \sum_{j,k} (-1)^{\aleph} \mu_{d|e-j+1}(a_{0,1},\ldots,
a_{0,d},m,a_{1,1},\ldots, \mu_{1,j}(a_{1,k},\ldots, a_{1,k+j-1}),\ldots,
a_{1,e})  $$
$$ + \sum_{i,j} (-1)^{\aleph} \mu_{d-i|e-j}( a_{0,1},\ldots, a_{0,d-i},
\mu_{i|j}(a_{0,d-i+1},\ldots, a_{0,d},m,a_{1,1},\ldots,a_{1,j}),
\ldots, a_{1,e}) = 0 $$
where we follow Seidel's convention \cite{seidel:sub} of denoting by
$(-1)^{\aleph}$ the sum of the reduced degrees to the left of the
inner expression, except that $m$ has ordinary (unreduced) degree.  As
before, the $\Z$-grading is not necessary and a $\Z_2$-grading
suffices.

A {\em morphism} $\phi$ of \ainfty-bimodules $M_0$ to $M_1$ of degree
$|\phi|$ is a collection of maps
$$ \phi_{d|e} : A_0^{\otimes d} \otimes M_0 \otimes A_1^{\otimes e}
\to M_1[|\phi|-d-e] $$
satisfying a splitting axiom
\begin{multline} \label{morphism} 
 \sum_{i,j} (-1)^{|\phi| \aleph}
\mu_{1,d|e}(a_{0,1},\ldots,a_{0,d-i}, \phi_{i|j}(a_{0,d-i+1},\ldots,
a_{0,d},m,a_{1,1},\ldots,a_{1,j}), a_{1,j+1},\ldots,a_{1,e}) \\ 
+ \sum_{i,j} (-1)^{|\phi| + 1 + \aleph} \phi_{d|e}(a_{0,1},\ldots,a_{0,d-i} ,
\mu_{0,i|j}(a_{0,d-i+1},\ldots,a_{0,d},m,a_{1,1},\ldots,a_{1,j} ),
a_{1,j+1},\ldots,a_{1,e}) \\
 + \sum_{i,j} (-1)^{|\phi | + 1 + \aleph}
\phi_{d|e}(a_{0,1}, \ldots, a_{0,j-1}, \mu_{0,i}(a_{0,j},\ldots, a_{0,j+i-1} ), \ldots,
a_{0,d},m,a_{1,1},\ldots,a_{1,e}) \\
+ \sum_{i,j} (-1)^{|\phi| + 1 + \aleph} \phi_{d|e}(a_{0,1},\ldots,
a_{0,e},m,a_{1,1},\ldots, \mu_{1,j}(a_{1,j},\ldots,a_{j+i-1}),
\ldots,a_{1,e}) = 0 .\end{multline}
Composition of morphisms $\phi: M_0 \to M_1, \psi: M_1 \to M_2$ is defined by 
\begin{multline} (\phi \circ \psi)_{d|e}(a_{0,1},\ldots, a_{0,d},m, a_{1,1},\ldots,a_{1,e})  \\
= \sum_{i,j} (-1)^{|\psi| \aleph} 
\phi_{i|j}(a_{0,1},\ldots,a_{0,i}, 
\psi_{d-i|e-j}(a_{0,i+1},\ldots, a_{0,d},m,
a_{1,1},\ldots,a_{1,e-j}),a_{1,e-j+1},
\ldots,a_{1,e}) .\end{multline}
A {\em homotopy} of morphisms $\psi_0,\psi_1: M_0 \to M_1$ of degree
zero is a collection of maps $(\phi_{d|e})_{d,e \ge 0}$ such that the
difference $\psi_1 - \psi_0$ is given by the expression on the left
hand side of \eqref{morphism}.  Any \ainfty algebra $A$ is an \ainfty
bimodule over itself with operations
\begin{multline} \label{change}
 \mu_{d|e}(a_{0,1},\ldots, a_{0,d},m,a_{1,1},\ldots,a_{1,e}) \\ =
(-1)^{1 + \diamondsuit } \mu_{d+ e + 1}(a_{0,1},\ldots,
a_{0,d},m,a_{1,1},\ldots,a_{1,e}), \quad \diamondsuit = \sum_{j=1}^e
(|a_{1,j}| + 1),
 \end{multline}
see Seidel \cite[2.9]{seidel:susp}.

\subsection{Treed strips}
\label{treedstrips}

In this section we describe the combinatorics of the objects used to
define our \ainfty bimodule.  A {\em treed strip with $(d,e)$
  markings} is a treed disk with $d + e + 2$ markings, such that the
edges connecting the $0$-th marking with the $d+1$-marking all have
length zero.  The disk components in between $z_0$ and $z_{d+1}$ have
canonical embeddings of holomorphic strips $\R \times [0,1]$, whose
complements are the nodes or markings connecting (eventually) to $z_0$
or $z_{d+1}$.  We denote by $\ol{MW}_{d|e}$ the moduli space of {\em
  stable} $(d,e)$-marked treed strips.  See Figure \ref{MW12} for the
case $d = 1, e= 2$.
\begin{figure}[ht]
\includegraphics[height=.5in]{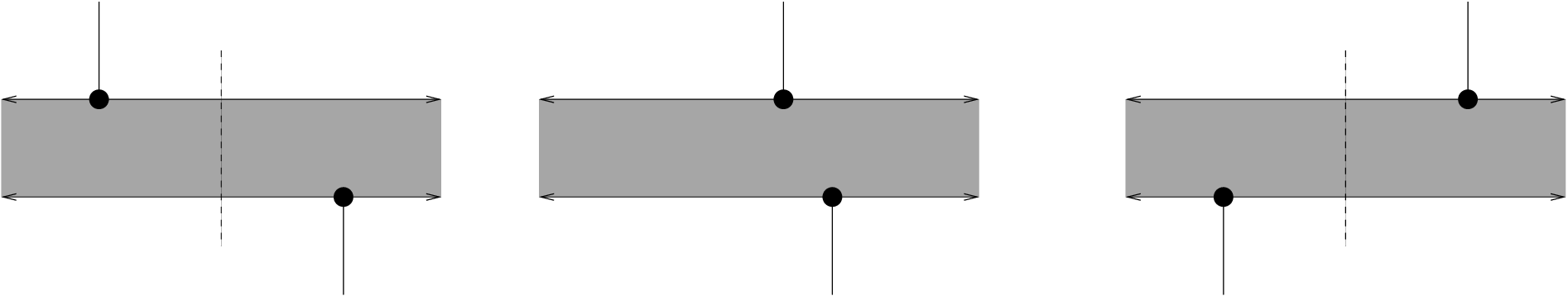}
\caption{Moduli space of stable treed strips with $d =1, e = 2$}
\label{MW12}
\end{figure} 
\noindent The orientation on $\ol{MW}_{d|e}$ can be chosen via its
identification with $\ol{MW}_{d+e+1}$.  Denote by $\ol{UMW}_{d|e} \to
\ol{MW}_{d|e}$ the {\em universal treed strip}.  An {\em infinitesimal
  deformation} of a treed strip of fixed type $\Sigma$ is an
infinitesimal deformation of the nodal points or lengths of the line
segments as in \eqref{defsig}.  Let $\Def_\Gamma(\Sigma)$ denote the
space of deformations preserving the combinatorial type.  We denote by
$z$ the number of nodes corresponding to edges of zero length not
connecting strips, and by $i$ the number of infinite nodes connected
edges of infinite length or strips.  Define
$$ \Def(\Sigma) = \Def_\Sigma(\Sigma) \times ( (-\infty,0) \cup \{ 0
\} \cup (0,\infty))^z \times [0,\infty)^i $$
with the second two factors representing {\em gluing parameters} for
line segments of length zero resp. infinity.  
The gluing construction described above
defines a homeomorphism from a neighborhood of zero in
$\Def(\Sigma)/\aut(\Sigma)$ to a neighborhood of $\Sigma$ in
$\ol{MW}_{d|e}$.

\subsection{Holomorphic treed quasistrips} 

Let $(X,\omega)$ be a compact or convex Hamiltonian $G$-manifold
equipped with a $G$-invariant compatible almost complex structure $J
\in \J(X,\omega)^G$ and $L_0,L_1 \subset X \qu G$ compact Lagrangian
submanifolds.  We suppose that the $L_i, i=0,1$ are equipped with
metrics induced from $G$-invariant metrics on $\ti{L}_i, i = 0,1$, and
Morse-Smale functions $F_i \to \R$.  Denote by $\ti{F}_i : \ti{L}_i
\to X$ the lifts of $F_i$ to $\ti{L}_i$.

\begin{definition}  A {\em holomorphic treed 
quasistrip} for $L_0,L_1 \subset X \qu G$, $J \in \J(X)^G$ and $H \in
  C^\infty_c([0,1] \times X)^G$ consists of a treed strip $\Sigma$, a
  continuous map $u : \Sigma \to X$ such that (i) on each strip, $u$
  is $(J,H)$-holomorphic strip with boundary in $(\ti{L}_0,\ti{L}_1)$;
  (ii) on each disk connecting to $\R \times \{ j \}$, $u$ is a
  $J$-holomorphic disk with boundary in $\ti{L}_j$, and (iii) on each
  edge connecting to $\R \times \{j \}, j = 0,1$, $u$ is a gradient
  trajectory of $\ti{F}_j$ on $\ti{L}_j$.
An {\em isomorphism} of treed quasistrips $u_j : \Sigma_j \to X, j =
0,1$ is a morphism of treed strips $\phi: \Sigma_0 \to \Sigma_1$
together with an element $g \in G$, such that $g u_0 = \phi^* u_1$.  A
treed quasistrip is {\em stable} if it has no automorphisms (that is,
each strip or disk on each $u$ is constant has at least three nodal or
marked points) and each node connecting two edges maps to a critical
point.
\end{definition}  

See Figure \ref{treedtraj}, in which the dotted line indicates a
broken trajectory.  For each treed disk we have a metric tree
$\Sigma_{(1)}$ obtained by forgetting the disk and strip components.
As in the construction of the Fukaya algebra we allow perturbation
systems $F_* = (F_{n,v} \in C^\infty(UW_{n,v} \times (L_0 \cup L_1)))$
depending on $\Sigma_{(1)}$.  We require that $F_{n,v}$ is equal to
$F$ on the complement of a compact subset of the union of open edges.
Let $\ol{MW}_{d|e}(L_0,L_1;H,F_*)$ denote the moduli space of
isomorphism classes of finite energy stable treed perturbed
holomorphic quasistrips with boundary in $(\ti{L}_0,\ti{L}_1)$ with
$d$ left resp. $e$ right markings.

\begin{figure}[h]
\begin{picture}(0,0)%
\includegraphics{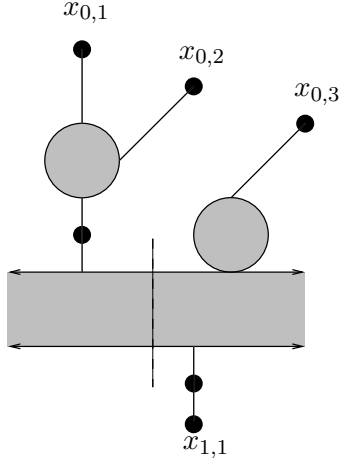}%
\end{picture}%
\setlength{\unitlength}{4144sp}%
\begingroup\makeatletter\ifx\SetFigFontNFSS\undefined%
\gdef\SetFigFontNFSS#1#2#3#4#5{%
  \reset@font\fontsize{#1}{#2pt}%
  \fontfamily{#3}\fontseries{#4}\fontshape{#5}%
  \selectfont}%
\fi\endgroup%
\begin{picture}(2059,2721)(1789,-5019)
\put(3514,-2867){\makebox(0,0)[lb]{{{$x_{0,3}$}%
}}}
\put(2847,-4980){\makebox(0,0)[lb]{{{$x_{1,1}$}%
}}}
\put(2134,-2378){\makebox(0,0)[lb]{{{$x_{0,1}$}%
}}}
\put(2825,-2644){\makebox(0,0)[lb]{{{$x_{0,2}$}%
}}}
\end{picture}%
\caption{A holomorphic treed strip}
\label{treedtraj}
\end{figure} 

Given a holomorphic quasistrip $u$ we denote by $D_u$ the
linearization of the map
\begin{multline} \cF_u: \Def_\Gamma(\Sigma) \oplus \Omega^0(\ti{\Sigma}_{(2)},u_2^*
TX, (\partial u_2)^* T (\ti{L}_0 \sqcup \ti{L}_1))_{1,p,\alpha} \oplus
\Omega^0(\ti{\Sigma}_{(1)},u_1^* T (\ti{L}_0 \sqcup
\ti{L}_1))_{1,p,\alpha} \\ \to \Omega^{0,1}(\ti{\Sigma}_{(2)},u^*
TX)_{0,p,\alpha} \oplus \Omega^1(\ti{\Sigma}_{(1)},u_1^* T (\ti{L}_0
\sqcup \ti{L}_1))_{0,p,\alpha} \oplus \bigoplus_{j=1}^m T_{u(w^\pm_j)}
(I_{\delta(j)}) \end{multline}
where $T(\ti{L}_0 \sqcup \ti{L}_1)$ denotes either boundary conditions
in $T \ti{L}_0$ or $T \ti{L}_1$, depending on the boundary component,
$I_{\delta(j)} = \ti{L}_0,\ti{L}_1$ or $\ti{L}_0 \cap \ti{L}_1$
depending on the type of node, given by $\xi \mapsto \cT_u(\xi)^{-1}
\olp_{J,H} \exp_u(\xi)$ on the strip components of dimension $2$, by
the Cauchy-Riemann operator on the disk components, the gradient
operator on the one-dimensional components, and the difference of
evaluation maps at the nodes $w^\pm_j$, see \eqref{diffs}.  If $H$ is
chosen so that the set of generalized intersection points
$\cI(L_0,L_1;H)$ of \eqref{genint} is transverse, then $D_u$ is
Fredholm, by the discussion in Section \ref{floer}.  We say that $u$
is {\em regular} if $D_u$ is surjective.  Let
$MW_{d|e,\Gamma}^{\reg}(X,L_0,L_1;H)$ denote the moduli space of
regular stable treed quasistrips with the combinatorial type $\Gamma$.
Then $MW_{d|e,\Gamma}^{\reg}(X,L_0,L_1;H)$ is a smooth manifold, with
tangent space at $u$ equal to the kernel of $D_u$ modulo the subspace
generated by $\g$ and $\aut(\Sigma)$, by an application of the
implicit function theorem for Banach spaces similar to that of Theorem
\ref{trajreg}.

As in the previous section, choices of perturbations used to define
the spaces $ \ol{MW}_{d'|e'}(L_0,L_1), \ol{MW}_i(L_0), \ol{MW}_j(L_1)
$ for $d'< d, e' <e$ and all $i,j$ induce perturbations
$F_{n,v}^{pre}$ on $\ol{MW}_{d|e}(L_0,L_1)$ in a neighborhood of the
image of the strata
$$ \ol{MW}_{d_0|e_0}(L_0,L_1)_0 \times_{\cI(L_0,L_1)} \ol{MW}_{d_1|e_1}(L_0,L_1)_0 $$
and
$$\ol{MW}_{d - i + 1|e}(L_0,L_1)_0 \times_{\crit(F_0)}
\ol{MW}_i(L_0)_0, \quad \ol{MW}_{d|e - i + 1}(L_0,L_1)_0 \times_{\crit(F_1)}
\ol{MW}_i(L_1)_0 $$ 
under the gluing construction.  These extend to a system of
perturbations over all of $\ol{MW}_{d'|e'}(L_0,L_1)$, since the space
of perturbations is convex.

\begin{theorem} \label{pairreg} 
Suppose that every stable holomorphic disk with boundary in
$\ti{L}_0,\ti{L}_1$ is regular, and a Hamiltonian perturbation $H$ has
been chosen so that every quasistrip is regular.  For any $d,e,\mdeg$,
given a system of perturbations $F_*$ for strata with $d' < d, e' < e,
E(\mdeg') < E(\mdeg)$ extending the given pairs on the semi-infinite
ends such that every holomorphic treed strip of class $\mdeg' $ is
regular, there exists a comeager subset of perturbations
$\PP^{\reg}(L_0,L_1)$ in $C^\infty_b(UW_{d+e+1,v} \times (L_0 \sqcup
L_1))$ such that if $F_1 \in \PP^{\reg}(L_0,L_1)$ and $\rho$ is a
cutoff function on $UW_{d+e+1,v}$ with sufficiently large compact
support on the union of open edges, every holomorphic treed quasistrip
$u: \Sigma \to X$ for $F_{n,v} := F_{n,v}^{pre} + \rho F_1$ with
homotopy class $\mdeg$ is regular.
\end{theorem} 
 
\begin{proof} As in Theorem \ref{regpert}; the only issue is to make
the evaluation maps at the nodes connecting edges with disks
transverse. This holds for generic perturbations $F_1$ by Sard-Smale,
by choosing sufficiently large perturbations on the interiors of the
edges so that the evaluation maps at the endpoints are submersions,
for perturbations of class $C^l$, and then for smooth perturbations by
density.  \end{proof}

\begin{theorem}  \label{mwde}  Suppose that $L_0,L_1$ are such 
that all holomorphic disks with boundary in $\ti{L}_0,\ti{L}_1$ are
regular; a compactly supported Hamiltonian perturbation $H$ has been
chosen as above so that all $(J,H)$-holomorphic strips are regular;
and perturbations $F_{n,v}$ have been chosen inductively as above so
that all treed quasistrips are regular.  Then
$\ol{MW}_{d|e}(L_0,L_1)_1$ is a compact one-manifold with boundary
given by the union of broken treed trajectories
$$ \ol{MW}_{d_0|e_0}(L_0,L_1)_0 \times_{\cI(L_0,L_1)} \ol{MW}_{d_1|e_1}(L_0,L_1)_0 $$
and configurations
$$\ol{MW}_{d - i + 1|e}(L_0,L_1)_0 \times_{\crit(F_0)}
\ol{MW}_i(L_0)_0, \quad \ol{MW}_{d|e - i + 1}(L_0,L_1)_0 \times_{\crit(F_1)}
\ol{MW}_i(L_1)_0 $$ 
of a treed strip and a treed disk.
\end{theorem} 

\begin{proof}   Compactness is Gromov compactness as in \cite{ms:jh}
combined with compactness for Floer trajectories for Lagrangian pairs
with clean intersection, discussed in Section \ref{floer}, and for
broken gradient trajectories; note that compact support of the
Hamiltonian perturbation $H$ allows us to apply the maximum principal
outside of a compact set given by the convexity condition.  Let $u$
represent a regular point in $\ol{MW}_{d|e}(L_0,L_1)$.  Define the
{\em preglued treed holomorphic strip} $G_\delta^{\on{approx}}(u):
\Sigma^\delta \to X$ by combining the gluing procedures for Morse
trajectories for pairs of Lagrangians with clean intersections
(discussed above in the proof of Theorem \ref{bound}) and gluing for
disks (see Abouzaid \cite{ab:ex}) and gluing for Morse trajectories
(see Schwarz \cite{sch:morse}).  The proofs of existence of a
uniformly bounded right inverse, error estimate and uniform quadratic
bound are the same as in those situations, since the third term in
$\cF_{G_\delta^{\on{approx}}(u)}$ does not involve the gluing
parameter.  By the implicit function theorem there exists a unique
solution $(\zeta,\xi)$ to $\cF_{G_\delta^{\on{approx}}(u)}(\zeta,\xi))
= 0$ in a sufficiently small neighborhood of $0$ in
$\Def_\Gamma(\Sigma) \oplus \Omega^0(\Sigma^\delta_{(2)},
(G_\delta^{\on{approx}}(u))^* TX)_{1,p,\delta,\alpha} \oplus
\Omega^0(\Sigma^\delta_{(1)}, (G_\delta^{\on{approx}}(u))^* T (\ti{L}_0
\sqcup \ti{L}_1))_{1,p,\delta,\alpha}$, and we set $G_\delta(u) :=
\exp_{G_\delta^{\on{approx}}(u)}(\xi): \Sigma^{\zeta,\delta} \to X$,
where $\cF_{G_\delta^{\on{approx}}(u)}$ is as in the proof of Theorem
\ref{bound} but with the perturbed Cauchy-Riemann operator on the
strips.
\end{proof} 

\begin{remark} In the proof above we required weighted Sobolev spaces on 
$\Sigma_{(1)}$ because the function $\ti{F}$ is Morse-Bott but not
  Morse.  However, there is bijection between gradient trajectories of
  $\ti{F}$ modulo $G$ and gradient trajectories of $F$, and by working
  with gradient trajectories of $F$ instead of $\ti{F}$ the use of
  weighted Sobolev spaces on $\Sigma_{(1)}$ may be avoided.
\end{remark}

Suppose that every element of $\ol{MW}_{d|e}(L_0,L_1)$ is regular.
Orientations are constructed by choosing for each such end an
orientation on the disk one with marking, and boundary given by a path
from $TL_0$ to $TL_1$, see \cite{orient}.  Define operations
$$ \mu_{d|e}: CQF(L_0)^{\otimes d} \otimes CQF(L_0,L_1) \otimes
CQF(L_1)^{\otimes e} \to CQF(L_0,L_1) $$
by counting treed Floer trajectories;
\begin{multline}
\bra{x_{0,1}} \otimes \ldots \otimes \bra{x_{0,d}} \otimes \bra{x}
\otimes \bra{x_{1,1}} \otimes \ldots \otimes \bra{x_{1,e}} \\ \mapsto
\sum_{ [u] \in
\ol{MW}_{d|e}(x_{0,1},\ldots,x_{0,d},x,x_{1,1},\ldots,x_{1,e},y)_0}
(-1)^{\heartsuit + \diamondsuit} \eps(u) \Hol_{L_0,L_1}(u) q^{A(u)}
\bra{y} \end{multline}
where $\diamondsuit$ is defined in \eqref{change}

\begin{theorem} Suppose that perturbations have been chosen
as in Theorem \ref{mwde}.  Then the maps $(\mu_{d|e})_{d,e \ge 0}$
induce on $CQF(L_0,L_1)$ the structure of a $\Z_2$-graded \ainfty
$(CQF(L_0), CQF(L_1))$-bimodule.
\end{theorem} 

\begin{proof}  By the previous Theorem \ref{mwde} and the definition 
in \eqref{morphism}.  The sign computation is equivalent to that in
Theorem \ref{ainfty}.
\end{proof} 

The \ainfty bimodule $CQF(L_0,L_1)$ is independent of the choice of
Hamiltonian perturbation $H$, metrics $B_j$ on $\ti{L}_j, j = 0,1$,
and perturbation system $F_*$ used to construct it, up to homotopy of
\ainfty bimodules over $\Lambda$.  (Independence of $F_*$ holds over
$\Lambda_0$, while independence from $H$ holds only over $\Lambda$.)
However, we will not give the argument, since it is very similar to
the \ainfty bimodule map given in the next section.  Note we are
assuming that all holomorphic disks are regular, and $H$ is such that
$(J,H)$-holomorphic strips are regular.

\begin{lemma} \label{mu00} 
Suppose that $J \in \J(X)^G$ is such that every non-constant stable
holomorphic disk with boundary in $L$ is regular and has positive
Maslov index.  Then for sufficiently small perturbations we have
$ \mu_{0,0}^2 = 0 $
and hence the quasimap Floer cohomology of the pair 
$$ HQF(L_0,L_1) := H(\mu_{0,0}) $$
is well-defined.
\end{lemma} 

\begin{proof}  The proof is similar to that of Proposition \ref{mu0}.   The \ainfty relation  
$\mu_{0,0}^2(\cdot) = \mu_{1,0}(\mu_0^{L_0}(1), \cdot) \pm
  \mu_{0,1}(\cdot, \mu_0^{L_1}(1))$ and the fact that $\mu_0^{L_j}(1)$
  is a multiple of the generator corresponding to the maximum of the
  Morse function implies that $\mu_{0,0}^2(\cdot) = 0$.
\end{proof} 

\section{\ainfty-bimodule isomorphism in the case of an equal pair}
\label{bimoduleiso}

We show that the \ainfty algebra constructed in Section \ref{algebra}
for a Lagrangian $L$, considered as an \ainfty bimodule over itself,
is isomorphic to the \ainfty bimodule constructed in Section
\ref{bimodule}, for the pair $(L,L)$.

\subsection{Quilted strips and quilted treed quasistrips} 

A {\em $(d,e)$-marked quilted strip} is the same as a $(d,e)$-marked
strip, that is, a strip with markings on the boundary, except that
isomorphisms do not include translations.  Equivalently, a
$(d,e)$-marked quilted strip is a $(d,e)$-marked strip $(\R \times
[0,1], z_{0,1},\ldots, z_{0,d}, z_{1,1}, \ldots, z_{1,e})$ with an
additional interior marking at some point $y = (s,1/2)$, and an
isomorphism from $(\Sigma, z_{0,1},\ldots, z_{0,d}, z_{1,1}, \ldots,
z_{1,e},y)$ to $(\Sigma', z_{0,1}',\ldots, z_{0,d}', z_{1,1}', \ldots,
z_{1,e}',y')$ is an isomorphism $\Sigma \to \Sigma'$ mapping $z_{i,j}$
to $z_{i,j}'$ and $y$ to $y'$.  Geometrically the interior marking $y$
represents the point at which we ``turn on'' a Hamiltonian
perturbation, and we will draw a quilted strip by drawing a shaded
region beginning at $y$ and extending to the right.  

\begin{figure}[ht]
\includegraphics[height=.5in]{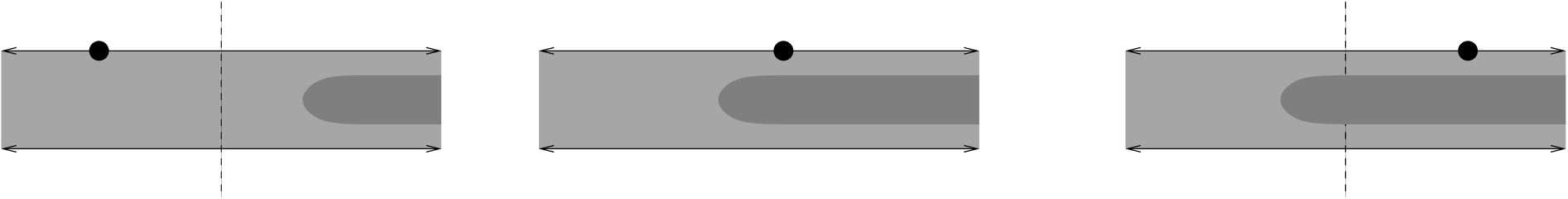}
\caption{Quilted strips}
\label{quiltedstrips}
\end{figure}

The space $M_{d|e,1}$ of $(d|e)$-marked quilted strips has a
compactification $\ol{M}_{d|e,1}$ allowing disk bubbles to develop on
the boundary and strips to bubble off on both ends.  In the pictures,
we find it convenient to indicate the quilted strip (as opposed to the
strip bubbles) by adding a shaded region, signifying that the
Cauchy-Riemann equation is given a Hamiltonian perturbation on this
region.  See Figure \ref{quiltedstrips} for the moduli space of
quilted $(1,0)$-marked strips.  A {\em quilted component} of the
quilted strip is a component with some shading, otherwise the
component is {\em unquilted}.  Each quilted component has either one
or two {\em quilted ends}, that is, ends with some shaded region.  The
cells of $\ol{M}_{d|e,1}$ are in one-to-one correspondence with
expressions in $d$ variables $x_1,\ldots,x_d$, a variable $m$, and
variables $y_1,\ldots,y_e$, and a function $f$ obtained by adding
parentheses to the expression $x_1\ldots x_d m y_e \ldots y_1 $ and
then adding $f(\ldots )$ around some expression including $m$.  For
example, for the case $(d,e) = (1,0)$ the possible expressions are $
f((x_1 m)), f(x_1 m), x_1 f(m)$.  This parametrization of cells is
similar to that of the associahedron, but the combinatorial types
correspond to $d + e + 2$-leaved trees together with a choice of
vertex on the path connecting the $0$-th to the $d + 1$-st leaf.

As before, there is also a treed version.  A {\em $(d,e)$-marked
  quilted treed strip} is the same as a quilted strip, but now
allowing components of dimension one of possibly infinite or zero
length between disk components or between disk and strip components or
between unquilted strip components or between unquilted strip
components and quilted strip components, but not between quilted strip
components.  In addition, there are one-dimensional components
attached to the markings on the boundary, and to the point at infinity
at the left of the strip.  Two quilted treed strips are {\em
  isomorphic} if they are related by automorphism.  A quilted treed
strip is {\em stable} if it has no infinitesimal automorphisms, and
any node connecting two components of dimension one is of infinite
type, that is, attaching ends of the segments of infinite length.
Let $\ol{MW}_{d|e,1}$ denote the moduli space of stable quilted treed
strips.  The case $d = 2,e = 0$ is shown in Figure \ref{coloredtreed}.
We have not drawn the line segments connecting the strips, as this
makes the picture even more complicated; thus dotted lines in the
left-hand configurations mean that the strip has broken and is
connected by a segment of infinite length.  Each picture shows the
configuration represented by a nearby stratum. The moduli spaces
$\ol{MW}_{d|e,1}$ are complicated even for low values of $d,e$,
as can be expected from the definition of morphism of \ainfty
bimodules and the fact that we are adding trees.

\begin{figure}[ht]
\includegraphics[height=3in]{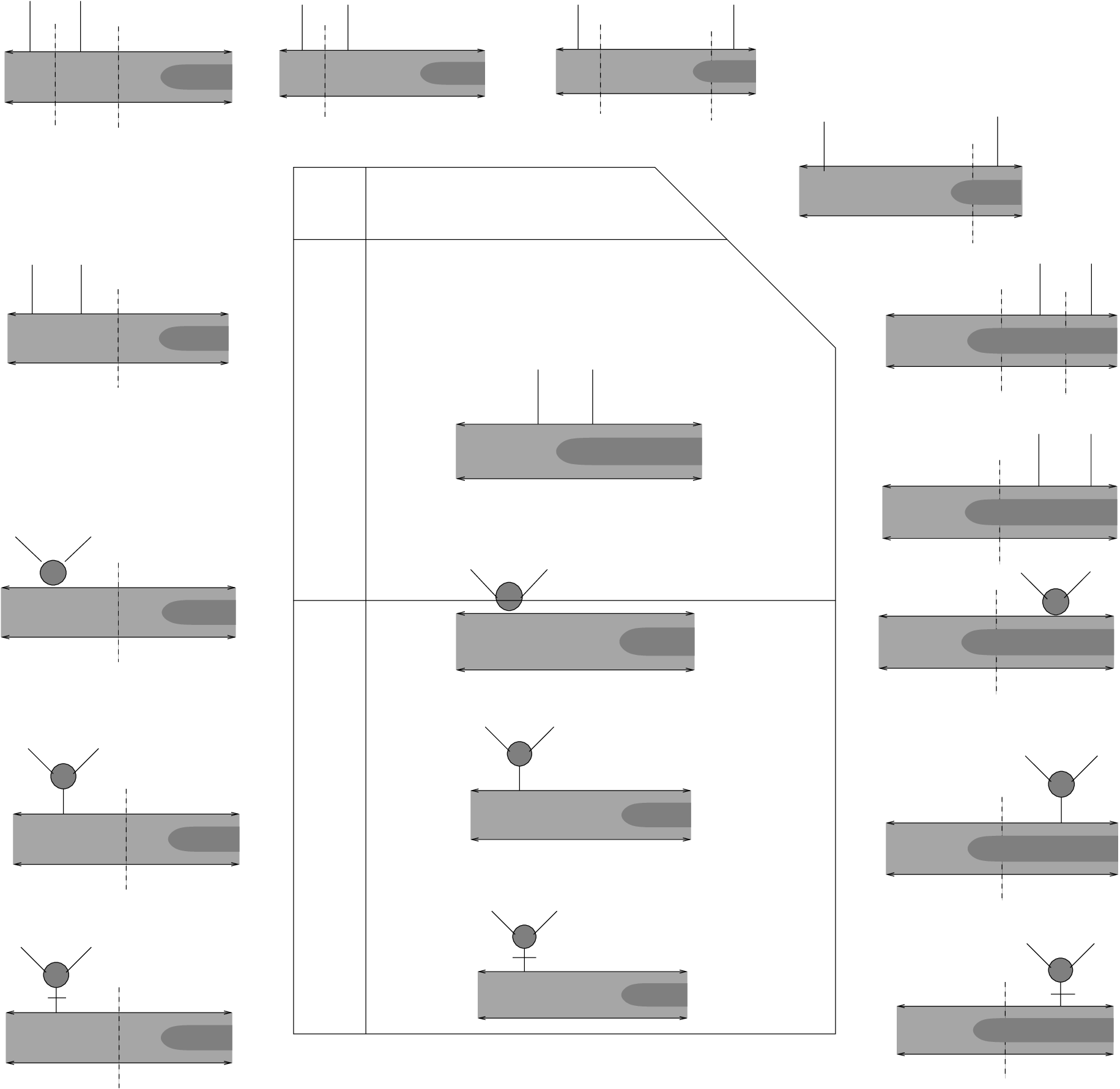}
\caption{The moduli space $\ol{MW}_{2|0,1}$ of $2|0$-marked treed quilted strips}
\label{coloredtreed}
\end{figure}

The forgetful map $\ol{MW}_{d|e,1} \to \ol{MW}_{d|e}$ has
one-dimensional fibers that are canonically oriented so that the
positive orientation corresponds to moving the quilting to the left,
and so the orientation of $\ol{MW}_{d|e}$ induces an orientation on
$\ol{MW}_{d|e,1}$.

\subsection{\ainfty bimodule morphisms} 

We construct a morphism of \ainfty bimodules $CQF(L) \to CQF(L,L)$ via
a continuation argument which counts {\em quilted treed holomorphic
  quasistrips}.  Let $H \in C^\infty_c((0,1) \times X)^G$ be a regular
perturbation for the pair $(L,L)$.  Given a quilted strip with shaded
region beginning at $s_-$ and shaded region ending at $s_0$ we
consider a Hamiltonian perturbation $H_1 = H_{1,s} \d s + H_{1,t} \d
t$ equal to $H \d t$ for $ s \gg 0$ and vanishing for $s \ll 0$.  For
any $(J,H_1)$-holomorphic $u$, we denote by $D_u$ the associated
linearized operator, using weighted Sobolev spaces on the strip-like
ends.  Let $\ti{F}_{n,v} \in C^\infty(UW_{n,v} \times \ti{L})^G$ be a
perturbation of $\ti{F}$ depending on the underlying tree
$\Sigma_{(1)}$ as above.

\begin{definition}  A (perturbed) {\em holomorphic quilted treed quasistrip} 
to $X$ consists of a quilted treed strip $\Sigma$ and a continuous map
$u: \Sigma \to X$ such that (i) on each quilted strip with one quilted
end, $u$ is a $(J,H_1)$-holomorphic map (ii) on each quilted strip
with two quilted ends, $u$ is $(J,H)$-holomorphic map (iii) on each
unquilted disk or strip, $u$ is a $J$-holomorphic map (iv) on each
edge, $u$ is a gradient trajectory of $\ti{F}_{n,v}$.  An {\em
  isomorphism} of holomorphic quilted treed quasistrips $u_j: \Sigma_j
\to X$ is an isomorphism $\phi: \Sigma_0 \to \Sigma_1$ and an element
$g \in G$ such that $\phi^* u_1 = g u_0$; in particular $\phi$ must
preserve the quilting on the quilted strip component.  A holomorphic
quilted treed quasistrip is {\em stable} if it has no automorphisms,
and every node connecting two edges maps to a critical point of $F$.
\end{definition}

See Figure \ref{continue} where the unquilted components are
represented as disks, the component where the quilted regions begins
is represented as a disk with a single strip-like end, and the quilted
components are represented by strips.

\begin{figure}[ht] 
\includegraphics[height=1in]{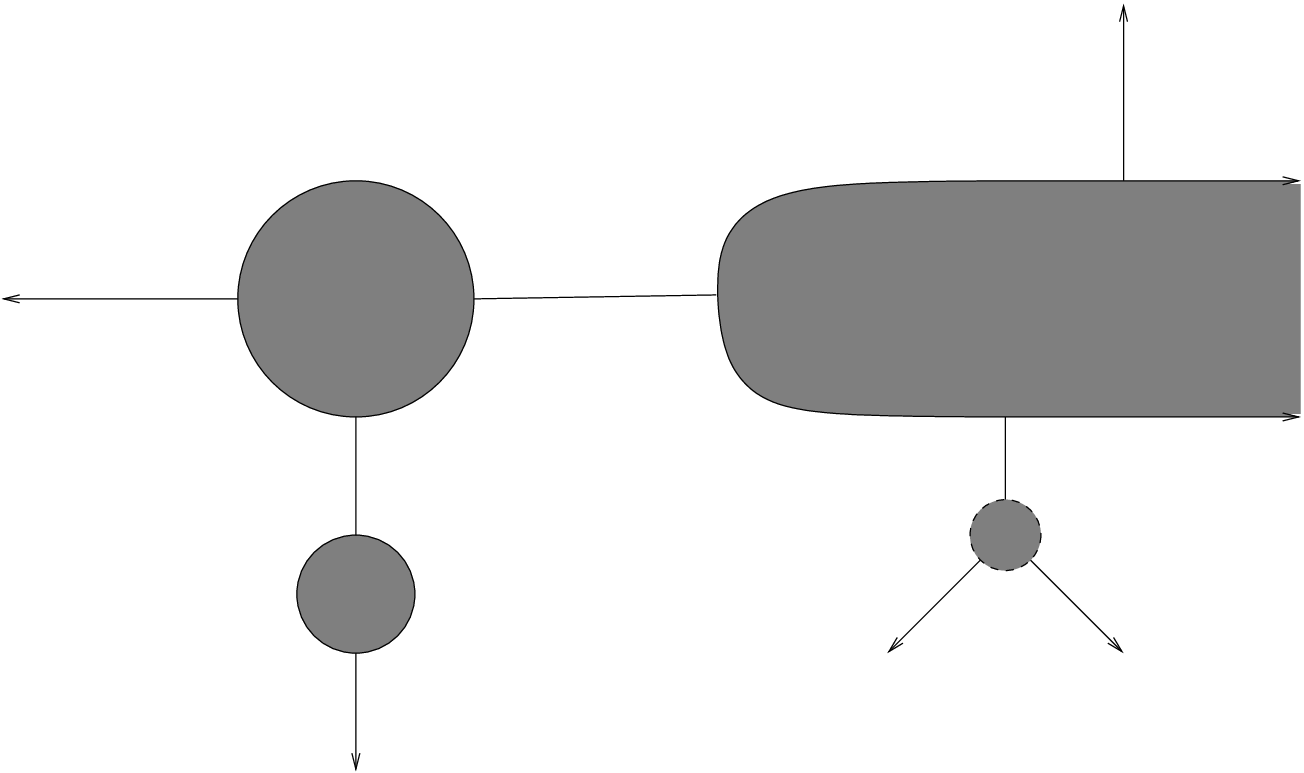}
\caption{A holomorphic quilted treed quasistrip}
\label{continue}
\end{figure} 

Let $\widetilde{\PP}(X)$ be the subset of $\Omega^1(\Sigma,
C^\infty(X)^G)$ consisting of forms $H_s \d s + H_t \d t$ with
$H_s,H_t \in C^\infty_b(X)^G$ having all derivatives bounded.  Let
$\rho \in C^\infty(\Sigma)$ be a function with $\rho(s) = 0, s \leq 0$
and $\rho(s) = 1, s \ge 1$.  Let $\rho_1 \in C^\infty_c(\Sigma \times
X)$ be non-zero on $[0,1] \times [0,1] \times (L_0 \cup L_1)$.  For
any element $H_1' \in \widetilde{\PP}(X)$, we form an admissible perturbation
$H_1 = \rho H \d t + \rho_1 H_1'$.  The proof of the following is
similar to that of Theorem \ref{regpert} and is omitted:

\begin{proposition}  \label{tiHreg}
There exists a comeager subset $\widetilde{\PP}^{\reg}(X,L) \subset
\widetilde{\PP}(X)$ such if $H_1' \in \widetilde{\PP}^{\reg}(X,L)$
then every $(J,H_1)$-holomorphic strip is regular.
\end{proposition} 

The moduli space $\ol{MW}_{d|e,1}(L)$ of holomorphic quilted treed
quasistrips with boundary in $\ti{L}$ has a natural evaluation map
$$ \ev: \ol{MW}_{d|e,1}(L) \to \crit(F)^d \times \crit(F) \times
\crit(F)^e \times \cI(L,L) $$
where $\cI(L,L) \cong L \cap \phi_1(L)$ is the set of generalized
intersection points of $L$ with itself. Let
$\ol{MW}_{d|e,1}(z_1,\ldots,z_{d+e+1},\ldots, y)$ denote the moduli
space of holomorphic quilted treed disks with limits
$z_1,\ldots,z_{d+e+1},y$ along the semi-infinite edges.  Assuming that
every stable holomorphic disk is regular and $H,H_1$ have been chosen
as above, perturbations of the functions on the edges can be
constructed inductively so that every quilted treed holomorphic strip
is regular.  Consider $CQF(L)$ as an \ainfty bimodule over itself
(note the change of signs \eqref{change}), and let $CQF(L,L)$ denote
the \ainfty bimodule defined by the perturbation $H$.  Using the
trivialization of the $\Lambda$-line bundle over the paths in
$\cI(L,L)$ we let $\Hol_L(u) \in \Lambda$ denote the holonomy around
the line bundle around the boundary of $u$.  Define
\begin{multline} 
 \phi_{d|e}: CQF(L)^{\otimes d} \otimes CQF(L;\Lambda) \otimes CQF(L)^{\otimes
   e} \to CQF(L, L) \\ \bra{z_{0,1}} \otimes \ldots \otimes
 \bra{z_{0,d}} \otimes \bra{z} \otimes \bra{z_{1,1}} \otimes \ldots
 \otimes \bra{z_{1,e}} \\ \mapsto \sum_{ [u] \in
   \ol{MW}_{d|e,1}(z_{0,1},\ldots,z_{0,d},z,z_{1,1},\ldots,z_{1,e},y)_0}
 (-1)^{\heartsuit + \diamondsuit} \eps(u) \Hol_{L}(u) q^{A(u)} \bra{y}
\end{multline} 
where the cochain groups involved are defined using $\Lambda$
coefficients; the values $A(u)$ are not necessarily positive because
of the additional term in the energy-area relation \eqref{energyarea}.

\begin{proposition} \label{phi00}  Assuming perturbations have been chosen so that 
all holomorphic treed quasistrips are regular, the collection $\phi =
(\phi_{d|e})_{d,e \ge 0}$ is a morphism of \ainfty bimodules from
$CQF(L)$ to $CQF(L,L)$.  If all non-trivial holomorphic disks have
positive Maslov index then $\phi_{0|0}$ is a chain map: $\phi_{0,0}
\circ \mu_{1,L} = \mu_{1,L,L} \circ \phi_{0,0}.$
\end{proposition} 

\begin{proof}   The boundary 
components of
$\ol{MW}_{d|e,1}(z_{0,1},\ldots,z_{0,d},z,z_{1,1},\ldots,z_{1,e},y)_1$
consist of configurations where a treed strip has broken off or a
treed disk has broken off.  The former case corresponds to one of the
first two terms in \eqref{morphism} while the latter corresponds to
the last two terms.  The signs are similar to that of Theorem
\ref{ainfty}: The first term in \eqref{morphism} (in which $\mu$
appears before $\phi$) has an additional sign from coming from the
definition of the orientation on $\ol{M}_{d|e,1}$ as an $[0,1]$-bundle
over $\ol{M}_{d|e}$, so that the orientation of the fiber corresponds
to moving the quilting to the right; this means that the positive
orientation on the gluing parameter for the boundary components
corresponding to terms of the first type becomes identified with the
negative orientation on these fibers, giving rise to the additional
sign.  The degree of the morphism is zero, which causes those
contributions to the sign in \eqref{morphism} to vanish.  This leaves
the contributions from $\aleph$, which are similar to those dealt with
before and left to the reader.  The proof of the assertion on
$\phi_{0|0}$ is similar to that of Lemma \ref{mu00}: in the absence of
holomorphic disks of non-positive index, the additional terms in the
\ainfty relation vanish.
 \end{proof}

On the other hand, counting strips with given limits on the {\em
  right} defines a similar morphism of \ainfty bimodules $\psi =
(\psi_n)_{n \ge 0}$ from $CQF(L,L)$ to $CQF(L)$, by the same argument
but reading everything from right to left.

\subsection{Homotopies of morphisms of \ainfty bimodules} 
\label{twice} 

To show that the composition of $\psi \circ \phi$ is homotopic to the
identity we introduce a moduli space of {\em twice} quilted treed
quasistrips.  This means that each strip $\Sigma = \R \times [0,1]$
has two distinguished lines, represented by values $s_-,s_+ \in \R$
where the quilted region starts and where it ends; the corresponding
components of $\Sigma$ are the quilted components.  Let
$\ol{MW}_{d|e,2}$ denote the moduli space of twice quilted strips with
$d$ markings on the first resp. second boundary component.  The moduli
space $\ol{M}_{0|0,2}$ is shown in Figure \ref{twicefig}; the shading
represents the region where a Hamiltonian perturbation is allowed.
\begin{figure}[ht] 
\includegraphics[height=.5in]{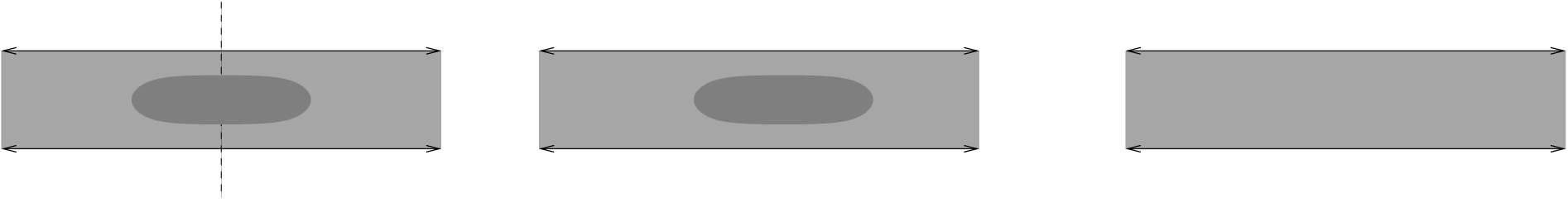}
\caption{Twice quilted quasistrips}
\label{twicefig}
\end{figure} 
As before there is a notion of {\em treed twice-quilted strip} which
allows one-dimensional segments between the disk and strip components.
Let $\ol{MW}_{d|e,2}$ denote the moduli space of stable treed
twice-quilted strips.  The reader is encouraged to draw the picture of
$\ol{MW}_{1|0,2}$ for his or herself.

The structure of a twice-quilting defines a region where the
Hamiltonian perturbation is non-zero.  Suppose we are given a
compactly supported Hamiltonian perturbation $H \in C^\infty_c([0,1]
\times X)^G$ that is a regular perturbation for the pair $L,L$, and a
one-form $H_1 \in \Omega_1(\Sigma, C^\infty_c([0,1] \times X)^G)$ that
makes the moduli space of quilted strips regular.  Given a twice
quilted strip with shaded region beginning at $s_-$ and shaded region
ending at $s_+$ consider a Hamiltonian-valued one-form $H_2$ depending
on $s_+ - s_-$ with support on $[s_-,s_+] \times [0,1]$ which ``turns
on and off the Hamiltonian perturbation''.  That is, if $s_+ - s_- >
2$ then $H_2$ is equal to $ H \d t$ between $s_- + 1 $ and $s_+ - 1$,
is equal to the translation of $H_1$ by $s_-$ on $(s_-,,s_- + 1)$ and
the reflection and translation of $H_1$ on $(s_+-1,s_+)$, and if $s_+
- s_- < 1$ then $H_2$ vanishes.  Let $F_{n,v} \in C^\infty(UW_{n,v}
\times L)$ be a perturbation of $F$ depending on the underlying tree
$\Sigma_{(1)}$ as above.  A (perturbed) {\em holomorphic twice-quilted
  treed quasistrip} to $X$ consists of a twice-quilted treed disk
$\Sigma$ and a continuous map $u: \Sigma \to X$ such that (i) on each
twice-quilted strip, $u$ is a $(J,H_2)$-holomorphic map (ii) on each
quilted strip, $u$ is a $(J,H_1)$-holomorphic map (recall $H_1$ is the
homotopy from $0$ to $H \d t$ used in \ref{tiHreg}) (iii) on each
unquilted strip between two quilted components, $u$ is a
$(J,H)$-holomorphic map (iv) on each unquilted strip not between
quilted components, $u$ is $J$-holomorphic (v) on each edge, $u$ is a
gradient trajectory of $\ti{F}_{n,v}$.  An {\em isomorphism} of
holomorphic twice-quilted treed quasistrips $u_j: \Sigma_j \to X$ is
an isomorphism $\phi: \Sigma_0 \to \Sigma_1$ and an element $g \in G$
such that $\phi^* u_1 = g u_0$ and the quilt structures are preserved.
A holomorphic twice-quilted treed quasistrip is {\em stable} if it has
no automorphisms, and every node connecting two edges maps to a
critical point of $F$.  Let $\ol{MW}_{d|e,2}(L)$ denote the moduli
space of stable holomorphic $(d,e)$-marked treed twice quilted
quasistrips.  Assume that all holomorphic quasidisks with boundary in
$\ti{L}$ are regular.  By a generic choice of perturbations on the
twice-quilted strips (that is, a generic perturbation of $H_2$
depending on $s_+ - s_-$ that is equal to the given function in a
neighborhood of the boundary, and a generic perturbation of the
function $\ti{F}$ equal to the given function on a neighborhood of the
boundary) one obtains transversality as in the previous discussion in
the proof of Proposition \ref{tiHreg}.  For generic choices of
perturbations chosen inductively, the moduli spaces of treed strips
$\ol{MW}_{d|e,2}(L)$ consists entirely of regular elements and the
boundary consists of configurations where either (i) the shaded strip
has separated into two components of a broken strip (ii) a strip has
broken off at $\pm \infty$ (iii) a treed disk has bubbled off along an
edge of infinite length.  Define
\begin{multline} 
 \tau_{d|e}: CQF(L)^{\otimes d} \otimes CQF(L) \otimes
 CQF(L)^{\otimes e} \to CQF(L)[-1]
 \\ (z_{0,1},\ldots,z_{0,d},z,\ldots,z_{1,1},\ldots,z_{1,e},y)
 \\ \mapsto \sum_{ [u] \in
   \ol{MW}_{d|e,2}(z_{0,1},\ldots,z_{0,d},z,z_{1,1},\ldots,z_{1,e},y)_0}
 (-1)^{\heartsuit + \diamondsuit} \eps(u) \Hol_{L_0,L_1}(u) q^{A(u)}
 \bra{y}.
\end{multline}

\begin{theorem}  Suppose that every stable holomorphic disk with boundary
in $L$ is regular, and perturbations have been chosen so that all
treed strips and treed quilted strips are regular leading to maps of
\ainfty bimodules $\phi: CQF(L) \to CQF(L,L)$ resp. $\psi: CQF(L,L)
\to CQF(L)$.  Then $(\tau_{d|e})_{d,e \ge 0}$ defines a homotopy of
morphisms of \ainfty bimodules from the identity to $\psi \circ \phi$
in $\Hom(CQF(L), CQF(L,L))$.  If every non-trivial holomorphic disks
has positive Maslov index then $\tau_{0|0}$ is a chain homotopy of the
identity to $\psi_{0|0} \circ \phi_{0|0}$.  In particular,
$HQF(L;\Lambda)$ is isomorphic to $HQF(L,L)$.
\end{theorem}  

\begin{proof} 
The components of the boundary of $
\ol{MW}_{d|e,2}(z_{0,1},\ldots,z_{0,d},z,z_{1,1},\ldots,z_{1,e},y)_1$
which do {\em not} involve the shaded region breaking, or the shaded
region disappearing, correspond the terms on the left-hand-side of
\eqref{morphism}.  The components where the shaded region breaks
correspond to the contributions to the composition $\psi \circ \phi$,
while the components where the shaded region vanishes give the
identity morphism of \ainfty modules.  The latter follows from a
standard symmetry argument: since we have chosen the perturbation so
that regularity has been achieved, and the equation on the strip with
no shading is translation invariant, the only solutions are the
constant strips.  The first assertion follows.  The proof of the
second assertion is similar to that of Proposition \ref{phi00}: In the
case that every non-trivial holomorphic disk has positive Maslov
index, $\mu_{0,L}(1)$ resp. $\mu_{0,L,L}(1)$ is a multiple of the
generator $1_L$ resp. $1_{L,L}$ and the additional terms in the
\ainfty relation vanish.
\end{proof} 

\begin{corollary}  Suppose that $L \subset X \qu G$ is a Lagrangian such that 
every non-trivial holomorphic disk in $X$ with boundary in $\ti{L}$ is
regular and positive Maslov index.  If $L$ is displaceable, then
$HQF(L;\Lambda) = 0$.
\end{corollary}  

\begin{proof}
Since $HQF(L;\Lambda) \cong HQF(L,L)$ and the latter vanishes if $L$
is displaceable, since $CQF(L,L)$ does for the function $H$ whose flow
displaces $L$.
\end{proof} 

\section{Gauged potentials for toric moment fibers} 

In this section we prove the main Theorem \ref{main}.  Let $X \cong
\C^N$, $G \subset H := U(1)^N$ a sub-torus, acting on $X$ with moment
map $\Phi: X \to \g^\dual$.  Suppose that the $G$ action on
$\Phinv(0)$ has only finite stabilizers, so that $X \qu G$ is a
presentation of a quasiprojective toric orbifold with residual action
of $T = H/G$ and moment map $\Psi: X \qu G \to \t^\dual$.  Recall that
the moment polytope $\Psi(X \qu G)$ is given by a finite set of linear
inequalities \eqref{lineq}.  Let $\lambda$ lie in the interior of
$\Psi(X \qu G)$.  Recall that $L_\lambda = \Psi^{-1}(\lambda) \subset
X \qu G$ is a Lagrangian torus, and $\ti{L}_\lambda \subset X$ its
inverse image in $X$.  By definition $ \lan v_i, \cdot \ran - c_j$ is
the function on $\t^\dual$ defined by the $i$-th coordinate function
on $\h^\dual \cong \R^N$, so the moment map for the action of $H$ on $\C^N$ takes
value $l_j(\lambda)$ on $\ti{L}_\lambda$. Hence $\ti{L}_\lambda$ is
obtained from the standard torus $(S^1)^N \subset \C^N$ by re-scaling
the factors:
$$ \ti{L}_\lambda = \{ (z_1,\ldots, z_N) \ | \ |z_j|^2/2 =
l_j(\lambda)/2\pi, j = 1,\ldots, N \} = \prod_{j=1}^N (
l_j(\lambda)/\pi)^{1/2} S^1 .$$
The following classification result of Cho-Oh \cite[Theorem
  5.3]{chooh:toric} for disks with boundary in $\ti{L}_\lambda$ allows
the computation of the gauged potential:

\begin{proposition} \label{blaschke} Any holomorphic disk in $X = \C^N$
with boundary in $\ti{L}_\lambda$ is given by a Blaschke product
$$ u(z) = \left( (l_j(\lambda)/\pi)^{1/2} \prod_{k=1}^{d_j}
\frac{z - \alpha_{j,k}}{1 - \ol{\alpha}_{j,k}z} \right)_{j=1}^N $$
for some constants $ \alpha_{j,k}$ of norm less than one and some
non-negative integers $d_j$.
\end{proposition} 

\begin{proof} For completeness we reproduce the proof.  Since the 
complex structure and Lagrangian are split, it suffices to consider
the case that $X = \C$ and $\ti{L}_\lambda = S^1$ is the unit circle.
Let $u: D \to X$ be a holomorphic disk with boundary in
$\ti{L}_\lambda$ with $d$ zeroes counted with multiplicity.  Let $v$
be the Blaschke product whose zeroes are those of $u$ with boundary in
$\ti{L}_\lambda$.  Then $u/v$ is a map without zeroes, whose boundary
lies in $\ti{L}_\lambda$.  But any such map must be constant (by e.g. the
maximum principal for the components of $u/v$ and $v/u$) so $v$ is
equal to $u$ up to scalars of unit norm.
\end{proof} 

\begin{corollary}  For the standard complex structure on $X$, 
every stable disk with boundary in $\ti{L}_\lambda$ is regular and has
positive Maslov index.
\end{corollary} 

\begin{proof} 
For smooth domain this is proved by Cho-Oh \cite[Section
  6]{chooh:toric}.  An alternative argument involves doubling the disk
to obtain a sum of rank one, non-negative index Fredholm problems on
the sphere, see e.g. \cite[p.5]{orient}; these always have trivial
cokernel.  Because of the $H$-action on space of holomorphic disks
with the given boundary condition the evaluation map at a marked point
on the boundary is a submersion, since $H$ acts transitively on
$\ti{L}_\lambda$.  Regularity for stable disks whose combinatorial
type is a tree with $k$ components, follows by induction on $k$, as in
Fukaya et al \cite[Theorem 11.1]{fooo:toric1}: Suppose that $\Sigma$
is a marked disk such that the component containing the root $z_0$ is
attached to $d$ other stable disks.  By the inductive hypothesis the
$d$ stable disks are regular and the evaluation map at the attaching
point is submersive. It follows that the evaluation maps at all nodes
are transverse to the diagonal, and in addition the evaluation map at
the root is a submersion.  The assertion on the Maslov index is
immediate from the explicit description.
\end{proof}

By Theorem \ref{regpert} we have

\begin{corollary}  For $X = \C^n$ and $L \subset X \qu G$ a toric
moment fiber as above, there exist compatible systems of perturbations
so that every holomorphic treed quasidisk is regular.
\end{corollary} 

\begin{corollary}  The {\em gauged potential} for $L_\lambda$, that is,
the potential associated to the \ainfty algebra $CQF(L_\lambda)$, is
given by
$$\mu_0(1) = W_\lambda^G(\beta) 1_L, \quad W_\lambda^G(\beta) =
\sum_{i=1}^N e^{ \lan v_i, \beta \ran } q^{l_i(\lambda)}  .$$
\end{corollary} 

\begin{proof}  
For degree reasons, only classes $\mdeg$ of index two contribute, in
which case the classification theorem shows there is exactly one up to
isomorphism for each class $\mdeg$ corresponding to a divisor, given
by $z \mapsto (l_i(\lambda)/\pi)^{1/2} z$ in the $i$-th component and
constant on the other components.  The holonomy of the brane structure
around the boundary of $u$ is $\exp( \lan v_i, \beta \ran )$ by
definition, while the area is $l_i(\lambda)$, hence the claim.
\end{proof}  

\begin{remark} The proof in Fukaya et al \cite{fooo:toric1}, uses a 
stronger version of the divisor equation, which we have avoided
by taking a more direct (but less general) definition of the potential. 
\end{remark} 

\begin{theorem} \label{hqfhm}
 If  $W_\lambda^G$ has a critical point  
$b \in H^1(L_\lambda,\Lambda_0)$ then 
$HQF(L_\lambda^b;\Lambda_0) =
  HM(L_\lambda;\Lambda_0)$.
\end{theorem}

\begin{proof} 
The argument is the same as in Cho-Oh \cite{chooh:toric}, Fukaya et al
\cite[Theorem 13.1]{fooo:toric1}.  Suppose that $b \in
\crit(W_\lambda^G)$, and consider the first page of the Morse-to-Floer
spectral sequence of Proposition \ref{spectral}, so that the
differential $\mu_1^b$ is defined on $HM(L_\lambda)$.  If the partial
derivatives of the potential $W$ vanish at $b$ then $\mu_1(\beta) = 0$
for all classes $\beta$ of degree one, by Corollary \ref{partials}.
In general we induce on the degree of a class $a$.  Suppose that the
coefficient $\mu_{1,\mdeg}^b(a)$ of $q^\mdeg$ in $\mu_1^b(a)$ vanishes
for any $a \in HM^{|a|}(L_\lambda)$ with $|a| \leq k$ and any $\mdeg
\in \pi_2(X,\ti{L})$.  For elements $a_1,a_2$ of degree at most $k$,
consider the Leibniz rule arising from the \ainfty structure equation,
\begin{multline}  \mu_{1,\mdeg}^b(a_1 \cup a_2) = \sum_{\mdeg_1 + \mdeg_2 = \mdeg} \pm
\mu^b_{2,\mdeg_1}(\mu^b_{1,\mdeg_2}(a_1),a_2) \pm
\mu^b_{2,\mdeg_1}(a_1,\mu^b_{1,\mdeg_2}(a_2)) \\ \pm \sum_{\mdeg_1 + \mdeg_2=
  \mdeg, \mdeg_2 \neq 0} \mu^b_{1,\mdeg_1}(\mu_{2,\mdeg_2}^b(a_1,a_2)) \end{multline}
where $\cup$ denotes the leading order term in the product structure,
that is, defined by gradient trees; there are no contributions from
$\mu_0(1)$, by unitarity.  The right-hand side vanishes by the
inductive hypothesis on the degree, since $\mu_{2,\mdeg_2}^b(a_1,a_2)$
has lower degree than $a_1 \cup a_2$ for $\mdeg_2 \neq 0$.  Since
$HM(L_\lambda)$ is generated by degree one classes, $\mu_1^b$ vanishes
on all classes, hence the higher differentials in the spectral
sequence vanish and $HQF(L_\lambda^b;\Lambda_0) =
HM(L_\lambda;\Lambda_0)$.
\end{proof} 

See also Biran-Cornea \cite{bc:ql} for further discussion of this
technique.  As explained in Fukaya et al \cite{fooo}, non-vanishing of
the free part of $HQF(L_\lambda;\Lambda_0)$ is an invariant of
Hamiltonian isotopy. In particular if $HQF(L_\lambda;\Lambda_0)$ is a
non-trivial free $\Lambda_0$-module then $HQF(L_\lambda;\Lambda) =
HQF(L_\lambda;\Lambda_0) \otimes_{\Lambda_0} \Lambda$ is
non-vanishing.  (In general, passing to $\Lambda$ coefficients may
kill the torsion, but non-vanishing of the free part is preserved.)
Hence if $W^G_\lambda$ has a critical point $b \in
H^1(L_\lambda;\Lambda_0)$ then $L_\lambda$ is non-displaceable, which
concludes the proof of Theorem \ref{main} except for the following
technical point in the case that $X \qu G$ is non-compact: So far we
have assumed that the Hamiltonian perturbations are compactly
supported.  However, suppose that $\ti{L}_\lambda$ is displaced by a
Hamiltonian flow $\phi_t$ generated by arbitrary Hamiltonians $H_t \in
C^\infty(X)^G$.  Since $\ti{L}_\lambda$ is compact, the union of
images $K = \cup_{t \in [0,1]} \phi_t(\ti{L}_\lambda)$ is compact, so
there exists a cutoff function $\rho$ equal to one on $K$, and
vanishing outside a compact set.  Let $\phi_t^\rho$ denote the flow of
$\rho H_t$.  Then $\phi_t^\rho$ is equal to $\phi_t$ on $K$ for all $t
\in [0,1]$ and so $\phi_1^\rho(\ti{L}_\lambda) \cap \ti{L}_\lambda =
\phi_1(\ti{L}_\lambda) \cap \ti{L}_\lambda = \emptyset$.

We have the following improvement of Theorem \ref{main}, under the
same assumptions.

\begin{proposition}  \label{est}
Suppose the {\em gauged potential} $ W^G_\lambda$ has a critical
point.  Then for any compactly supported Hamiltonian diffeomorphism
$\phi \in \on{Diff}^h(X \qu G)$ such that $\phi(L_\lambda)$ intersects
$L_\lambda$ transversally, the number of intersection points is at
least $2^{\dim(T)}$.
\end{proposition} 

\begin{proof}   
Suppose that $\phi$ is the flow of $H \in C^\infty_c((0,1) \times X
\qu G)$.  A generic small perturbation $H'$ of $H$ has the property
that every $(J,H')$-holomorphic quasistrip is regular by Theorem
\ref{tiHreg}.  Let $\phi' \in \on{Diff}^h(X \qu G)$ denote the flow of
$H'$.  For $H'$ sufficiently small, the number of points in
$\phi'(L_\lambda) \cap L_\lambda$ is the same as $\phi(L_\lambda) \cap
L_\lambda$.  By Theorem \ref{hqfhm}, this number is at least
$2^{\dim(T)}$, which proves the proposition.
\end{proof} 

The following characterization of the class of non-compact symplectic
toric manifolds admitting presentations as symplectic quotients of
vectors spaces (so that the main result Theorem \ref{main} applies)
resulted from discussions with E. Lerman and M. Abouzaid, and appears
in \cite[Theorem 1.14]{le:nc}.

\begin{proposition} 
A non-compact symplectic toric orbifold $Y$ with action of a torus $T$
and moment map $\Psi: Y \to \t^\dual$ admits a presentation as a
symplectic quotient of a vector space $X$ by a torus $G$ if and only
if the following three conditions are satisfied: (i) $\Psi$ is proper,
(ii) $\Psi(Y)$ has finitely many facets, and (iii) $\Psi(Y)$ has at
least one vertex.
\end{proposition} 

Indeed, any $Y$ which can be realized as $X \qu G$ satisfies these
conditions: (i) and (ii) follow from the corresponding property the
$N$-torus action on $X \cong \C^N$, while (iii) follows from the fact
that $f: X \to \R, (z_1,\ldots, z_N) \mapsto \sum_{j=1}^N |z_j|^2$ is
proper and descends to a proper function $f \qu G$ on $Y $ and
generates a Hamiltonian circle action. Its minimum is a compact
symplectic toric manifold and as such automatically contains a
$T$-fixed point, which is then a vertex of $\Psi(Y)$.  Conversely, any
symplectic toric manifold satisfying these conditions has the property
that $\Psi(Y)$ is contained in the interior of the moment image of
some Hamiltonian $T$-action on a vector space $Z$, see \cite{le:nc},
and so can be constructed by symplectic cutting.  

Theorem \ref{main} can be generalized to the case of an arbitrary
symplectic toric manifold with proper moment map such that the moment
polytope has finitely many facets (not necessarily a symplectic
quotient) as follows.  Any convex polyhedron $P$ is a prism over a
convex polyhedron $Q$ containing a vertex, that is, isomorphic to $Q
\times R$ for some vector space $R$.  By the classification in
\cite{le:nc} it follows:

\begin{corollary} \label{prodcor} Suppose that $Y$ is a  
symplectic toric orbifold $Y$ with proper moment map $\Psi$ such that
$\Psi(Y)$ has finitely many facets.  Then $Y$ is symplectomorphic to
the product of a quotient $X \qu G$ of a vector space $X$ by a torus
$G$ with a cotangent bundle $T^* (S^1)^r$, for some $r \ge 0$.
\end{corollary}

\begin{proposition}  
Let $Y$ be as in Corollary \ref{prodcor}.  Then the statement of
Theorem \ref{main} holds.
\end{proposition} 

\begin{proof}  
The Floer homology of a toric moment fiber in $Y$ is the tensor
product of Floer homologies of a moment fiber $(S^1)^r$ and a moment
fiber in $X \qu G$.  Consider the quasimap Floer theory for the action
of $G$ on $X \times T^*(S^1)^r$ trivial on the second factor.  The
various moduli spaces used to defined quasimap Floer cohomology are
compact, since any holomorphic maps in $X \times T^* (S^1)^r$ projects
to holomorphic maps in $X$ and $T^* (S^1)^r$, which are convex; one
may then apply the maximum principle on the factors.  The quasimap
Floer cohomology of $L_\lambda \times (S^1)^r$ is the tensor product
of quasimap Floer cohomology of $L_\lambda$ with the Morse homology of
$(S^1)^r$, since there are no non-trivial holomorphic disks in $T^*
(S^1)^r$ with boundary on $(S^1)^r$.  Non-vanishing obstructs
displaceability as before by the same argument involving the quasimap
Floer cohomology of the pair. \end{proof}

\begin{remark} The statement of Theorem \ref{main}, and so the Proposition,
depends on the choice of a realization of $X \qu G$ as a symplectic
quotient, of which there are several; it seems from computations to be
the case that the minimal presentation of $X \qu G$ (that is,
involving only irredundant inequalities) gives the same list of
non-displaceable fibers as the other presentations.\end{remark}

\begin{remark}   
The existence of a non-displaceable torus can be rephrased as follows
for those readers interested mainly in Hamiltonian dynamics: Recall
that a {\em quasifixed point} of a diffeomorphism $\varphi$ of a
symplectic $T$-orbifold $Y$ is an orbit $O \subset Y$ such that
$\varphi(O) \cap O \neq \emptyset$.  Any compact completely integrable
torus action is a projective toric variety, by Delzant's theorem and
its extension to orbifolds \cite{le:ha}.  Theorem \ref{main} then
implies: For any completely integrable Hamiltonian torus action on a
compact symplectic orbifold $Y$ with rational symplectic class and
possibly with orbifold singularities, there exists at least one orbit
$O \subset Y$ that is a quasifixed point for {\em any} Hamiltonian
diffeomorphism.
\end{remark} 

\section{Non-displaceability via bulk deformations}

In this section we briefly describe an extension of quasimap Floer
cohomology to include {\em bulk-deformations} as explored in
\cite{fooo:toric2}.  These groups give additional obstructions to
displaceability.

\subsection{Bulk deformations via twists}

We first give a somewhat naive definition of bulk-deformed \ainfty
algebra for degree two classes by twisting the coefficients of the
composition maps.  Let $X$ be a Hamiltonian $G$-manifold so
that $X \qu G$ is locally free, $L \subset X \qu G$ a Lagrangian
contained in the free locus, and $\ti{L} \subset X$ its inverse image
in $X$.  For $\alpha \in H^2_G(X,\ti{L};\Lambda_0)$, define an {\em
  $\alpha$-twisted} \ainfty algebra $CQF^\alpha(L)$ by 
$$ \mu_n^\alpha(\bra{x_1}, \ldots, \bra{x_n}) = \sum_{[u] \in
  \ol{MW}_{n}(x_0,\ldots,x_n)_0} (-1)^{\heartsuit} \eps(u) e^{\lan \alpha, [u] \ran } q^{A(u)}
\Hol_L(u) \bra{x_0} $$
well-defined via perturbations as above if every holomorphic disk with
boundary in $\ti{L}$ is regular.  The proof of \ainfty associativity
is the same as in Theorem \ref{ainfty}.  We denote by
$$W_\alpha^G : H^1(L,\Lambda_0) \to \Lambda_0, \quad \beta \mapsto \sum_{I(u) = 2} e^{\lan \alpha, [u] \ran } q^{A(u)} \Hol_{L^\beta}(u) $$
the $a$-deformed gauged potential corresponding to the family of brane
structures $L^\beta$ on $L$ determined by $\beta \in H^1(L,\Lambda_0)$.
Similarly given Lagrangians $L_0,L_1 \subset X \qu G$ and a class
$\alpha \in H^2_G(X,\ti{L}_0 \cup \ti{L}_1)$ there is, if every
holomorphic disk with boundary in $\ti{L}_0$ and $\ti{L}_1$ is
regular, an $\alpha$-deformed \ainfty bimodule $CQF(L_0,L_1)$,
isomorphic to $CQF(L)$ in the case $L_0 = L_1 = L$.  In particular, if
the free part of $HQF(L)$ is non-vanishing then $\ti{L}$ is not
displaceable by $H \in C^\infty_c([0,1] \times X)^G$.  This involves
no new analysis, but only a check that the \ainfty associativity
relations are unchanged by the twisting above.

Suppose that $X \qu G$ is a (possibly orbifold) projective toric
variety and $L_\lambda$ is a toric moment fiber.  Then $H_G^2(X,\C) =
\bigoplus_{i=1}^N \C c_1^G(\C_i) $ where $\C_i$ is the $i$-th
component of $\C^N$, and $c_1^G(\C_i)$ denotes the equivariant first
Chern class, and each of these classes vanishes on $\ti{L}_\lambda$.
Suppose that
$ \alpha = \sum_{i=1}^N \alpha_i c_1^G(\C_i) .$
Then for $\beta \in H^1(L_\lambda;\Lambda_0)$ we have
$$ W_{\lambda,\alpha}^G(\beta) = \sum_{i=1}^N \exp( \alpha_i + \lan v_i,
\beta \ran ) q^{l_i(\lambda)} $$
using the fact that the pull-back of $\alpha$ to $L$ vanishes, that
is, $\alpha$ is in the image of $H^2_G(X,\ti{L}_\lambda) \to
H^2_G(X)$, since each $c_1^G(\C_i)$ is a Thom class for the $i$-th
coordinate.

\begin{theorem}\label{bulk} If $W^G_{\alpha,\lambda}$ has a critical point
at $\beta \in H^1(L_\lambda;\Lambda_0)$, then
$HF^{\alpha,\beta}(L_\lambda,\Lambda_0)$ is isomorphic to
$H(L_\lambda,\Lambda_0)$, and $\ti{L}_\lambda$ is not displaceable by
the flow of any invariant time-dependent Hamiltonian.
\end{theorem} 

The proof is the same as in the untwisted case.  We refer to Fukaya et
al \cite{fooo:toric2} for examples and computations.

\subsection{Bulk deformations via insertions} 

We now explain the connection with ``bulk insertions''.  Let
$\ol{M}_{d;e}$ denote the moduli space of stable disks with $d$
markings on the boundary and $e$ markings in the interior.  The
boundary strata are formed when markings on the boundary come together
to form disks, interior markings in the interior go to the boundary to
form disks, or interior markings come together to form spheres.

\begin{figure}[ht]
\includegraphics[height=1in]{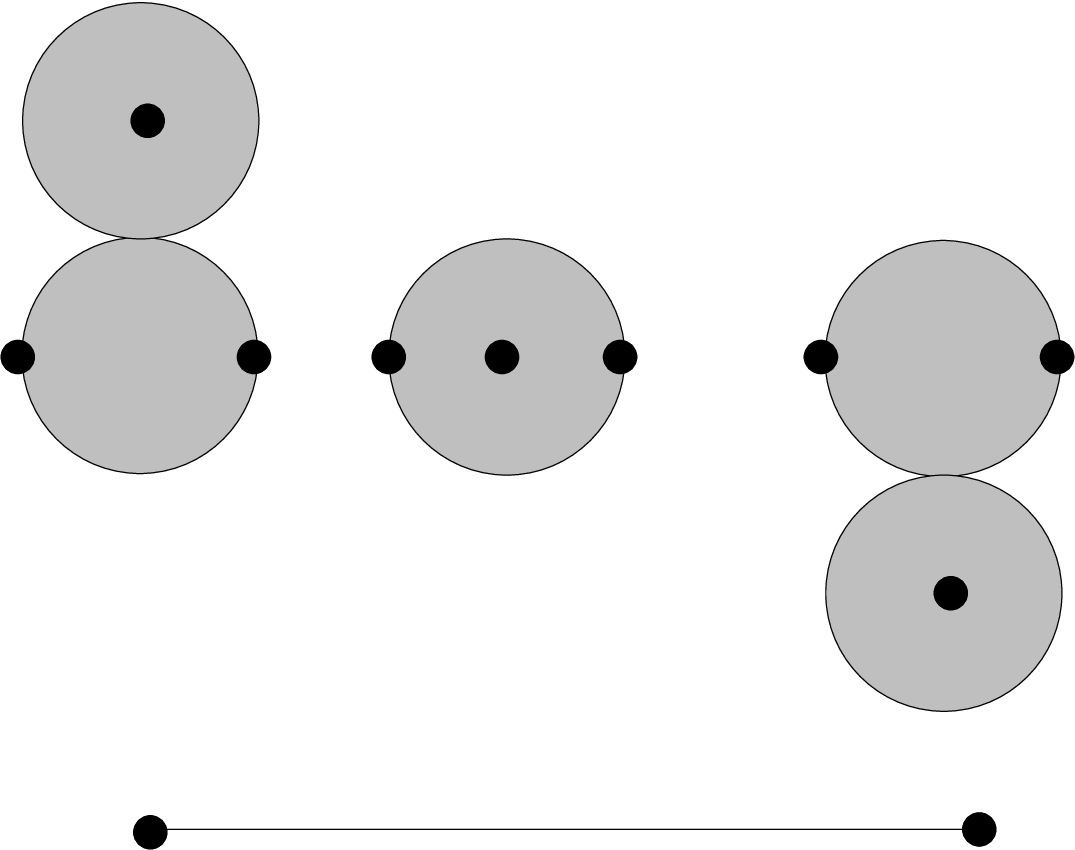}
\caption{Moduli space of stable $2;1$ marked disks}
\end{figure} 

Similarly let $\ol{MW}_{d;e}$ denote the moduli space of stable treed
disk with $d$ markings on the boundary and $e$ markings in the
interior; this means that there are, in addition to disk and sphere
components, dimension one components attached to the boundary.  See
Figure \ref{treeddedisk} for the case $(d;e) = (2,1)$, where the 
edges attached to boundary markings have been omitted to save space. 
\begin{figure}[ht]
\includegraphics[height=1in]{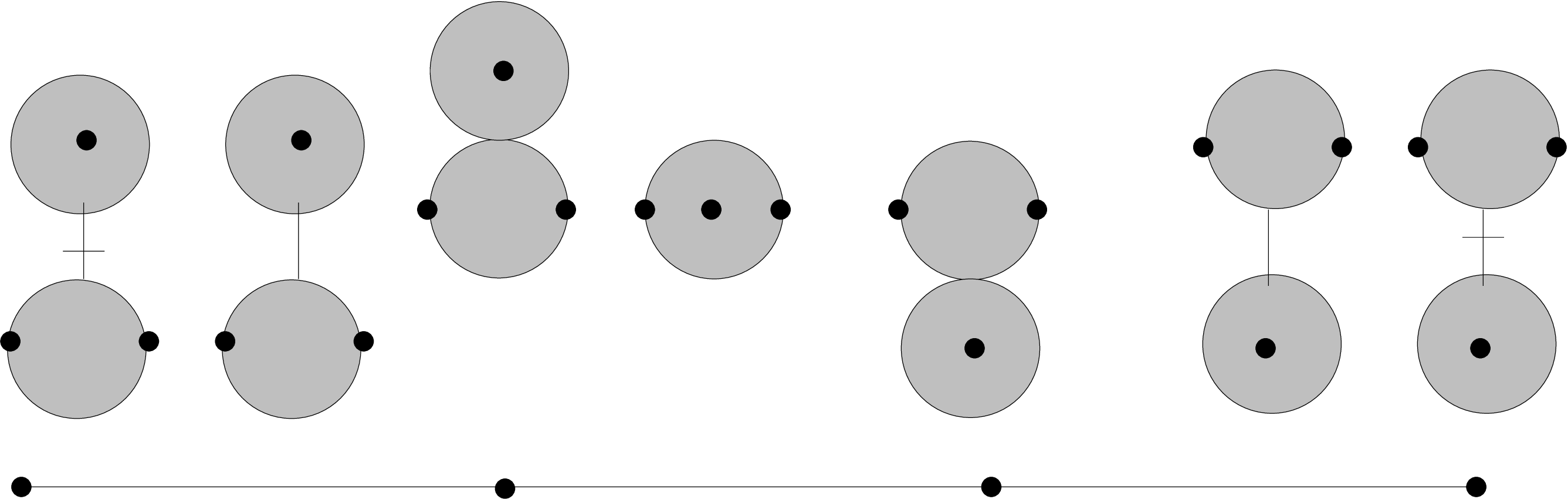}
\caption{Moduli space of treed disks with interior marking}
\label{treeddedisk}
\end{figure} 
The boundary of $\ol{MW}_{d;e}$ is the same as that for
$\ol{MW}_{d;e}$, but there are also codimension two strata
corresponding to sphere bubbles.  Standard arguments imply that
$\ol{MW}_{d;e}$ is compact and local descriptions of the moduli space
may be obtained by combining the local descriptions of $\ol{MW}_d$
with local descriptions of the moduli space of stable marked curves;
these are left to the reader.

Let $X$ and $L \subset X \qu G$ be as above.  A {\em treed holomorphic
  $(d;e)$-marked quasidisk} to $(X,\ti{L})$ consists of a treed
$(d;e)$-marked disk together with a continuous map $u: \Sigma \to X$
which is holomorphic on each disk and sphere component and a gradient
trajectory on each edge.  A treed holomorphic quasidisk is {\em
  stable} if it has no infinitesimal automorphisms, and each node
connecting two segments maps to a critical point.  Let
$\ol{MW}_{d;e}(L)$ denote the moduli space of stable tree marked
quasidisks.  Let $\ol{MW}^{\fr}_{d;e}(X,L)$ denote the moduli space of
{\em framed} quasidisks: this means that we do not mod out by the
$G$-action, so that $\ol{MW}_{d;e}(X,L) = \ol{MW}^{\fr}_{d;e}(X,L)/G$.
Associated to any treed holomorphic $(d;e)$-marked disk is a {\em
  linearized operator} $D_u$ as before.  The map $u$ is {\em regular}
if $D_u$ is surjective.

\begin{theorem}  The moduli space $\ol{MW}^{\reg}_{d;e}(X,L)$ 
resp. $\ol{MW}^{\fr,\reg}_{d;e}(X,L)$ of {\em regular} holomorphic
$(d;e)$-marked treed disks resp. framed $(d;e)$-marked treed disks has
the structure of a topological manifold with boundary given by the
union over strata corresponding to types where a tree has broken off,
resp. principal $G$-bundle over $\ol{MW}^{\reg}_{d;e}(X,L)$.
\end{theorem} 

This theorem requires a slightly more detailed analysis of the gluing
construction, see the review section of \cite{deform} which constructs
$C^1$-compatible charts for the moduli space of holomorphic maps on
surfaces without boundary.  (Since we do not give a proof for the case
with boundary, the reader should feel free to consider this an
axiom for the following discussion.)  The claim on
$\ol{MW}^{\fr,\reg}_{d;e}(X,L)$ follows from the fact that any free
differentiable $G$-action has the structure of a principal $G$-bundle.

The moduli space admits combined evaluation maps to $\crit(F)^{d+1}$,
by evaluation at the boundary markings, and an evaluation map at the
interior markings
$$ \ev^\fr: \ol{MW}^{\fr}_{d;e}(X,L) \to X^e .$$
By the second part of the theorem, there is a classifying map to the
Borel construction $\ol{MW}^{\fr,\reg}_{d;e}(X,L) \to EG .$ Combining
this with the framed evaluation maps we obtain evaluation maps at the
interior markings
$$ \ev: \ol{MW}^{\reg}_{d;e}(X,L) \to X_G^e .$$
Assuming every point in the moduli space is regular, ``integration'' over these
moduli spaces defines operations
\begin{multline}
 \mu_{d;e} : CQF(L)^{\otimes d} \otimes Z_G(X)^{\otimes e} \to CQF(L)
 \\
\bra{x_1} \otimes \ldots \otimes \bra{x_d} \otimes  a_1 \otimes \ldots \otimes a_e \\
\mapsto \sum_{\mdeg \in
  \pi_2(X,\ti{L})} q^{A(\mdeg)} \Hol_L(\mdeg) 
( \ol{MW}_{d;e}(X,L;x_0,\ldots,x_d;\mdeg), 
\ev^*_1 a_1 \cup \ldots \cup \ev^*_e a_e ) 
\bra{x_0} \end{multline}
where $A(\mdeg)$ denotes the symplectic area, depending only on on the
homotopy class $\mdeg$ of $u$, $( \cdot, \cdot)$ denotes the pairing
of the cochain $\ev^* a = \ev_1^* a_1 \cup \ldots \cup \ev^*_e a_e $
with a chain representing the relative fundamental class, obtained by
for example triangulating $\ol{MW}_{d;e}(X,L;x_0,\ldots,x_d;\mdeg)$;
since $\ev^*a $ is closed the pairing is independent of the choice of
triangulation.  One can also define the pairing using equivariant de
Rham theory.  In this setting, the form $\ev^*a$ is obtained by
pull-back of an equivariant form $a \in \Omega_G(X)$ \cite{gu:eqdr} to
$\ol{MW}^\fr_{d;e}(X,L;x_0,\ldots,x_d;\mdeg)$ and then obtaining an
ordinary form on the quotient
$\ol{MW}_{d;e}(X,L;x_0,\ldots,x_d;\mdeg)$ using the Cartan
construction.

The operations satisfy a relation similar to that of \ainfty
associativity, (continuing to assume that every stable marked quasimap
is regular)
\begin{multline}
 0 = \sum_{k,i,J \subset \{1,\ldots,e\}}
 \mu_{d-i;e-|J|}(x_1,\ldots,x_j \\ \mu_{i;|J|}(x_{j+1},\ldots,x_{j+i};
 a_j, j \in J),\ldots,x_n; a_j, j \notin J) .\end{multline}
For any cycle $\alpha \in Z_G(X,\Lambda_+)$ we obtain an \ainfty algebra
with operations denotes $\mu^\alpha_n$ called the {\em bulk-deformed}
\ainfty algebra of $L$, by summing over all possible insertions:
$$ \mu^\alpha_d(x) = \sum_{e \ge 0} \mu_{d;e}(x;\alpha,\ldots,
\alpha)/e! .$$
Any cohomologous choice $\alpha' \in Z_G(X,\Lambda_+)$ gives a homotopic
\ainfty algebra, so that the {\em bulk-deformed Floer cohomology}
$HQF^\alpha(L) = H(\mu^\alpha_1)$ depends only on the class $[\alpha] \in H_G(X)$, up
to isomorphism.  The bulk-deformed invariants satisfy a divisor
relation for $\alpha \in Z^2_G(X)$ representing a class in the image of
$H^2_G(X,\ti{L}) \to H^2_G(X)$
$$ \mu_{\alpha,n}(b;\mdeg) = e^{ \lan [\alpha], \mdeg \ran} \mu_n(b;\mdeg) $$
obtained from the forgetful map $\ol{MW}_{d;e}(L,\mdeg) \to
\ol{MW}_{d;e-1}(L,\mdeg)$.  In particular, this divisor relation
implies that the bulk-deformed \ainfty algebra is equal to that
defined by twisting in the previous section for classes of degree $2$.

One also has a bulk-deformed \ainfty bimodule $CQF^\alpha(L_0,L_1)$
for a pair $L_0,L_1 \subset X \qu G$ of Lagrangians, obtained by
adding bulk insertions to the holomorphic strips, and an isomorphism
of \ainfty bimodules $CQF^\alpha(L_0,L_1) \to CQF^\alpha(L)$ in the
case of an equal pair $L = L_0 = L_1$.  In particular, if the free
part of $HF^{\alpha}(L,\Lambda_0)$ is non-trivial then $\ti{L}$ is not
displaceable by any flow generated by a $G$-invariant family $H \in
C^\infty_c([0,1] \times X)^G$.  However, except for $\alpha$ in
$H^2_G(X,\ti{L})$ these seem difficult to compute.

\section{Relationship to Floer cohomology in the quotient}

In this final section we outline a formal argument which relates the
potential of a symplectic quotient to the gauged potential by a
combined bulk-boundary quantum Kirwan map.  Let $\Sigma$ be the unit
disk equipped with standard area form $\Vol_\Sigma$.  For any $\eps
\in [0,\infty)$ let $\ol{M}_{d;e}(x_0,\ldots,x_d)_\eps$ denote the
  moduli space of solutions to the vortex equations on $\Sigma$ with
  area form $\eps \Vol_\Sigma$ with $d$ markings on the boundary, $e$
  markings in the interior, and boundary in $\ti{L}$, compactified as
  for the previous moduli spaces, with boundary markings mapping to
  $x_0,\ldots,x_d \in \crit(F)$, modulo holomorphic automorphisms of
  $\Sigma$ which preserve $\Vol_\Sigma$.  Let $ \ev:
  \ol{M}_{d;e}(x_1,\ldots,x_d)^\eps \to X_G^e $ denote the evaluation
  maps at the interior markings (obtained by combining with a
  classifying map as above).  Integration over
  $\ol{M}_{d;e}(x_1,\ldots,x_n)_\eps$ should define invariants
$$ \tau_{d;e}^\eps: CQF(L_\lambda,\Lambda_0)^{\otimes d} \otimes
  Z_G(X,\Lambda_0)^{\otimes e} \to \Lambda_0 $$
via the formula
$$ \tau_{d;e}^\eps(x_1,\ldots,x_d,\alpha) = \sum_{\mdeg} q^{A(\mdeg)}
\Hol_{L}(\mdeg) ([\ol{M}_{d;e}(x_1,\ldots,x_d;\mdeg)_\eps ], \ev^*
\alpha) .$$
If $\eps = 0$ then these moduli spaces are quasimap moduli spaces, but
with parametrized domain.  Define a {\em vortex potential}
$$ W_\lambda^\eps(\beta,\alpha) = \sum_{e \ge 1} (1/e!)  \tau_{d;e}
(\alpha,\ldots, \alpha) $$
where $\beta \in H^1(L,\Lambda_0)$ corresponds to the local system and
$\alpha \in Z_G(X,\Lambda_0)$.  Since the vortex moduli spaces are
always reducible free one expects these potentials to be independent
of $\eps$.  We study the large area limit $\eps \to \infty$.  We
consider, as in Gaio-Salamon \cite{ga:gw}, a sequence $(A_\nu,u_\nu)$
of vortices on $\Sigma = \R \times [0,1]$ with boundary in $\ti{L}$
with area form $\eps_\nu \Vol_{\Sigma}$ with $\eps_\nu \to \infty$.
Suppose that 
$$c_\nu := |\d_{A_\nu} u_\nu (z_\nu)| := \sup |\d_{A_\nu} u_\nu|  .$$
We denote by $\dist(z, \partial \Sigma)$ the distance of a point $z
\in \Sigma$ to the boundary $\partial \Sigma$.  Six types of bubbling
are possible (we do not exclude the possibility of sphere bubbling in
$X$ in the following discussion:)

\begin{enumerate} 
\item 
If $c_\nu/\eps_\nu \to \infty$ and $c_\nu \dist(z_\nu, \partial
\Sigma) \to \infty$ as $\nu \to \infty$ then after re-scaling and
passing to a subsequence one obtains a sphere bubble in $X$;
\item 
If $c_\nu/\eps_\nu \to \infty$ and $c_\nu \dist(z_\nu, \partial
\Sigma)$ has a finite limit as $\nu \to \infty$ then after re-scaling
and passing to a subsequence one obtains a disk bubble in
$(X,\ti{L})$;
\item If $c_\nu/\eps_\nu $ has a finite limit and $c_\nu \dist(z_\nu,
  \partial \Sigma) \to \infty$ as $\nu \to \infty$ then after
  re-scaling and passing to a subsequence one obtains a vortex on $\C$
  with values in $X$;
\item 
If $c_\nu/\eps_\nu$ has a finite limit and $c_\nu \dist(z_\nu,
\partial \Sigma)$ has a finite limit as $\nu \to \infty$ then after
re-scaling and passing to a subsequence one obtains a vortex on $\H =
\{ \on{Im}(z) \ge 0\}$ with boundary in $\ti{L}$;
\item 
If $c_\nu/\eps_\nu \to 0$ and $c_\nu \dist(z_\nu, \partial \Sigma)
\to \infty$ as $\nu \to \infty$ then after re-scaling and passing
to a subsequence one obtains a sphere bubble in $X \qu G$;
\item 
If $c_\nu/\eps_\nu \to 0$ and $c_\nu \dist(z_\nu, \partial \Sigma)$
has a finite limit as $\nu \to \infty$ then after re-scaling and
passing to a subsequence one obtains a disk bubble in $X \qu G$ with
values in $L$.
\end{enumerate} 

One expects the potential $W^\infty_\lambda(\beta,\alpha)$ counting
quasimap invariants to be related to the parametrized potential on the
quotient after incorporating vortex bubbles.  Let $MW_{d;e}(\H,X,L)$
denote the moduli space of treed vortices on $(\H,\R)$ with values in
$(X,\ti{L})$ with $d$ markings on the boundary and $e$ markings in the
interior, and $\ol{MW}_{d,e}(\H,X,L)$ its compactification allowing
disk and sphere bubbles.  The codimension one boundary strata consist
of configurations where a holomorphic disk with values in $\ti{L}$ has
bubbled off
$$ (\ol{MW}_{d_0;e_0}^{\fr}(\H,X,L) \times_{\ti{L}}
\ol{MW}_{d_1;e_1}(X,\ti{L}))/G^n $$
and configurations
$$ \ol{MW}^\fr_{r;e_0}(X \qu G, L) \times_{L^r} \prod_{j=1}^r
\ol{MW}_{|I_j|;e_j}(\H,X,L) $$
where $r$ vortices on $\H$ have bubbled off, leaving as the ``main
component'' a holomorphic disk in $X \qu G$ with values in $L$.  Let
$CQF(L)$ denote the quasimap Fukaya algebra defined above and suppose
that the Fukaya algebra $CF(L)$ of $L \subset X \qu G$ using
holomorphic disks in $X \qu G$ has been rigorously defined.  If we
assume that all moduli spaces are regular and compact, then one can
define an {\em open quantum Kirwan morphism}
$$Q\kappa_{X,L}: CQF(L)^{\otimes d} \to CF(L)$$
$$ (\bra{x_1},\ldots,\bra{x_d}) \mapsto \sum_{[u] \in
  \ol{M}_d(\H,X,L;x_0,\ldots,x_d)} (-1)^{\heartsuit} \eps(u) q^{A(u)}
\Hol_L(u) \bra{x_0}.
$$
Adding bulk insertions produces maps 
$$Q\kappa_{X,L}: CQF(L)^{\otimes d} \otimes Z_G(X) \to
CF(L)$$
\begin{multline} 
(\bra{x_1},\ldots,\bra{x_d};\alpha) \mapsto \sum_{\mdeg,e} (1/e!)
(-1)^{\heartsuit} q^{A(\mdeg)} \Hol_L(\mdeg) \\ (
\ol{M}_{d;e}(\H,X,L;x_0,\ldots,x_d), \ev_1^* \alpha \cup \ldots \cup
\ev_e^* \alpha) \bra{x_0}.
\end{multline}
Formally, $Q\kappa_{X,L}$ without bulk insertions satisfies the axioms
of an \ainfty-morphism, because the boundary components of the
one-dimensional strata correspond to facets of the multiplihedra as in
Ma'u-Woodward \cite{mau:mult}.  Similarly we have a version for
vortices on $\C$, which was already discussed in for example
Woodward-Ziltener \cite{manifesto} and formally produces a map 
$$ Q \kappa_X:  Z_G(X) \otimes \Lambda \to Z(X \qu G) \otimes \Lambda .$$
Since the bubbling that appears in the limit have been incorporated
into the bulk and boundary quantum Kirwan morphisms, one expects

\begin{conjecture} With $X, X \qu G $ and $L \subset X$ any $G$-Lagrangian
brane,
$$ W^\infty_\lambda(Q\kappa_{L,X}(\beta,\alpha), Q\kappa_X(\alpha)) =
W^0_\lambda(\beta,\alpha) .$$
\end{conjecture}

This conjecture is closely related to a question of Auroux (related to
the discussion in \cite{auroux:wall}), who asked whether $W_\lambda$
is related to $W^G_\lambda$ by a change of coordinates.  In particular
we conjecture that the ``bulk'' part of the necessary change of
coordinates is equivalent to that appearing in Givental's work as the
``mirror map''.  There are at least two explicit examples which
provide evidence for this conjecture, besides the conceptual framework
provided by the large area limit explained above: first, Auroux's
computation of the potential for the second Hirzebruch surface
\cite[Proposition 3.1]{auroux:wall} shows that the corrected potential
is related to the naive potential by a transformation equivalent to
Givental's mirror map.  Secondly, the computation in \cite{pand:disk}
relates a disk potential defined without markings on the boundary
(thus, not the potential that appears in \ainfty setting) to a naive
potential via a coordinate change that is the same as in the closed
case.  However, it is not clear in general how to fit the invariants
of \cite{pand:disk} into the above framework, in general; the relation
between the various disk potentials seems an interesting topic for
further investigation.

\def\cprime{$'$} \def\cprime{$'$} \def\cprime{$'$} \def\cprime{$'$}
  \def\cprime{$'$} \def\cprime{$'$}
  \def\polhk#1{\setbox0=\hbox{#1}{\ooalign{\hidewidth
  \lower1.5ex\hbox{`}\hidewidth\crcr\unhbox0}}} \def\cprime{$'$}
  \def\cprime{$'$}


\begin{thebibliography}{10}

\bibitem{ab:ex} M.~Abouzaid.  \newblock Framed bordism and
  {L}agrangian embeddings of exotic spheres.  \newblock {\em Ann. of
    Math.} 175: 71--185, 2012.

\bibitem{ab:top} M.~Abouzaid.  \newblock A topological model for the
  {F}ukaya categories of plumbings.  \newblock {\em Jour. Diff. Geom.}
  87:1--80, 2011.

\bibitem{alb:lag} P.~Albers.  \newblock A {L}agrangian
  {P}iunikhin-{S}alamon-{S}chwarz morphism and two comparison
  homomorphisms in {F}loer homology.  \newblock {\em
    Int. Math. Res. Not. IMRN}, (4), 2008.

\bibitem{alston:cliff} G.~Alston.  \newblock Lagrangian {F}loer
  homology of the {C}lifford torus and real projective space in odd
  dimensions. \newblock {\em J. Symplectic Geom.} (9) 2011) 83--106.

\bibitem{al:tf} G.~Alston and L.~Amorim.  \newblock Floer cohomology
  of torus fibers and real {L}agrangians in {F}ano toric manifolds.
  \newblock {\em Int. Math. Res. Not. IMRN} (12) 2012 2751–-2793.
 

\bibitem{auroux:wall}
D.~Auroux.
\newblock Special {L}agrangian fibrations, wall-crossing, and mirror symmetry.
\newblock In {\em Surveys in differential geometry. {V}ol. {XIII}. {G}eometry,
  analysis, and algebraic geometry: forty years of the {J}ournal of
  {D}ifferential {G}eometry}, volume~13 of {\em Surv. Differ. Geom.}, pages
  1--47. Int. Press, Somerville, MA, 2009.

\bibitem{auroux:wpp}
D.~Auroux, L.~Katzarkov, and D.~Orlov.
\newblock Mirror symmetry for weighted projective planes and their
  noncommutative deformations.
\newblock {\em Ann. of Math. (2)}, 167(3):867--943, 2008.

\bibitem{bc:ql}
P.~Biran and O.~Cornea.
\newblock Quantum structures for {L}agrangian submanifolds.
\newblock arXiv:0708.4221.

\bibitem{boardman:hom}
J.~M. Boardman and R.~M. Vogt.
\newblock {\em Homotopy invariant algebraic structures on topological spaces}.
\newblock Lecture Notes in Mathematics, Vol. 347. Springer-Verlag, Berlin,
  1973.

\bibitem{buh:mul}
L.~Buhovsky.
\newblock The {M}aslov class of {L}agrangian tori and quantum products
in {F}loer cohomology. 
\newblock arXiv:math/0608063.

\bibitem{cho:products}
C.-H.~Cho.
\newblock Products of {F}loer cohomology of torus fibers in toric {F}ano
  manifolds.
\newblock {\em Comm. Math. Phys.}, 260(3):613--640, 2005.

\bibitem{chooh:toric}
C.-H.~Cho and Y.-G.~Oh.
\newblock Floer cohomology and disc instantons of {L}agrangian torus fibers in
  {F}ano toric manifolds.
\newblock {\em Asian J. Math.}, 10(4):773--814, 2006.

\bibitem{ci:symvortex}
K.~Cieliebak, A.~Rita Gaio, I.~Mundet~i Riera, and D.~A. Salamon.
\newblock The symplectic vortex equations and invariants of {H}amiltonian group
  actions.
\newblock {\em J. Symplectic Geom.}, 1(3):543--645, 2002.

\bibitem{entov:rigid}
M.~Entov and L.~Polterovich.
\newblock Rigid subsets of symplectic manifolds.
\newblock {\em Compos. Math.}, 145(3):773--826, 2009.

\bibitem{floer:lag}
A.~Floer.
\newblock Morse theory for {L}agrangian intersections.
\newblock {\em J. Differential Geom.}, 28(3):513--547, 1988.

\bibitem{fhs:tr}
A.~{F}loer, Helmut Hofer, and D.~Salamon.
\newblock Transversality in elliptic {M}orse theory for the symplectic action.
\newblock {\em Duke Math. J.}, 80(1):251--292, 1995.

\bibitem{frau:thesis}
U.~Frauenfelder.
\newblock {\em Floer homology of symplectic quotients and the Arnold-Givental
  conjecture}.
\newblock PhD thesis, ETH Zurich, 2003.

\bibitem{frauen:mfh}
U.~Frauenfelder.
\newblock The {A}rnold-{G}ivental conjecture and moment {F}loer homology.
\newblock {\em Int. Math. Res. Not.}, (42):2179--2269, 2004.

\bibitem{fooo:toric2}
K.~Fukaya, Y.-G.~Oh, H.~Ohta, and K.~Ono.
\newblock {L}agrangian {F}loer theory on compact toric manifolds {I}{I} : Bulk
  deformations.
\newblock 
{\em Selecta Math. (N.S.)} 17: 609--711, 2011.

\bibitem{fooo}
K.~Fukaya, Y.-G.~Oh, H.~Ohta, and K.~Ono.
\newblock {\em Lagrangian intersection {F}loer theory: anomaly and
  obstruction.}, volume~46 of {\em AMS/IP Studies in Advanced Mathematics}.
\newblock American Mathematical Society, Providence, RI, 2009.

\bibitem{fooo:toric1}
K.~Fukaya, Y.-G.~Oh, H.~Ohta, and K.~Ono.
\newblock Lagrangian {F}loer theory on compact toric manifolds. {I}.
\newblock {\em Duke Math. J.}, 151(1):23--174, 2010.

\bibitem{ga:gw}
A.~Rita~Pires Gaio and D.~A. Salamon.
\newblock Gromov-{W}itten invariants of symplectic quotients and adiabatic
  limits.
\newblock {\em J. Symplectic Geom.}, 3(1):55--159, 2005.

\bibitem{ganatra:duality}
S.~Ganatra.
\newblock {\em Symplectic cohomology and duality for the wrapped Fukaya category}.
\newblock arXiv:1304.7312. 

\bibitem{gi:eq}
A.~B. Givental.
\newblock Equivariant {G}romov-{W}itten invariants.
\newblock {\em Internat. Math. Res. Notices}, (13):613--663, 1996.

\bibitem{cross}
E.~Gonzalez and C.~Woodward.
\newblock Quantum {W}itten localization and abelianization for qde solutions.
\newblock arXiv:0811.3358.

\bibitem{small}
E.~Gonzalez and C.~Woodward.
\newblock Gauged {G}romov-{W}itten theory for small spheres.
\newblock
{\em Math. Z.} 273:485--514, 2013.

\bibitem{deform}
E.~Gonzalez and C.~Woodward.
\newblock Deformations of symplectic vortices.
\newblock {\em Ann. Global Anal. Geom.}  39:45–82, 2011. 

\bibitem{gu:eqdr}
V.~W. Guillemin and S.~Sternberg.
\newblock {\em Supersymmetry and equivariant de {R}ham theory}.
\newblock Springer-Verlag, Berlin, 1999.
\newblock With an appendix containing two reprints by Henri Cartan [MR {\bf
  13},107e; MR {\bf 13},107f].

\bibitem{ho:mi}
K.~Hori and C.~Vafa.
\newblock Mirror symmetry.
\newblock hep-th/0002222.

\bibitem{ho:an3}
L.~H{\"o}rmander.
\newblock {\em The analysis of linear partial differential operators.
  {I}{I}{I}}.
\newblock Springer-Verlag, Berlin, 1994.
\newblock Pseudo-differential operators, Corrected reprint of the 1985
  original.

\bibitem{ka:coh}
J.~Kati{\'c} and D.~Milinkovi{\'c}.
\newblock Coherent orientation of mixed moduli spaces in {M}orse-{F}loer
  theory.
\newblock {\em Bull. Braz. Math. Soc. (N.S.)}, 40(2):253--300, 2009.

\bibitem{ki:coh}
F.~C. Kirwan.
\newblock {\em Cohomology of Quotients in Symplectic and Algebraic Geometry},
  volume~31 of {\em Mathematical Notes}.
\newblock Princeton Univ. Press, Princeton, 1984.

\bibitem{ks:tf}
M.~Kontsevich and Y.~Soibelman.
\newblock Homological mirror symmetry and torus fibrations.
\newblock In {\em Symplectic geometry and mirror symmetry ({S}eoul, 2000)},
  pages 203--263. World Sci. Publ., River Edge, NJ, 2001.

\bibitem{le:nc}
E.~Lerman and Y.~Karshon.
\newblock Non-compact toric manifolds.
\newblock arXiv:0907.2891.

\bibitem{le:ha}
E.~Lerman and S.~Tolman.
\newblock Hamiltonian torus actions on symplectic orbifolds and toric
  varieties.
\newblock {\em Trans. Amer. Math. Soc.}, 349(10):4201--4230, 1997.

\bibitem{loc:ell}
R.~B. Lockhart and R.~C. McOwen.
\newblock Elliptic differential operators on noncompact manifolds.
\newblock {\em Ann. Scuola Norm. Sup. Pisa Cl. Sci. (4)}, 12(3):409--447, 1985.

\bibitem{ho:phb}
K.~Hori M.~Herbst and D.~Page.
\newblock Phases of {N}=2 theories in $1 + 1$ dimensions with boundary.
\newblock arXiv:0803.2045v1.

\bibitem{ainfty}
S.~Mau, K.~Wehrheim, and C.T. Woodward.
\newblock ${A}_\infty$-functors for {L}agrangian correspondences.
\newblock in preparation.

\bibitem{mau:mult}
S.~Ma'u and C.~Woodward.
\newblock Geometric realizations of the multiplihedron and its
  complexification.
\newblock {\em Compos. Math.}, 146:1002-1028, 2010.

\bibitem{mau:gluing}
S.~Ma'u.
\newblock Gluing pseudoholomorphic quilted disks.
\newblock arxiv:0909.339.

\bibitem{mc:di}
D.~McDuff.
\newblock Displacing {L}agrangian toric fibers via probes.
\newblock 
Proc. Sympos. Pure Math., 82, Amer. Math. Soc., Providence, RI, 2011. 

\bibitem{ms:jh}
D.~McDuff and D.~Salamon.
\newblock {\em {$J$}-holomorphic curves and symplectic topology}, volume~52 of
  {\em American Mathematical Society Colloquium Publications}.
\newblock American Mathematical Society, Providence, RI, 2004.

\bibitem{oh:fl1}
Y.-G. Oh.
\newblock {F}loer cohomology of {L}agrangian intersections and
  pseudo-holomorphic disks. {I}.
\newblock {\em Comm. Pure Appl. Math.}, 46(7):949--993, 1993.

\bibitem{oh:fc}
Y.-G. Oh.
\newblock {F}loer cohomology, spectral sequences, and the {M}aslov class of
  {L}agrangian embeddings.
\newblock {\em Internat. Math. Res. Notices}, (7):305--346, 1996.

\bibitem{pand:disk}
R.~Pandharipande, J.~Solomon, and J.~Walcher.
\newblock Disk enumeration on the quintic 3-fold.
\newblock {\em J. Amer. Math. Soc.}, 21(4):1169--1209, 2008.

\bibitem{pss}
S.~Piunikhin, D.~Salamon, and M.~Schwarz.
\newblock Symplectic {F}loer-{D}onaldson theory and quantum cohomology.
\newblock In {\em Contact and symplectic geometry (Cambridge, 1994)}, volume~8
  of {\em Publ. Newton Inst.}, pages 171--200. Cambridge Univ. Press,
  Cambridge, 1996.

\bibitem{po:cl}
M.~Po{\'z}niak.
\newblock Floer homology, {N}ovikov rings and clean intersections.
\newblock In {\em Northern {C}alifornia {S}ymplectic {G}eometry {S}eminar},
  volume 196 of {\em Amer. Math. Soc. Transl. Ser. 2}, pages 119--181. Amer.
  Math. Soc., Providence, RI, 1999.

\bibitem{rs:asym}
J.~W. Robbin and D.~A. Salamon.
\newblock Asymptotic behaviour of holomorphic strips.
\newblock {\em Ann. Inst. H. Poincar\'e Anal. Non Lin\'eaire}, 18(5):573--612,
  2001.

\bibitem{royden}
H.~L. Royden.
\newblock {\em Real analysis}.
\newblock The Macmillan Co., New York, 1963.

\bibitem{sch:morse}
M.~Schwarz.
\newblock {\em Morse homology}, vol. 111 of {\em Progress in Math}.
\newblock Birkh\"auser Verlag, Basel, 1993.

\bibitem{seidel:def}
P.~Seidel.
\newblock Fukaya categories and deformations.
\newblock In {\em Proceedings of the International Congress of Mathematicians,
  Vol. II (Beijing, 2002)}, pages 351--360, Beijing, 2002. Higher Ed. Press.

\bibitem{se:bo}
P.~Seidel.
\newblock {\em Fukaya categories and {P}icard-{L}efschetz theory}.
\newblock Zurich Lectures in Advanced Mathematics. European Mathematical
  Society (EMS), Z\"urich, 2008.

\bibitem{seidel:genustwo}
P.~Seidel.
\newblock Homological mirror symmetry for the genus two curve.
\newblock {\em J. Algebraic Geom.} 20:727--769, 2011.

\bibitem{seidel:susp}
P.~Seidel.
\newblock Suspending {L}efschetz fibrations, with an application to mirror
  symmetry.
\newblock {\em Comm. Math. Phys.}  297:515--528, 2010.

\bibitem{seidel:sub}
P.~Seidel.
\newblock {$A_\infty$}-subalgebras and natural transformations.
\newblock {\em Homology, Homotopy Appl.}, 10(2):83--114, 2008.

\bibitem{sh:hmsfano}
N.~Sheridan.
\newblock Homological mirror symmetry for {F}ano hypersurfaces.
\newblock 2014 preprint. 

\bibitem{st:ho}
J.~Stasheff.
\newblock {\em {$H$}-spaces from a homotopy point of view}.
\newblock Lecture Notes in Mathematics, Vol. 161. Springer-Verlag, Berlin,
  1970.

\bibitem{jw:vi}
J.~Wehrheim.
\newblock Vortex invariants and toric manifolds.
\newblock arXiv:0812.0299.

\bibitem{we:co} K.~Wehrheim and C.~Woodward.  \newblock Functoriality
  for {L}agrangian correspondences in {F}loer theory.  \newblock {\em
    Quantum Topology} 1:129--170, 2010.

\bibitem{ww:quilts}
K.~Wehrheim and C.~Woodward.
\newblock Pseudoholomorphic quilts.
\newblock arXiv:0905.1369.

\bibitem{orient}
K.~Wehrheim and C.T. Woodward.
\newblock Orientations for pseudoholomorphic quilts.
\newblock In preparation. 

\bibitem{wi:ph}
E.~Witten.
\newblock Phases of {$N=2$} theories in two dimensions.
\newblock {\em Nuclear Phys. B}, 403(1-2):159--222, 1993.

\bibitem{manifesto}
C.~Woodward and F.~Ziltener.
\newblock Functoriality for {G}romov-{W}itten invariants under symplectic
  quotients.
\newblock 2008 preprint.

\bibitem{xu:gf}
G.~{Xu}.
\newblock {Gauged Floer homology for Hamiltonian isotopies I: definition of the
  Floer homology groups}.
\newblock 1312.6923.

\bibitem{zil:decay}
F.~Ziltener.
\newblock The invariant symplectic action and decay for vortices.
\newblock {\em J. Symplectic Geom.}, 7(3):357--376, 2009.

\end{thebibliography}
\end{document}